\documentclass{amsart}
\usepackage{graphicx,amsmath,amssymb,amsxtra,tikz,tikz-cd,anysize,hyperref,comment,pgfplots} % Required for inserting images
\usepackage{enumerate}
\usepackage{adjustbox}

\newcommand{\UU}{\mathbf{U}}
\newcommand{\F}{\mathbb{F}}
\newcommand{\w}{\mathbf{w}}
\newcommand{\z}{\mathbf{z}}
\newcommand{\Z}{\mathbb{Z}}

\newcommand{\Ker}{\mathrm{Ker}}
\newcommand{\Coker}{\mathrm{Coker}}
\newcommand{\Imm}{\mathrm{Im}}

\newcommand{\cB}{\mathcal{B}}
\newcommand{\cD}{\mathcal{D}}
\newcommand{\cZ}{\mathcal{Z}}

\newcommand{\gr}{\mathrm{gr}}

\newcommand{\HFL}{\mathcal{HFL}}
\newcommand{\CFL}{\mathcal{CFL}}
\newcommand{\Hom}{\mathrm{Hom}}

\newcommand{\Br}{\mathrm{Br}}
\newcommand{\Am}{\mathfrak{A}^{-}}

\newcommand{\HFLhat}{\widehat{\mathrm{HFL}}}
\newcommand{\HFLminus}{\mathrm{HFL}^{-}}

\newcommand{\CFLminus}{\mathrm{CFL}^{-}}
\newcommand{\CFLhat}{\widehat{\mathrm{CFL}}}

\newcommand{\HH}{\mathbb{H}}
\newcommand{\lk}{\mathrm{lk}}

\newcommand{\Kh}{\mathrm{Kh}}
\newcommand{\FL}{\mathcal{FL}}

\newcommand{\jj}{\mathbf{j}}

\DeclareMathOperator{\Coinv}{Coinv}
\DeclareMathOperator{\Inv}{Inv}
\newcommand{\lvl}{\mathrm{level}}

\newcommand{\cHFL}{\widehat{\mathrm{cHFL}}}

\newcommand{\alg}{\mathrm{alg}}
\newcommand{\diag}{\mathrm{diag}}

\newcommand{\scU}{U}
\newcommand{\scV}{V}

\newcommand{\norm}{\mathrm{norm}}

\numberwithin{equation}{section}

\newtheorem{theorem}{Theorem}[section]
\newtheorem{proposition}[theorem]{Proposition}
\newtheorem{corollary}[theorem]{Corollary}
\newtheorem{lemma}[theorem]{Lemma}

\theoremstyle{definition}
\newtheorem{remark}[theorem]{Remark}

\newtheorem{convention}[theorem]{Convention}
\newtheorem{notation}[theorem]{Notation}

\newtheorem{definition}[theorem]{Definition}
\newtheorem{problem}[theorem]{Problem}
\newtheorem{example}[theorem]{Example}
\newtheorem{conjecture}[theorem]{Conjecture}
\newtheorem{assumption}[theorem]{Assumption}

\renewcommand{\AA}[1]{{\color{blue}#1}}

\newcommand{\BL}[1]{{\color{purple}#1--Beibei}}

\title{Floer lasagna modules for some 2-handlebodies}

\author{Akram Alishahi}
\thanks{\texttt{AA was partly supported by NSF Grant DMS-2238103}}
\address{Department of Mathematics, University of Georgia, Athens, GA 30602 USA}
\email{\href{mailto:akram.alishahi@uga.edu}{akram.alishahi@uga.edu}}

\author{Eugene Gorsky}
\thanks{\texttt{EG was partly supported by NSF Grant DMS-2302305}}
\address{Department of Mathematics, University of California Davis\\ One Shields Avenue, Davis CA 95616 USA}
\email{\href{mailto:egorskiy@ucdavis.edu}{egorskiy@ucdavis.edu}}

\author{Beibei Liu}
\thanks{\texttt{BL was partly supported by NSF Grant DMS-2417229}}
\address{Department of Mathematics, The Ohio State University, 100 Math Tower, 231 W 18th Ave, Columbus, OH 43201 USA}
\email{\href{mailto:liu.11302@osu.edu}{liu.11302@osu.edu}}

\begin{document}

\begin{abstract}
Morrison, Walker and Wedrich used Khovanov homology to define the skein lasagna module for a 4-manifold with a link in its boundary. The skein lasagna module was shown to be a deep and powerful invariant; in particular, Ren and Willis proved that it detects exotic smooth structures. Still, computing the skein lasagna module outside some carefully chosen grading range remains an important open problem.

Motivated by this construction, Chen defined the Floer lasagna module using Heegaard Floer link homology. In this paper, we give a complete, but partially conjectural description of  Floer lasagna modules for a class of 4-manifolds obtained by adding 2-handles along plumbed L-space knots (with certain framing constraints). 

In particular, we prove that for $\overline{\mathbb{CP}^2}$ it is infinite-dimensional, and conjecture that for $\mathbb{CP}^2$ it is infinite-dimensional as well. We compute this module for $\overline{\mathbb{CP}^2}$ (and conjecturally for $\mathbb{CP}^2$) in all Alexander and Maslov gradings.
This strengthens and extends the recent results of McCulloch.
\end{abstract}

\maketitle

\section{Introduction}

\subsection{Skein lasagna modules}

Given a smooth 4-manifold $W$ with a link $L\subset \partial W$, Morrison, Walker and Wedrich  \cite{MWW} defined  the \emph{skein lasagna module} $S_2(W,L)$ using Khovanov homology. For $W=B^4$ and $L\subset \partial W=S^3$ this recovers the usual Khovanov homology of $L$. In general,  $S_2(W,L)$ is a vector space with homological and quantum gradings, as well as an additional grading by the second relative homology $H_2(W,L)$. By computing $S_2(W,L)$ in certain gradings, Ren and Willis proved in \cite{RW} that the skein lasagna module detects exotic smooth structures on $W$.

In this paper we will be interested in the case when $W=W_m(K)$ is obtained from the 4-ball by adding a 2-handle along some framed, oriented knot $K\subset S^3$, and $L$ is empty. In this case, Manolescu and Neithalath \cite{MN} showed that the general definition can be greatly simplified  as follows.  

Consider the cable of $K$ with $n$ parallel components (for some chosen framing $m$) such that $k_{+}$ components are oriented  positively and $k_{-}$ are oriented negatively. We denote such link by $K_m(k_{+},k_{-})$. Braids on $k_{+}$ and $k_{-}$ strands define cobordisms from $K_m(k_{+},k_{-})$ to itself. Furthermore, there is a natural ribbon cobordism from $K_m(k_{+},k_{-})$ to $K_m(k_{+}+1,k_{-}+1)$ which adds a pair of oppositely oriented strands. All these cobordisms define maps in Khovanov homology. The main result of \cite{MN} identifies the skein lasagna module of $W$ with the quotient
\begin{equation}
\label{eq: cabled skein lasagna intro}
S_2(W_m(K),\emptyset)=\bigoplus_{k_{+},k_{-}}\Kh\left(K_m(k_{+},k_{-})\right)/\sim
\end{equation}
where the equivalence relation $\sim$ is defined by the above cobordism maps. The homological and quantum gradings are the same as in Khovanov homology, up to some normalization, while the additional grading by $H_2(W)=\Z$ is given by $\alpha=k_{+}-k_{-}$.
To use \eqref{eq: cabled skein lasagna intro}, one needs to compute Khovanov homologies for infinitely many links with an arbitrarily large number of components, and all of the above cobordism maps between them. Even for $K$ being the $(-1)$-framed unknot, corresponding to $W=\overline{\mathbb{CP}^2}$, this remains an open problem for most gradings.

In \cite{Chen} Chen used Heegaard Floer link homology to define the \emph{Floer lasagna module} $\FL(W,L)$. The construction follows the ideas in \cite{MWW,MN} with several modifications. Most importantly, the cobordism maps in Khovanov homology are replaced by the more complicated \emph{decorated cobordisms} developed by Zemke \cite{Zemke}. Another crucial difference is that the link cobordism maps of Zemke are only defined over $\F=\Z/2\Z$ which complicates many computations with braid groups and symmetric groups. Nevertheless, Chen was able to prove the following analogue of \eqref{eq: cabled skein lasagna intro}:
\begin{equation}
\label{eq: cabled floer lasagna intro}
\FL(W_m(K),\emptyset)=\bigoplus_{k_{+},k_{-}}\HFLhat\left(K_m(k_{+},k_{-})\right)/\sim
\end{equation}
where the precise equivalence relation $\sim$ is defined in Section \ref{sec: lasagna}. %In this paper we compute \eqref{eq: cabled floer lasagna intro} when $K$ is a plumbed L-space knot with framing 
%$$
%m>2g(K)-1,
%$$
%or its mirror.

\subsection{$0$-framed unknot}

To describe the main results, we start with the case of the unknot $O$. Let $W_m=W_m(O)$ denote the result of attaching a 2-handle along the $m$-framed unknot, equivalently, it is a disk bundle over $S^2$ with the Euler number $m$. 
%We have the following three results depending on the sign of $m$. 
First, we consider the case $m=0$ 
which corresponds to $W_0=S^2\times D^2$.

\begin{theorem}
\label{thm: unknot zero intro}
Suppose $m=0$ and $\alpha\in H_2(W_0)=\Z$, then:
\begin{enumerate}
\item[(a)]  For $\alpha$ odd, $\FL(W_0,\emptyset;\alpha)$ has one generator in each odd normalized Maslov grading $1,3,\ldots$.  
\item[(b)] For $\alpha$ even, $\FL(W_0,\emptyset;\alpha)$ has one generator in each even normalized Maslov grading $2,4,\ldots$ and an additional generator in normalized Maslov grading $1$.
\end{enumerate}
All these generators have Alexander grading 0.
\end{theorem}

 The action of the braid groups on $k_{+}$ and $k_{-}$ strands clearly factors through the corresponding symmetric groups $S_{k_{+}}$ and $S_{k_{-}}$. Thus the main step in computing \eqref{eq: cabled floer lasagna intro} reduces to the computation of {\emph coinvariants} of a certain $S_k$-representation in characteristic 2 which requires some caution. We recall all necessary facts from representation theory in Section \ref{sec: rep theory}.

\subsection{The case of $\overline{\mathbb{CP}^2}$.}

Next, we consider the case $m=-1$ which corresponds to $\overline{\mathbb{CP}^2}$. In this case we get the complete description of the lasagna module, generalizing the main result of \cite{McCulloch}.

\begin{theorem}
\label{thm: unknot negative intro}
Suppose $m=-1$, and $\alpha\in H_2(W_m)=\Z$, then:
\begin{enumerate}
\item[(a)]  Given $\alpha$ odd and $\jj\in \Z$, the space $\FL(W_m,\emptyset;\alpha)$ has generators in Alexander grading $-\jj\alpha$ and normalized Maslov gradings $$
\left(\frac{\alpha+1}{2}-\jj\right)\left(\frac{\alpha-1}{2}-\jj\right)+\{1,3,\ldots\}
$$  
\item[(b)] Given $\alpha$ even and $\jj\in \Z+\frac{1}{2}$, the space $\FL(W_m,\emptyset;\alpha)$ has generators in Alexander grading $-\jj\alpha$ and normalized Maslov gradings  
$$
\left(\frac{\alpha+1}{2}-\jj\right)\left(\frac{\alpha-1}{2}-\jj\right)+\{1,2,4,\ldots\}.
$$ 
\end{enumerate}
\end{theorem}

We have the following. 

\begin{corollary}
\label{cor: finite dimensional}
For $m=-1$ and any given $\alpha$, the space $\FL(W_m,\emptyset;\alpha)$ is finite-dimensional in each bigrading. Furthermore, the normalized Maslov gradings  are bounded below and $\FL(W_m,\emptyset;\alpha)$ is finite-dimensional in each Maslov grading.
%\EG{Compare with \cite{MN}}
%\EG{Cite \cite{McCulloch} in detail}
\end{corollary}

\begin{proof}
%For $m\ge 0$ this is clear from Theorems \ref{thm: unknot zero intro} and \ref{thm: unknot positive intro}. 

For   $\alpha\neq 0$ this is clear from Theorem \ref{thm: unknot negative intro}. 
For $\alpha=0$ all generators have Alexander grading zero, so one needs to be careful. However we note that for a given Maslov grading $M$ there are only finitely many half-integers $\jj$ such that 
$$
\left(\frac{\alpha+1}{2}-\jj\right)\left(\frac{\alpha-1}{2}-\jj\right)<M
$$
and the result follows.
\end{proof}

\begin{example}
For $m=-1$ we show the generator in $(\jj,M)$ coordinates in Figure \ref{fig: negative unknot} for $\alpha=1$ (left) and $\alpha=0$ (right).
\begin{figure}[ht!]
\begin{tikzpicture}[scale=0.8]
\draw [dotted] (0.5,0.75) parabola (3,7);
\draw [dotted] (0.5,0.75) parabola (-2,7);
\draw[dotted,->] (0,0)--(0,7.2);
\draw[dotted,->] (-3,0)--(4,0);
\draw (-2,-0.1)--(-2,0.1);
\draw (-1,-0.1)--(-1,0.1);
\draw (0,-0.1)--(0,0.1);
\draw (1,-0.1)--(1,0.1);
\draw (2,-0.1)--(2,0.1);
\draw (3,-0.1)--(3,0.1);
\draw (-0.1,2)--(0.1,2);
\draw (-0.1,4)--(0.1,4);
\draw (-0.1,6)--(0.1,6);
\draw (4,-0.3) node {$\jj$};
\draw (0.2,7) node {$M$};
\draw (-2,-0.3) node {$\scriptstyle{-2}$};
\draw (-1,-0.3) node {$\scriptstyle{-1}$};
\draw (0,-0.3) node {$\scriptstyle{0}$};
\draw (1,-0.3) node {$\scriptstyle{1}$};
\draw (2,-0.3) node {$\scriptstyle{2}$};
\draw (3,-0.3) node {$\scriptstyle{3}$};
\draw (-2,7) node {$\bullet$};
\draw (-1,3) node {$\bullet$};
\draw (-1,5) node {$\bullet$};
\draw (-1,7) node {$\bullet$};
\draw (0,1) node {$\bullet$};
\draw (0,3) node {$\bullet$};
\draw (0,5) node {$\bullet$};
\draw (0,7) node {$\bullet$};
\draw (1,1) node {$\bullet$};
\draw (1,3) node {$\bullet$};
\draw (1,5) node {$\bullet$};
\draw (1,7) node {$\bullet$};
\draw (2,3) node {$\bullet$};
\draw (2,5) node {$\bullet$};
\draw (2,7) node {$\bullet$};
\draw (3,7) node {$\bullet$};
\end{tikzpicture}
\quad\quad
\begin{tikzpicture}[scale=0.8]
\draw [dotted] (0.5,0.75) parabola (3,7);
\draw [dotted] (0.5,0.75) parabola (-2,7);
\draw[dotted,->] (0.5,0)--(0.5,7.2);
\draw[dotted,->] (-3,0)--(4,0);
\draw (-1.5,-0.1)--(-1.5,0.1);
\draw (-0.5,-0.1)--(-0.5,0.1);
\draw (0.5,-0.1)--(0.5,0.1);
\draw (1.5,-0.1)--(1.5,0.1);
\draw (2.5,-0.1)--(2.5,0.1);
\draw (3.5,-0.1)--(3.5,0.1);
\draw (0.4,1)--(0.6,1);
\draw (0.4,2)--(0.6,2);
\draw (0.4,3)--(0.6,3);
\draw (0.4,4)--(0.6,4);
\draw (0.4,5)--(0.6,5);
\draw (0.4,6)--(0.6,6);
\draw (4,-0.3) node {$\jj$};
\draw (0.2,7) node {$M$};
\draw (-1.5,-0.3) node {$\scriptstyle{-2}$};
\draw (-0.5,-0.3) node {$\scriptstyle{-1}$};
\draw (0.5,-0.3) node {$\scriptstyle{0}$};
\draw (1.5,-0.3) node {$\scriptstyle{1}$};
\draw (2.5,-0.3) node {$\scriptstyle{2}$};
\draw (3.5,-0.3) node {$\scriptstyle{3}$};
\draw (-2,7) node {$\bullet$};
\draw (-1,3) node {$\bullet$};
\draw (-1,4) node {$\bullet$};
\draw (-1,6) node {$\bullet$};
\draw (0,1) node {$\bullet$};
\draw (0,2) node {$\bullet$};
\draw (0,4) node {$\bullet$};
\draw (0,6) node {$\bullet$};
\draw (1,1) node {$\bullet$};
\draw (1,2) node {$\bullet$};
\draw (1,4) node {$\bullet$};
\draw (1,6) node {$\bullet$};
\draw (2,3) node {$\bullet$};
\draw (2,4) node {$\bullet$};
\draw (2,6) node {$\bullet$};
\draw (3,7) node {$\bullet$};
\end{tikzpicture}
\caption{The generators for the $(-1)$-framed unknot for $\alpha=1$ (left) and $\alpha=0$ (right).}
\label{fig: negative unknot}
\end{figure}
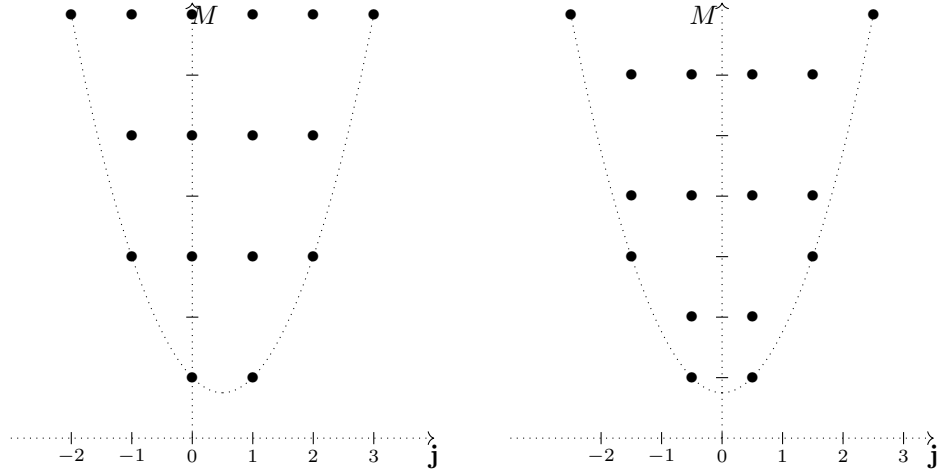
The pictures for general odd (resp. even) $\alpha$ are obtained from this one by a horizontal shift.
\end{example}

The cable $O_1(k_{+},k_{-})$ is the positive torus link $T(n,n)$ (where $n=k_{+}+k_{-}$) with appropriately oriented components. The homology $\HFLhat(T(n,n))$ was computed  in \cite{GHom}. Furthermore, the ``full" homology $\HFL(T(n,n))$ was computed in \cite{BLZ}
and it was proved in \cite{BLZ} that the chain complex $\CFL(T(n,n))$ is quasi-isomorphic to the free resolution of $\HFL(T(n,n))$. This result equips $\CFL(T(n,n))$ and $\HFLhat(n,n)$ with the additional \emph{level grading}.

The cable $O_{-1}(k_{+},k_{-})$, by contrast, is the negative torus link $T(n,-n)$   with appropriately oriented components.  To prove Theorem \ref{thm: unknot negative intro}, we use the above results and the duality between the Heegaard Floer homology of $T(n,n)$ and $T(n,-n)$, see Section \ref{sec: -1 framed unknot}.

\subsection{Negative framings for L-space knots}

The above results can be generalized to the case when $K$ is an L-space knot. Recall that a rational homology 3-sphere $M$ is an L-space if $\widehat{\mathrm{HF}}(M)$ has smallest possible rank \cite{OSlens}. A link $L$ is an L-space link if all sufficiently large Dehn surgeries of $S^3$ along $L$ yield L-spaces. Note that the definition of an L-space link does not depend on the orientation of its components. 

If $K$ is an L-space knot and $m>2g(K)-1$ then by \cite[Theorem 2]{GHom} all cables of $K$ with framing $m$ and arbitrary number of components $n=k_{+}+k_{-}$ are L-space links. 
Furthermore, \cite{GHom} gives an explicit (although somewhat involved) recipe for computing $\HFLhat$ for such cables of $K$ in terms of the Alexander polynomial of $K$. This generalizes the computation for $T(n,mn)$. Assuming additionally that $K$
 is a \emph{ plumbed} L-space knot, the cables of $K$ with framing $m>2g(K)-1$ are plumbed L-space links, and the results of \cite{BLZ} apply. 

 The duality argument used in the proof of Theorem \ref{thm: unknot negative intro} can be generalized as well. The mirror of the $(n,mn)$ cable of a knot $K$ is the $(n,-mn)$ cable of its mirror $\overline{K}$, and we prove the following.  Consider the equivalence relation on $\Z_{>0}\times \Z_{>0}$ given by 
 $$
 (j,n)\sim (j+m,n+2).
 $$

\begin{theorem}
\label{thm: L space negative intro}
Let $K$ be a plumbed L-space knot, $\overline{K}$ its mirror and $m>\max\{0, 2g(K)-1\}$.  Then the Floer lasagna module   $\FL\left(W_{-m}(\overline{K}\right),\emptyset,\alpha)$ has  generators indexed by the equivalence classes of $(j,n), n=\alpha\mod 2$, and an additional integer $i$  such that:
\begin{enumerate}
\item[(a)] For $\alpha$ odd, $i\in \{0,2,4,\ldots\}$
\item[(b)] For $\alpha$ even, $i\in \{0,1,3,\ldots\}$.
\end{enumerate}
Such a generator has Alexander grading $-\alpha(g_n(K)-j)$ and normalized Maslov grading 
$
-M(Y'_{j,n})+1+i.
$
%\begin{enumerate}
%\item[(a)] For $\alpha$ odd, we have normalized gradings
%$$
%-M(Y'_{j,n})+1+\{0,2,4,\ldots\}
%$$
%\item[(a)] For $\alpha$ even, we have normalized gradings
%$$
%-M(Y'_{j,n})+1+\%{0,1,3,\ldots\}.
%$$
%\end{enumerate}
Here $\alpha=k_{+}-k_{-}$, $g_n(K)$ is defined by \eqref{eq: def gn} and $M(Y'_{j,n})$ is defined by \eqref{eq: Maslov Y' L space}.
\end{theorem}

Again, this is closely related to the work of McCulloch \cite{McCulloch} who computed the component of  the lasagna module $\FL\left(W_{-m}(\overline{K}),\emptyset;\alpha\right)$ in certain specific Alexander and Maslov gradings.

\subsection{The case of $\mathbb{CP}^2$.}

Next, we consider the case of $m=1$, then $W_1=W_1(O)=\mathbb{CP}^2$.

\begin{conjecture}
\label{conj: unknot positive intro}
Suppose $m=1$, then $\FL(\mathbb{CP}^2,\emptyset)$ is described as follows:
\begin{enumerate}
\item[(a)]  For $\alpha$ odd and $\jj\in \Z$ the space $\FL(\mathbb{CP}^2,\emptyset;\alpha)$ has generators in Alexander grading $\jj\alpha$ and normalized Maslov gradings 
$$
-\left(\frac{\alpha+1}{2}-\jj\right)\left(\frac{\alpha-1}{2}-\jj\right)+2+\{0,2,4,\ldots\}
$$
\item[(b)] For $\alpha$ even and $\jj\in \Z+\frac{1}{2}$ the space $\FL(\mathbb{CP}^2,\emptyset;\alpha)$ has generators in Alexander grading $\jj\alpha$ and normalized Maslov gradings  
$$
-\left(\frac{\alpha+1}{2}-\jj\right)\left(\frac{\alpha-1}{2}-\jj\right)+2+\{0,1,3,\ldots\}
$$
\end{enumerate}
Here $\alpha=k_{+}-k_{-}$ corresponds to the grading by $H_2(W_1(O))=H_2(\mathbb{CP}^2)$.
\end{conjecture}

We provide a substantial evidence for Conjecture \ref{conj: unknot positive intro} in Section \ref{sec: 1 framed unknot} and Appendices \ref{sec: appendix A},\ref{sec: appendix B}.   In Appendix \ref{sec: appendix A} we propose an explicit, but quite complicated algebraic model for $\CFL(T(n,n))$ building on results of \cite{BLZ,GHom}, see Conjecture \ref{conj: algebraic complex}.

%In Appendix \ref{sec: appendix C}
%we use the results of \cite{BLZ,GHom} and the interaction between the Alexander, Maslov and level gradings, to confirm some parts of this conjecture.

We also describe unconditionally the action of the braid group  and basepoint moving maps for $\HFLhat(n,n)$. Assuming Conjecture \ref{conj: algebraic complex}, we can also describe the pair-of-pants cobordism maps between $\HFLhat(n,n)$ and $\HFLhat(n+2,n+2)$ and prove the following, see Theorem \ref{thm: 1 framed unknot} for more details.

\begin{theorem}
\label{thm: unknot positive intro}
Algebraic Conjecture \ref{conj: algebraic complex} implies Conjecture \ref{conj: unknot positive intro}. 
\end{theorem}

We would like to emphasize that, if true, Conjecture \ref{conj: unknot positive intro} is in stark contrast with its Khovanov homology analogue, where by \cite[Theorem 1.4]{RW} the skein lasagna module for $\mathbb{CP}^2$ vanishes in all gradings. We also highlight another issue with the proposed answer for $\mathbb{CP}^2$, compare with Corollary \ref{cor: finite dimensional}.

\begin{corollary}
\label{cor: infinite dimensional}
Assuming Conjecture \ref{conj: unknot positive intro}, for $\alpha\neq 0$ the space $\FL(\mathbb{CP}^2,\emptyset;\alpha)$  is finite-dimensional in each bigrading.  

However, for $\alpha=0$ the space $\FL(\mathbb{CP}^2,\emptyset;\alpha)$  is concentrated in Alexander grading zero, is at most two-dimensional in even Maslov gradings and is infinite-dimensional in each odd Maslov grading.
%Furthermore, the normalized Maslov gradings  are unbounded below and $\FL(W_m,\emptyset;\alpha)$ is infinite-dimensional in each Maslov grading.
\end{corollary}

\begin{proof}
The case $\alpha\neq 0$ is clear. For $\alpha=0$ we have for each $\jj=\frac{1}{2}+\jj'\in \Z+\frac{1}{2},i\in \{0,1,3,\ldots\}$ we have a generator of normalized Maslov grading
$$
M_{\norm}(\jj,i)=-(\frac{1}{2}-\jj)(-\frac{1}{2}-\jj)+2+i=-\jj'(\jj'+1)+2+i.
$$
Note that for $i=0$ the value of $M_{\norm}(\jj,i)$  is even while for $i>0$ it is odd.
For each odd $M$ and sufficiently large $|\jj|$ we have 
$$
-\jj'(\jj'+1)+2<M
$$
and hence we can a positive odd  $i=M+\jj'(\jj'+1)-2$ such that $M_{\norm}(\jj,i)=M$.
\end{proof}

%The proof of Theorem \ref{thm: unknot positive intro}  is considerably more complicated than Theorem \ref{thm: unknot zero intro}. For $m>0$, the cable $O_m(k_{+},k_{-})$ is the torus link $T(n,mn)$ (where $n=k_{+}+k_{-}$) with appropriately oriented components. The homology $\HFLhat(T(n,mn))$ was computed  in \cite{GHom}. Furthermore, the ``full" homology $\HFL(T(n,mn))$ was computed in \cite{BLZ}
%and it was proved in \cite{BLZ} that the chain complex $\CFL(T(n,mn))$ is quasi-isomorphic to the free resolution of $\HFL(T(n,mn))$. This result equips $\CFL(T(n,mn))$ and $\HFLhat(n,mn)$ with the additional \emph{level grading}. 

%Next, we use the above results to study the cobordism maps. First, we prove that just as in 0-framed case the braid group actions factors through the symmetric groups, and determine the corresponding representations over $\F$. Next, we study the connecting maps  between $\HFLhat(T(n,mn))$ and $\HFLhat(T(n+2,m(n+2)))$ by comparing and constrasting various facts about $\HFL$ and $\HFLhat$ for these links from \cite{BLZ,GHom}, understanding the interaction of connecting maps with the Alexander, Maslov and level gradings, and using some relations between the cobordism maps from Section \ref{sec: cobordisms}. This allows us to prove Theorem \ref{thm: unknot positive intro} as Theorem \ref{thm: 1 framed unknot} for $m=1$. The case $m>1$ is similar and is proved in Theorem \ref{thm: L space positive}.

\begin{example}
For $m=1$ we get the generators in $(\jj,M)$ coordinates in Figure \ref{fig: positive unknot} for $\alpha=1$ (left) and $\alpha=0$ (right).
\begin{figure}[ht!]
\begin{tikzpicture}[scale=0.8]
\draw [dotted] (0.5,2.25) parabola (3,-4);
\draw [dotted] (0.5,2.25) parabola (-2,-4);
\draw[dotted,->] (0,-4)--(0,4.2);
\draw[dotted,->] (-3,0)--(4,0);
\draw (-2,-0.1)--(-2,0.1);
\draw (-1,-0.1)--(-1,0.1);
\draw (0,-0.1)--(0,0.1);
\draw (1,-0.1)--(1,0.1);
\draw (2,-0.1)--(2,0.1);
\draw (3,-0.1)--(3,0.1);
\draw (-0.1,2)--(0.1,2);
\draw (-0.1,4)--(0.1,4);
%\draw (-0.1,6)--(0.1,6);
\draw (4,-0.3) node {$\jj$};
\draw (0.2,4) node {$M$};
\draw (-2,-0.3) node {$\scriptstyle{-2}$};
\draw (-1,-0.3) node {$\scriptstyle{-1}$};
\draw (0,-0.3) node {$\scriptstyle{0}$};
\draw (1,-0.3) node {$\scriptstyle{1}$};
\draw (2,-0.3) node {$\scriptstyle{2}$};
\draw (3,-0.3) node {$\scriptstyle{3}$};
\draw (-2,-4) node {$\bullet$};
\draw (-2,-2) node {$\bullet$};
\draw (-2,0) node {$\bullet$};
\draw (-2,2) node {$\bullet$};
\draw (-2,4) node {$\bullet$};
\draw (-1,0) node {$\bullet$};
\draw (-1,2) node {$\bullet$};
\draw (-1,4) node {$\bullet$};
\draw (0,2) node {$\bullet$};
\draw (0,4) node {$\bullet$};
%\draw (0,5) node {$\bullet$};
%\draw (0,7) node {$\bullet$};
%\draw (1,0) node {$\bullet$};
\draw (1,2) node {$\bullet$};
\draw (1,4) node {$\bullet$};
%\draw (1,7) node {$\bullet$};
\draw (2,0) node {$\bullet$};
\draw (2,2) node {$\bullet$};
\draw (2,4) node {$\bullet$};
\draw (3,-4) node {$\bullet$};
\draw (3,-2) node {$\bullet$};
\draw (3,-0) node {$\bullet$};
\draw (3,2) node {$\bullet$};
\draw (3,4) node {$\bullet$};
\end{tikzpicture}
\quad\quad
\begin{tikzpicture}[scale=0.8]
\draw [dotted] (0.5,2.25) parabola (3,-4);
\draw [dotted] (0.5,2.25) parabola (-2,-4);
\draw[dotted,->] (0.5,-4)--(0.5,4.2);
\draw[dotted,->] (-3,0)--(4,0);
\draw (-1.5,-0.1)--(-1.5,0.1);
\draw (-0.5,-0.1)--(-0.5,0.1);
\draw (0.5,-0.1)--(0.5,0.1);
\draw (1.5,-0.1)--(1.5,0.1);
\draw (2.5,-0.1)--(2.5,0.1);
\draw (3.5,-0.1)--(3.5,0.1);
\draw (-0.1,2)--(0.1,2);
\draw (-0.1,4)--(0.1,4);
%\draw (-0.1,6)--(0.1,6);
\draw (4,-0.3) node {$\jj$};
\draw (0.7,4) node {$M$};
\draw (-1.5,-0.3) node {$\scriptstyle{-2}$};
\draw (-0.5,-0.3) node {$\scriptstyle{-1}$};
\draw (0.5,-0.3) node {$\scriptstyle{0}$};
\draw (1.5,-0.3) node {$\scriptstyle{1}$};
\draw (2.5,-0.3) node {$\scriptstyle{2}$};
\draw (3.5,-0.3) node {$\scriptstyle{3}$};
\draw (-2,-4) node {$\bullet$};
\draw (-2,-3) node {$\bullet$};
\draw (-2,-1) node {$\bullet$};
\draw (-2,1) node {$\bullet$};
\draw (-2,3) node {$\bullet$};
\draw (-1,0) node {$\bullet$};
\draw (-1,1) node {$\bullet$};
\draw (-1,3) node {$\bullet$};
\draw (0,2) node {$\bullet$};
\draw (0,3) node {$\bullet$};
%\draw (0,5) node {$\bullet$};
%\draw (0,7) node {$\bullet$};
%\draw (1,0) node {$\bullet$};
\draw (1,2) node {$\bullet$};
\draw (1,3) node {$\bullet$};
%\draw (1,7) node {$\bullet$};
\draw (2,0) node {$\bullet$};
\draw (2,1) node {$\bullet$};
\draw (2,3) node {$\bullet$};
\draw (3,-4) node {$\bullet$};
\draw (3,-3) node {$\bullet$};
\draw (3,-1) node {$\bullet$};
\draw (3,1) node {$\bullet$};
\draw (3,3) node {$\bullet$};
\end{tikzpicture}
\caption{The generators for the $1$-framed unknot for $\alpha=1$ (left) and $\alpha=0$ (right).} 
\label{fig: positive unknot}
\end{figure}
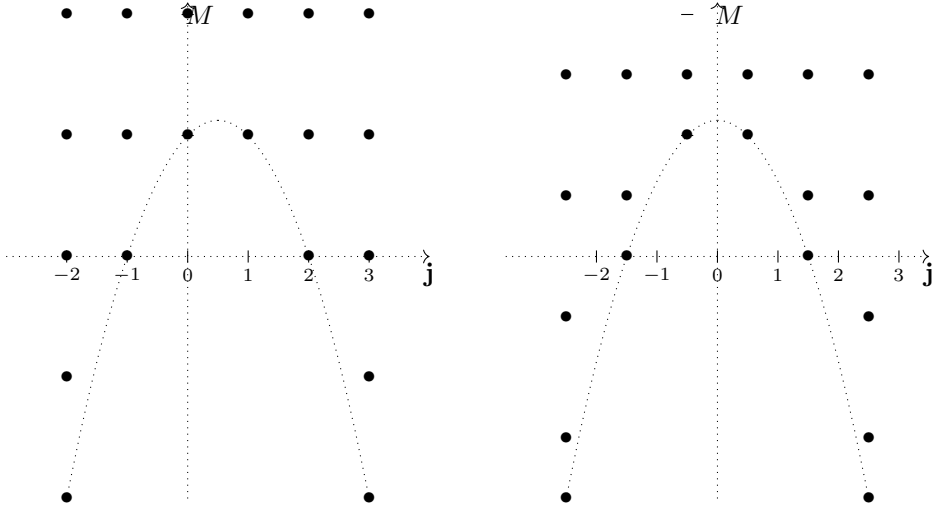
The pictures for general odd (resp. even) $\alpha$ are obtained from this one by a horizontal shift.
\end{example}

%\subsection{Positive framings for L-space knots}

Finally, for positively framed L-space knots we propose the following generalization of Conjecture \ref{conj: unknot positive intro}. As above,  we consider the equivalence relation on $\Z_{>0}\times \Z_{>0}$ given by 
 $$
 (j,n)\sim (j+m,n+2).
 $$

\begin{conjecture}
\label{conj: L space positive intro}
Suppose $K$ is a plumbed L-space knot and $m>\max(2g(K)-1,0)$. Then the Floer lasagna module $\FL\left(W_m(K),\emptyset;\alpha\right)$ has generators indexed by an equivalence class of $(j,n), n=\alpha\mod 2$ and an integer $i$ satisfying constraints from Theorem \ref{thm: L space negative intro}. Such a generator has Alexander grading 
$\alpha(g_n(K)-j)$ 
and Maslov grading
$M(Y'_{j,n})+2+i$.
%we have the following description of the cabled Floer homology:
%\begin{enumerate}
%\item[(a)]  For $\alpha$ odd and an equivalence class $(j,n), n=\alpha\mod 2$ the space $\FL(W_m(K),\emptyset;\alpha)$ has generators in Alexander degree $\alpha(g_n(K)-j)$ and normalized Maslov degrees 
%$$
%M(Y'_{j,n})+2+\{0,2,4,\ldots\}
%$$
%\item[(b)] For $\alpha$ even and and an equivalence class $(j,n), n=\alpha\mod 2$ the space $\FL(W_m(K),\emptyset;\alpha)$ has generators in Alexander degree $\alpha(g_n(K)-j)$ and normalized Maslov degrees  
%$$
%M(Y'_{j,n})+2+\{0,1,3,\ldots\}
%$$
%\end{enumerate}
\end{conjecture}

\subsection{Wall stabilization and future directions}

In the future, we aim to prove Conjectures \ref{conj: unknot positive intro} and \ref{conj: L space positive intro} by constructing explicit algebraic models for the full Heegaard floer complexes of cables of L-space links as in Appendix \ref{sec: appendix A}, and using these to compute various cobordism maps.

We would also like to address the potential topological issues raised by Conjecture \ref{conj: unknot positive intro} and Corollary \ref{cor: infinite dimensional}. 
If $(X,X')$ are an exotic pair of simply connected 4-manifolds then by a result of Wall \cite{Wall} they become diffeomorphic after stabilization:
$$
X\sharp k\mathbb{CP}^2\sharp \ell\overline{\mathbb{CP}^2}\simeq X'\sharp k\mathbb{CP}^2\sharp \ell\overline{\mathbb{CP}^2},\ k,\ell\gg 0.
$$
Similarly to \cite[Theorem 1.4]{MN}, one would expect that the connect sum of 4-manifolds corresponds to the tensor product of Floer lasagna modules. Since the answers for $\mathbb{CP}^2$ and $\overline{\mathbb{CP}^2}$ are infinite-dimensional (and the proposed answer for $\mathbb{CP}^2$ is even infinite-dimensional in infinitely many bigradings), this does not immediately obstruct a possibility of $\FL(X)\neq \FL(X')$ for such an exotic pair. 

Also, since the proposed answers for $\mathbb{CP}^2$ and $\overline{\mathbb{CP}^2}$ are structurally different, the Floer lasagna module seems to be sensitive to the orientation which might be a good sign for a potential exotica detection. 

Still, the contrast with Khovanov homology is somewhat disconcerting. As part of the proof of Theorem \ref{thm: 1 framed unknot}, we note that none of the classes in $\HFLhat(T(n,n))$ which might survive in $\cHFL(O_1)$ correspond to the generators of $\HFLminus(T(n,n))$. Motivated by this observation, we propose the following:

\begin{problem}
Define and study variants of Floer lasagna modules for other flavors of link Floer homology such as $\HFLminus$ or $\HFL$. Compute these invariants for $W_m(K)$ when $K$ is an L-space knot or its mirror.
\end{problem}

\subsection{Organization of the paper}

The paper is organized as follows.  Sections \ref{sec: HF background}, \ref{sec: cobordisms}, \ref{sec: lasagna} and \ref{sec: rep theory} cover the necessary background.

In Sections \ref{sec: HF background} and \ref{sec: cobordisms} we give some background on various flavors of Heegaard Floer homology, in particular recall some useful facts about L-space links. We also specify some decorated cobordisms relating  cable links $K_m(k_{+},k_{-})$ and the corresponding maps in Heegaard Floer homology, and establish various relations between these. In Section \ref{sec: lasagna} we define Floer lasagna modules following \cite{Chen}. In Section \ref{sec: rep theory} we review some useful facts about representation theory of $S_n$ in characteristic 2. 

Next, we start proving our main results. In Section \ref{sec: unknot} we study the Floer lasagna module for the 0-framed unknot and prove Theorem \ref{thm: unknot zero intro}. 

In Section \ref{sec: Tnn} we review the results of \cite{BLZ,GHom} on Heegaard Floer homology of $T(n,n)$, as well as various gradings and cobordism maps.

In Section \ref{sec: -1 framed unknot} we utilize the duality between the Heegaard Floer homologies of $T(n,n)$ and $T(n,-n)$ to study the latter and prove Theorem \ref{thm: unknot negative intro} for $m=-1$.

In Section \ref{sec: 1 framed unknot} we discuss the case $m=1$ and prove Theorem \ref{thm: unknot positive intro}. 

In Section \ref{sec: L space knots} we generalize these results to arbitrary L-space knots, provide some evidence for Conjecture \ref{conj: L space positive intro} and prove Theorem \ref{thm: L space negative intro}.  Many of the proofs in Section \ref{sec: L space knots} are similar to the ones in Sections \ref{sec: -1 framed unknot} and \ref{sec: 1 framed unknot}, but require heavy notations and, in particular, a detailed study of the so-called $h$-functions for L-space links. The notations for the $\pm 1$ framed unknot are easier and  more explicit. This justifies first proving the special case and then the general case along the same lines. 

In the Appendices \ref{sec: appendix A},\ref{sec: appendix B}  we discuss an algebraic model for the full Floer complex of $T(n,n)$
and its implications for the Floer lasagna module of $\mathbb{CP}^2$.
\subsection*{Acknowledgments}

We are grateful to Daren Chen, Ciprian Manolescu, Qiuyu Ren, Ian Sullivan, Paul Wedrich and Melissa Zhang for many useful discussions about lasagna modules. We thank Monica Vazirani for explaining the subtleties of representation theory in characteristic 2.

The paper is partially based upon work supported by the National Science Foundation under Grant DMS-2424139, while B. L. and E. G. were in the residence at the Simons Laufer Mathematical Sciences Institute in Berkeley, California, during the Spring 2026 and Fall 2026 semesters. 
%\EG{Beibei + Eugene - thank MSRI}

%The work of E. G. was partially supported by the NSF grant DMS-2302305.

\subsection*{AI declaration}

All of the results and the proofs are due to the authors. We have used OpenAI's ChatGPT (GPT-5.6 Sol) and Claude Opus 5.5 for proofreading and checking some technical details of the proofs.

\section{Background on link Floer homology}
\label{sec: HF background}

\subsection{Variants and coefficients}

We refer to \cite{OSKnots,OSProperties,OSLinks} for definitions and foundational results on Heegaard Floer homology.
We will be working with links in $S^3$ and the ``full” version of Heegaard Floer complex $\CFL(L)$ for links, defined in \cite{Zemke}. Let
$L = L_1\cup \ldots\cup L_n$ be an oriented link with $n$ components. Unless stated otherwise, we will
assume that each component $L_i$ has   two marked points $\z_i$ and $\w_i$.
We define $\ell_i=\sum_{j\neq i}\lk(L_i,L_j)$ and 
$$
\HH_L=\Z^n+\frac{1}{2}(\ell_1,\ldots,\ell_n).
$$
We consider the ring $R_n=\F[U_1,\ldots,U_n,V_1,\ldots,V_n]$. If the number of components is clear from context we will write $R$ instead of $R_n$.
The complex $\CFL(L)$ consists of free $R$-modules which have $\Z$-valued Maslov  grading and $\HH_L$-valued Alexander  grading, the differential preserves the Alexander grading and decreases the Maslov grading by 1. The variables $U_i$ have Maslov grading $-2$ and Alexander grading $-e_i$ while the variables $V_i$ have Maslov grading $0$ and Alexander grading $e_i$. Here $e_i$ is the $i$-th basis vector in $\Z^n$ which acts on $\HH_L$ by a shift. %\AA{We use $\delta_i$ in section 4 instead of $e_i$, I think}
For $s\in \HH_L$ we define $\mathfrak{A}^-(L,s)$ as the subcomplex of $\CFL(L)$ with Alexander grading $s$, so that
$$
\CFL(L)=\bigoplus_{s\in \HH_L}\mathfrak{A}^-(L,s).
$$
There are chain maps
$$
U_i:\mathfrak{A}^-(L,s)\to \mathfrak{A}^-(L,s-e_i),\ V_i:\mathfrak{A}^-(L,s)\to \mathfrak{A}^-(L,s+e_i).
$$
We also define rings
$$
R_{U}=R/(V_1=\ldots=V_n=0)\simeq \F[U_1,\ldots,U_n],\ R_V=R/(U_1=\ldots=U_n=0)\simeq \F[V_1,\ldots,V_n].
$$

\begin{definition}
We define the ``full" link Floer homology
$$
\HFL(L)=H_*(\CFL(L))=\bigoplus_{s\in \HH_L}H_*(\mathfrak{A}^-(L,s)),
$$
``minus"   homology
$$
\HFLminus(L)=H_*\left(\CFL(L)\otimes_{R}R_U\right)
$$
and ``hat" homology
$$
\HFLhat(L)=H_*\left(\CFL(L)\otimes_{R} 
\F\right).
$$
\end{definition}

The complex $\CFL(L)$, viewed in the homotopy category of $\Z\oplus \HH_L$-graded $R$-modules, is an invariant of $L$. Correspondingly, all three versions of homology $\HFL(L),\HFLminus(L)$ and $\HFLhat(L)$ are invariants of $L$.

\subsection{L-space links}
\label{sec: L space background}

A rational homology sphere $M$ is an L-space if and only if $\widehat{\mathrm{HF}}(M)$ has minimal possible rank. A link $L=L_1\cup\cdots\cup L_n$ is called an L-space link if the Dehn surgery $S^3_{d_1,\ldots,d_n}(L)$ is an L-space for all $d_1,\ldots,d_n\gg 0$.

Equivalently, by \cite{Yajing} a link $L$ is an L-space link if and only if for all $s\in \HH_L$ we have 
$$
H_*(\mathfrak{A}^{-}(L,s))\simeq \F[\UU][-2h(s)]
$$
where $h:\HH_L\to \Z$ is the so-called $h$-function of the link $L$. In particular, $\HFL(L)$ is determined as an $\F$-vector space by the $h$-function, and one can check (see e.g. \cite{BLZ}) that the $R$-module structure is determined by the $h$-function as well. 
By \cite{GN} for an L-space link $L$ the $h$-function is determined by the Alexander polynomial of $L$ and all of its sublinks.

\begin{theorem}\cite{BLZ}
\label{thm: resolution}
Assume $L$ is a \emph{plumbed} L-space link. Then $\CFL(L)$ is homotopy equivalent to a free resolution of $\HFL(L)$ over $R$.
\end{theorem}

The immediate consequence of Theorem \ref{thm: resolution} is that for plumbed L-space links the complex $\CFL(L)$ acquires an additional {\bf level} grading such that $\HFL(L)$ has level $0$, differential decreases the level by $1$, and the action of $R$ preserves the level.  
By definition, the homologies $\HFLminus(L)$ and $\HFLhat(L)$ also inherit the level grading from $\CFL(L)$ in this case.

Next, we recall a standard result from homological algebra (see for example \cite[Chapter 17, Proposition 4]{DummitFoote}) which will be used repeatedly below. 
\begin{lemma}
\label{lem: lifting to resolutions}
Suppose $A_{\bullet}$ and $B_{\bullet}$ are two chain complexes of projective modules such that $H(B_{\bullet})$ is concentrated at level zero. Given a linear map $f:H_0(A_{\bullet})\to H_0(B_{\bullet})$, there exists a level-preserving chain map $F:A_{\bullet}\to B_{\bullet}$ which induces $f$ in homology (for $i>0$ the homology $H_i(A_{\bullet})$ is necessarily mapped to 0). Furthermore, such $F$ is unique up to a chain homotopy. 
\end{lemma}

We will need to consider more general maps which do not necessarily preserve the level. 
Assume that $A_{\bullet}$ and $B_{\bullet}$ are bounded complexes. 
We say that $F:A_{\bullet}\to B_{\bullet}$ is a chain map if $[\partial, F]=0.$
We can write $F=\sum_{k}F_k$ where $F_k$ changes the level by $k$, and only finitely many $F_k$ are nonzero.

\begin{definition}
We say that $F$ is level-filtered if $F_k=0$ for $k>0$.
\end{definition}

The following corollary will be useful for the topological applications. 

\begin{corollary}
\label{cor: homotopic to level filtered}
Suppose that $A_{\bullet}$ and $B_{\bullet}$ are two bounded complexes of projective modules such that $H(B_{\bullet})$ is  concentrated at level zero. Then:

a) Any chain map $F:A_{\bullet}\to B_{\bullet}$ which changes the level by $k>0$ is null-homotopic.

b) Any chain map $F:A_{\bullet}\to B_{\bullet}$ is homotopic to a level-filtered chain map. Its level-preserving component $F_0$ is determined by the map $F_{\ast}:H_0(A_{\bullet})\to H_0(B_{\bullet})$ and $F_0$ is homotopic to the map from Lemma \ref{lem: lifting to resolutions}.
\end{corollary}

\begin{proof}
Part (a) follows from Lemma \ref{lem: lifting to resolutions} since the corresponding map on homology vanishes. 

For part (b), we have $[\partial,F]=\sum_{k} [\partial,F_k]=0$, and $[\partial,F_k]$ changes the level by $k-1$. Therefore $[\partial,F_k]=0$ for all $k$, so  $F_k$ is a chain map and by part (a) it is null-homotopic for $k>0$. We conclude that $F$ is homotopic to $\sum_{k\le 0} F_k$.
\end{proof}

The following example shows that the components $F_{k}$ do not need to be null-homotopic for $k<0$.

\begin{example}
The following diagram defines a chain map which decreases the level by 1:
$$
\begin{tikzcd}
A_{\bullet}: & R_n \arrow{r}{U_1} \arrow{dr}{1} & R_n\\
B_{\bullet}: &   & R_n
\end{tikzcd}
$$
It is easy to check that it is not null-homotopic.
\end{example}

\subsection{Spectral sequences}

The free resolution of $R_U$ over $R$ is given by 
$$
R_{U}\simeq (R\otimes\wedge(\mu_1,\ldots,\mu_n),\partial_{\mu}),\ \partial_{\mu}(\mu_i)=V_i.
$$
Here $\mu_i$ have Alexander grading $e_i$ and Maslov grading $1$.
Similarly, the free resolution of $R_V$ over $R$ is given by 
$$
R_{V}\simeq (R\otimes\wedge(\nu_1,\ldots,\nu_n),\partial_{\nu}),\ \partial_{\nu}(\nu_i)=U_i.
$$
Here $\nu_i$ have Alexander grading $-e_i$ and Maslov grading $-1$.
Finally, the free resolution of $\F$ over $R$ is given by 
$$
\F\simeq (R\otimes\wedge(\mu_1,\ldots,\mu_n)\otimes\wedge(\nu_1,\ldots,\nu_n),\partial_{\mu}+\partial_{\nu}).
$$

The next lemma is a reformulation of some useful spectral sequences considered in \cite{GN,GHom}.

\begin{lemma}
\label{lem: spectral sequences}
Assume $L$ is a plumbed L-space link. Then:

a) There is a spectral sequence
$$
\HFL(L)\otimes\wedge(\mu_1,\ldots,\mu_n)\Rightarrow \HFLminus(L)
$$
with the first differential given by $\partial_{\mu}$. All higher differentials vanish and the spectral sequence collapses at the $E_2$ page.

b)  There is a spectral sequence
$$
\HFLminus(L)\otimes\wedge(\nu_1,\ldots,\nu_n)\Rightarrow \HFLhat(L)
$$
with the first differential given by $\partial_{\nu}$.

The differentials in the spectral sequences preserve Alexander grading and lower Maslov grading by 1.
Moreover, they decrease the level grading by 1 if we assume that $\mu_i,\nu_i$ have level 1. 
\end{lemma}

\begin{proof}
a) By Theorem \ref{thm: resolution} we can compute $\HFLminus(L)$ in two ways. First, we consider the resolution of $\HFL(L)$ over $R$ (given by $\CFL(L)$) and tensor it with $R_U$. Second, we can consider the resolution of $R_U$ and tensor it with $\HFL(L)$. This is exactly
$$
\left(\HFL(L)\otimes\wedge(\mu_1,\ldots,\mu_n),\partial_{\mu}\right).
$$
By the properties of the $\mathrm{Tor}$ functor, the homologies of the resulting complexes are the same, and the result follows.

b) We have a spectral sequence with the $E_2$ page
$$
H_*\left(H_*\left(\HFL(L)\otimes\wedge(\mu_1,\ldots,\mu_n)\otimes\wedge(\nu_1,\ldots,\nu_n),\partial_{\mu}\right),\partial_{\nu}\right)
$$
and the $E_{\infty}$ page
$$
H_*\left(\HFL(L)\otimes\wedge(\mu_1,\ldots,\mu_n)\otimes\wedge(\nu_1,\ldots,\nu_n),\partial_{\mu}+\partial_{\nu}\right).
$$
By part (a), the $E_2$ page can be rewritten as
$$
H_*\left(\HFLminus(L) \otimes\wedge(\nu_1,\ldots,\nu_n),\partial_{\nu}\right).
$$
On the other hand, the $E_{\infty}$ page computes  $\HFLhat(L)$ using the resolution of $\F$ over $R$ (similarly to (a)).
\end{proof}

\begin{remark}
\label{rem: cube model ss}
For a fixed Alexander grading $s\in \HH_L$ the spectral sequences decompose into cubes
$$
\bigoplus_{A\subset \{1,\ldots,n\}}H_*(\mathfrak{A}^{-}(L,s-e_A))\otimes \mu_A\Rightarrow \HFLminus(L,s)
$$
and
$$
\bigoplus_{A\subset \{1,\ldots,n\}}\HFLminus(L,s+e_A)\otimes \nu_A\Rightarrow \HFLhat(L,s).
$$
Here $\mu_A=\prod_{a\in A}\mu_a$ and $\nu_A=\prod_{a\in A}\nu_a.$
This is the form in which these spectral sequences appeared in \cite{GHom}, and the level grading $|A|$ was referred to as the ``cube grading". 
\end{remark}

Recall that for all $a\neq b$ the actions of $U_aV_a$ and $U_bV_b$ on $\CFL(L)$ are homotopic, so there is a map 
$$
h_{a,b}:\CFL(L)\to \CFL(L),\ [\partial,h_{a,b}]=U_aV_a-U_bV_b.
$$
These homotopies can be chosen to anticommute, square to zero and satisfy $h_{a,b}+h_{b,c}=h_{a,c}$.
In particular, it is sufficient to consider $h_{1,a}$ for $a>1$. After specializing $V_i=0$ or $U_i=V_i=0$ we get chain maps 
$$
h_{a,b}:\CFLminus(L)\to \CFLminus(L),\quad  \CFLhat(L)\to \CFLhat(L)
$$
which descend to the action of the exterior algebra $\wedge\left(h_{1,2},\ldots,h_{1,n}\right)$ in homology $\HFLminus(L)$ or $\HFLhat(L)$. %We will refer to this as to unlink action in homology, 
See Lemma \ref{lem: UV homotopy arc} for a more topological interpretation.

The next lemma describes the action of $\wedge(h_{1,2},\ldots,h_{1,n})$ in terms of spectral sequences of Lemma \ref{lem: spectral sequences}. 

\begin{lemma}
\label{lem: unknot action ss}
Assume that $L$ is a plumbed L-space link. The action of $\wedge(h_{1,2},\ldots,h_{1,n})$ on $\HFLminus(L)$ commutes with the spectral sequence 
$$
\HFL(L)\otimes\wedge(\mu_1,\ldots,\mu_n)\Rightarrow \HFLminus(L)
$$ 
where the action on the left is given by 
$$
h'_{a,b}=U_a\mu_a-U_b\mu_b.
$$
\end{lemma}

\begin{proof}
Consider the chain complex 
$$
\left(\CFL(L)\otimes \wedge(\mu_1,\ldots,\mu_n),\partial+\partial_{\mu}\right),
$$
which carries chain maps
$$
h''_{a,b}=h_{a,b}-U_a\mu_a+U_b\mu_b,\ [\partial+\partial_{\mu},h''_{a,b}]=(U_aV_a-U_bV_b)-(U_aV_a-U_bV_b)=0. 
$$
If we compute the homology with respect to $\partial_{\mu}$,
we get $\CFLminus(L)$ with chain maps $h_{a,b}$.
On the other hand, the complex is quasi-isomorphic to 
$$
(\HFL(L)\otimes \wedge(\mu_1,\ldots,\mu_n),\partial_{\mu}),
$$
where $h''_{a,b}$ correspond to $h'_{a,b}$ and the result follows. 
\end{proof}

\begin{remark}
In terms of the cube model from Remark \ref{rem: cube model ss} the action of $U_a\mu_a$ corresponds to the map
$$
U_a: H_*\left(\mathfrak{A}^{-}(L,s-e_A)\right)\to H_*\left(\mathfrak{A}^{-}(L,s-e_A-e_a)\right),\ a\notin A
$$
and similarly for $U_b\mu_b$.
\end{remark}

\subsection{Reversing orientation}

We will often need to reverse orientation on some of the link components. The following result describes some useful facts about this operation. It is a variant of  \cite[Proposition 5.4]{MOST}, see also \cite{BaldwinGrigsby}.

\begin{proposition}
\label{prop: reversing orientation}
Consider the link $L'$ obtained from $L$ by reversing orientation on the $i$-th component $L_i$. We have the following:

a) The chain complexes $\CFL(L)$ and $\CFL(L')$ can be chosen to be isomorphic. The isomorphism reverses the $i$-th Alexander grading $A_i$ and preserves $A_j$ for $j\neq i$. The Maslov grading is changed by 
$$
M\mapsto M-2A_i+\ell_i.
$$

b) The actions of $R$ on $\CFL(L)$ and $\CFL(L')$ differ by an automorphism of $R$ which swaps $U_i$ and $V_i$.

c) As a consequence, we have isomorphisms
$$
\HFL(L)\simeq \HFL(L'),\ \HFLhat(L)\simeq \HFLhat(L')
$$
which change the gradings as in (a).

d) If $L$ is an L-space link then $L'$ is also an L-space link. If, moreover, they are plumbed, then the isomorphism  in (a) preserves the level grading on $\CFL$. 
\end{proposition}

\begin{corollary}
\label{cor: change orientation many components}
Consider the link $L'$ obtained from $L$ by reversing orientation on components indexed by a subset $A\subset \{1,\ldots,n\}.$ Then the isomorphism between $\CFL(L)$ and $\CFL(L')$ shifts the Maslov grading by 
$$
-2\sum_{i\in A}A_i+\sum_{i\in A,j\in \overline{A}}\lk(L_i,L_j).
$$
Here the linking numbers are computed in the original link $L$.
\end{corollary}

\begin{proof}
We can reindex the components so that $A=\{1,\ldots,k\}$ and prove the statement by induction on $k$. 
Let $L^{(k)}$ be the link with the orientation on the first $k$ components reversed, denote its components by $L^{(k)}_i$.

The base case $k=1$ follows from Proposition \ref{prop: reversing orientation}. For the step of the induction, assume that we already changed the orientation on the first $k$ strands to get the shift
$$
S_k:=-2\sum_{i=1}^{k}A_i+\sum_{i\le k,j>k} \lk(L_i,L_j).
$$
Now we change the orientation on the $(k+1)$-st strand and get the overall shift
$$
S_{k+1}=S_k-2A_{k+1}+\sum_{i=1}^{k}\lk\left(L_i^{(k)},L_{k+1}^{(k)}\right)+\sum_{j=k+2}^{n}\lk\left(L_j^{(k)},L_{k+1}^{(k)}\right)=
$$
$$
S_k-2A_{k+1}-\sum_{i=1}^{k}\lk\left(L_i,L_{k+1}\right)+\sum_{j=k+2}^{n}\lk\left(L_j,L_{k+1}\right)=
$$
$$
-2\sum_{i=1}^{k+1}A_i+\sum_{i\le k+1,j>k+1} \lk(L_i,L_j).
$$
Note that since the orientation of $L_i^{(k)}$ is reversed for $i\le k$, the sign of the linking number $\lk\left(L_i^{(k)},L_{k+1}^{(k)}\right)$ is reversed as well.
\end{proof}

\section{Some cobordism maps}
\label{sec: cobordisms}

%\EG{Probably can borrow some details from Juhasz-Marengon}

%\EG{Introduce notations $K(n),K(n+2)$ with specific orientations}
In this section, we recall some elementary cobordisms and establish the properties necessary for subsequent computations.

All the links below are in $S^3$ and cobordisms in $S^3\times [0,1]$. We will be using the language of decorated link cobordisms, see \cite{Zemke, Zemke2} for definitions and details.

\subsection{Conventions} Given a link $L$, we assume every component of $L$ has exactly one $z$ and one $w$  basepoint unless stated otherwise. Moreover, if for a decorated link cobordism $(W,\Sigma)$ we have $W\cong S^3\times [0,1]$ then we will drop $W$ from the notation and denote the cobordism by $\Sigma$. 

Given a decorated cobordism $(W,\Sigma)$ from $L_1$ to $L_2$ there is a universal coloring (corresponding to associating distinct colors to the components of dividing curves on $\Sigma$) and we denote the cobordism map associated to this coloring by  $f_{(W,\Sigma)}$ or $f_{\Sigma}$ if $W=S^3\times [0,1]$. Moreover, we denote the induced maps on either $\HFL$ or $\HFLhat$ by $F_{(W,\Sigma)}$ or $F_{\Sigma}$.

%\EG{Maybe add a reference to Zemke on decorated cobordisms and colorings}

In the literature on Heegaard Floer homology for $3$-manifolds with multiple basepoints, two distinct conventions for the Maslov grading have been used. Depending on the convention, the top-grading generator in  $\widehat{\mathrm{HF}}(S^3, \w)$ has a grading of either $0$ or $\frac{1}{2}(|\w|-1)$. In order to distinguish these two grading conventions, we denote the former grading by $M$, and the latter grading by $\gr_{\w}$. Note that 
$$\gr_{\w}(x)=M(x)+\dfrac{1}{2}(|\w|-1).$$

%In this paper, we use the first convention and assume 
%$$\widehat{HF}(S^3, {\bf{w}})\cong \bigotimes_{i=1}^{|\bf{w}|-1} ((\F)_{-1}\oplus (\F)_{0})$$

%In Heegaard Floer homology for $3$-manifolds with multi basepoint, two different conventions for Maslov grading has been used in the literature. The top degree generator in $\widehat{HF}(S^3, {\bf{w}})$ has grading $0$ in one convention, and grading $\frac{1}{2}(|\bf{w}|-1)$ in the other. 

 %We use the following grading convention in the paper to require that the relatively graded $\F$-module 
%$$\widehat{HF}(S^3, {\bf{w}})\cong \bigotimes_{i=1}^{|\bf{w}|-1} ((\F)_{-1}\oplus (\F)_{0})$$
%have top degree generator in grading $0$, regardless of how many basepoints are in $\bf{w}$. Another grading convention is by setting the Maslov grading of the top element of $\widehat{HF}(S^3, \bf{w})$ to be $\dfrac{1}{2}(|{\bf{w}}|-1)$. In order to distinguish these two grading convention, we denote the former grading by $M$, referring to it as Maslov grading, and the latter grading by $\gr_w$. Note that  
%$$\gr_w(v)=M(v)+\dfrac{1}{2}(|\bf{w}|-1).$$
By \cite[Theorem 1.4]{Zemke2}, for a decorated cobordism $(W,\Sigma)$ from $L_1$ to $L_2$ the cobordism map $F$  shifts $\gr_{\w}$ by:
$$\gr_{\w}(F(x))-\gr_{\w}(x)=\dfrac{c_1^{2}(\mathfrak{s})-2\chi(W)-3\sigma(W)}{4}+\chi(\Sigma_{\w})-\dfrac{1}{2}(|\w_1|+|\w_{2}|)$$
and so it shifts $M$ by 
$$M(F(x))-M(x)=\dfrac{c_1^{2}(\mathfrak{s})-2\chi(W)-3\sigma(W)}{4}+\chi(\Sigma_{\w})-|\w_2|$$
%where $F: \HFLhat(S^3, L_1, {\bf{w_{1}}}, {\bf{z_{1})}}\rightarrow \HFLhat(S^3, L_2, {\bf{w_2}}, {\bf{z_2}})$. 
Here, $\Sigma_{\w}\subset \Sigma$ denotes the union of the components of the dividing curve complement that contain the $\w$ basepoint on their boundaries. In this paper, we will use the first convention, i.e. $M$ for Maslov grading.

\subsection{Birth and death cobordisms} 
Given a link $L$, the \emph{birth cobordism} is the decorated cobordism $B$ from $L$ to $L\sqcup O$ where $B\subset S^3\times [0,1]$ is an unknotted disk bounding the unknot $O$ and decorated with one dividing arc. Similarly, the \emph{death cobordism} is the decorated cobordism $D$ from $L\sqcup O$ to $L$ where $D\subset S^3\times [0,1]$ is the unknotted disk bounding $O$ and decorated with one dividing arc.  

Assume that a link $L$ has $n$ components, then the link $L\sqcup O$ has $(n+1)$ components, and we denote the variables corresponding to the basepoints on $O$ by $U_{n+1}$ and $V_{n+1}$. 
Recall that  
\[\CFL(L\sqcup O)\simeq \CFL(L)\otimes_{\F} \left(B\xrightarrow{U_{n+1}V_{n+1}-U_nV_n}T\right)\]
where $M(T)=0,\ M(B)=-1$ and 
\[A_i(x\otimes T)=A_{i}(x\otimes B)=A_i(x)\ \ \text{for}\ \ 1\le i\le n\]
while $A_{n+1}(x\otimes T)=A_{n+1}(x\otimes B)=0$ for all $x\in \CFL(L)$. Thus, 
$$\HFLhat(L\sqcup O)\cong\HFLhat(L)\otimes_{\F} V$$
where $V=\langle B, T\rangle$.
%is a free 2-dimensional module with 
%$$M(T)=0,\ M(B)=-1, \quad A(T)=A(B)=0.$$
Then, under this identification, birth and death cobordism maps 
$$f_B: \CFL(L)\otimes_{\F} \F[U_{n+1},V_{n+1}]\rightarrow \CFL(L\sqcup O), \quad \quad f_D: \CFL(L\sqcup O)\rightarrow \CFL(L)\otimes_{\F} \F[U_{n+1},V_{n+1}]$$
are given by 
 $f_B(x)=x\otimes T$, and $f_D(x\otimes T)=0, f_D(x\otimes B)=x$. 

\begin{remark}
The above model for $\CFL(L\sqcup O)$ depends on the choice of a component of the link $L$, in this case we used $L_n$. However, for $1\le i\le n$ the action of $U_iV_i$ on $\CFL(L)$ is homotopic to $U_nV_n$, so all such choices yield homotopy equivalent complexes for $\CFL(L\sqcup O)$.

\end{remark}

\subsection{Quasi-(de)stablization cobordisms} Given a  multi-based link $(L,\w,\z)$, let $(L,\w\cup \{w'\},\z\cup \{z'\})$ be the multi-based link obtained by stabilizing $L$ i.e. adding an extra pair of basepoints $\{w',z'\}$ in the same connected component of $L\setminus (\w\cup\z)$. Suppose the endpoints of the component containing $\{w',z'\}$ are $z_i$ and $w_i$ and for simplicity assume that they are on the $i$-th component of $L$, denoted by $L_i$. Then, by \cite{ZemkeBasept,OSLinks}   we have
\[\CFL(L,\w\cup \{w'\},\z\cup \{z'\})\simeq \CFL(L)\otimes \left( \begin{tikzcd}
\theta \arrow[r, bend left, "V'-V_i"] & \xi \arrow[l, bend left, "U'-U_i"]
\end{tikzcd}\right) \]
as curved complexes. Here, $U'$, $V'$, $U_i$ and $V_i$ denote the variables for $w'$, $z'$, $w_i$ and $z_i$, respectively. In particular, if we consider $\HFL$ with a coefficient ring in which $U'=U_i$ and $V'=V_i$ then 
\[\HFL(L,\w\cup \{w'\},\z\cup \{z'\})\cong \HFL(L)\otimes \F\langle \theta,\xi\rangle\]
where 
   $$M(\theta)=0, \quad M(\xi)=-1, \quad A(\theta)=1/2, \quad A(\xi)=-1/2. $$
   
Quasi-stabilization cobordisms from $(L,\w,\z)$ to  $(L,\w\cup \{w'\},\z\cup \{z'\})$ are the cobordisms obtained from $(S^3\times I, L\times I)$ by decorating the $i$-th annulus (i.e. from $L_i$ to $L_i$) as in $S^+$ or $T^+$ in Figure \ref{fig:quasi-stabilization}. Similarly, decorations $S^-$ and $T^-$ give the quasi-destabilization cobordisms from  $(L,\w\cup \{w'\},\z\cup \{z'\})$ to $(L,\w,\z)$. These cobordisms induce the quasi-(de)stabilization maps $S^+:=f_{S^+}, T^+:=f_{T^+}$ (resp. $S^-:=f_{S^-}$, $T^-:=f_{T^-}$) which are defined as follows:
$$S^+(x)=x\otimes \theta, \quad T^+(x)=x\otimes \xi$$
$$S^-(x\otimes \theta)=0, \quad S^-(x\otimes \xi)=x$$
$$T^-(x\otimes \theta)=x, \quad T^-(x\otimes \xi)=0.$$

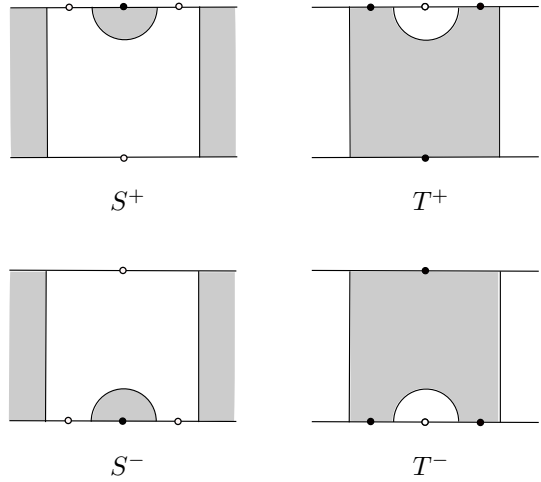
\begin{figure}[ht!]
%\def\svgwidth{8cm}
%\centering
 \begin{tikzpicture}
\node[anchor=south west,inner sep=0] at (4,0) {\includegraphics{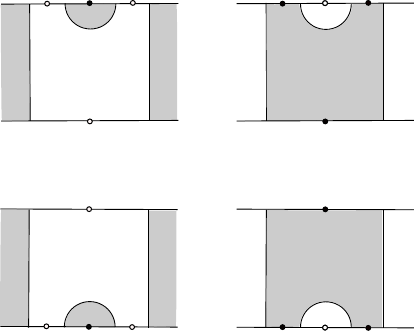}};
\node[label=right:$S^+$] at (5.1,3){};
\node[label=right:$S^-$] at (5.1,-0.5){};
\node[label=right:$T^+$] at (9.1,3){};
\node[label=right:$T^-$] at (9.1,-0.5){};
 \end{tikzpicture}
\caption{Top: Quasi-stabilization cobordisms, Bottom: Quasi-destabilization cobordisms}\label{fig:quasi-stabilization}
\end{figure} 
Consequently, they satisfy the following relations:
\begin{lemma}\label{lem:stab-relations}\cite[Lemma 14.13]{Zemke}
The quasi-(de)stabilization maps satisfy  the following relations:
\[S^-\circ S^+\simeq 0,\ \ T^-\circ T^+\simeq 0,\ \ T^-\circ S^+\simeq \mathrm{id},\ \ S^-\circ T^+\simeq\mathrm{id}\]
\end{lemma}

\begin{lemma}\cite[Lemma 7.1]{Zemke2} \label{lem:quasi-shitfs}
    The quasi-stabilization maps are graded and satisfy 
    $$A_i(S^+(x))-A_i(x)=\dfrac{1}{2}, \quad A_i(T^+(x))-A_i(x)=-\dfrac{1}{2}$$
    $$M(S^+(x))-M(x)=0, \quad M(T^+(x))-M(x)=-1.$$
  Moreover, these cobordisms do not change $A_j$ for $j\neq i$.    
  %  Here, $A$ is the component of the Alexander multi-grading corresponding to the link component where we add/remove $w'$ and $z'$.  
    
%    \BL{task: write the composition maps for $S^\pm, T^\pm$.}
\end{lemma}

\subsection{Pair of pants cobordisms.}\label{sec:PoPCob}
%In this subsection, we discuss cobordism maps induced by pair of pants but with different decorations. 
Let $K$ be an oriented knot and $L=K_m(k_+,k_-)$. Recall that $K_m(k_+,k_-)$ denotes the $n$-cable of $K$ ($n=k_++k_-$) with framing $m$, such that $k_+$ components are oriented positively and $k_-$ components are oriented negatively. In this subsection, we introduce multiple decorated pair of pants cobordisms from $L$ to $L'=K_m(k_++1,k_-+1)$ and discuss their properties.

There is a natural pair of pants cobordism from $L\sqcup O$ to $L'$ defined as follows. Note that $L'$ is obtained from $L$ by adding two components with opposite orientation, denote the newly added components by $L'_{n+1}$ and $L'_{n+2}$. These components co-bound an annulus $\mathcal{A}$. The pair of pants cobordism from $L\sqcup O$ to $L'$ is the disjoint union of product cobordisms from $L_j$ to $L'_j$ for $1\le j\le n$, and the pair of pants cobordism from $O$ to $L'_{n+1}\sqcup L'_{n+2}$, obtained by removing a disk from $\mathcal{A}$.

The product components will be decorated as product cobordisms, while we equip the pair of pants component with one of the decorations in Figure \ref{fig:pair-of-pants}. Note that in $P_z$ and $P_w$ the unknot has two pairs of $w$ and $z$ basepoints and these cobordisms correspond to band attachments to the unknot over $z$ and $w$ basepoints, respectively.

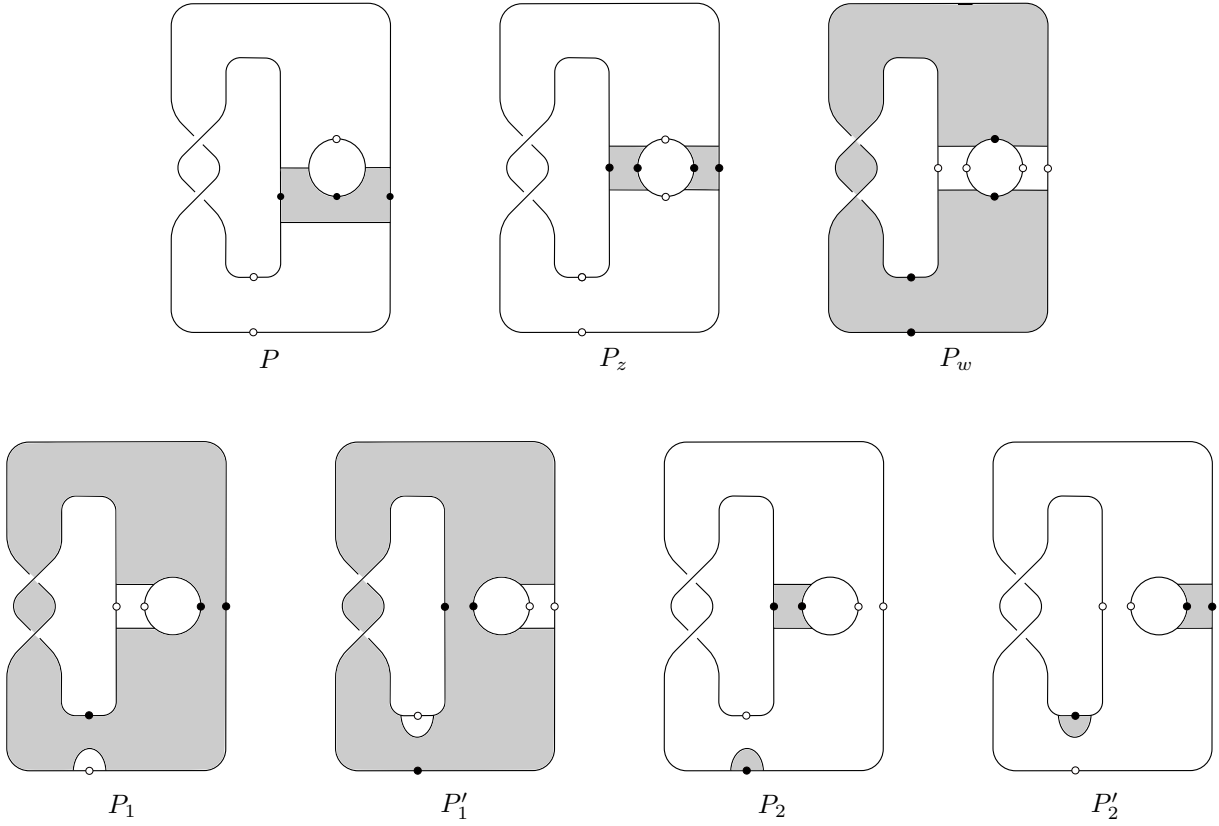
\begin{figure}[ht!]
    \centering
     \begin{tikzpicture}
\node[anchor=south west,inner sep=0] at (0,0)
    {\includegraphics[width=1\linewidth]{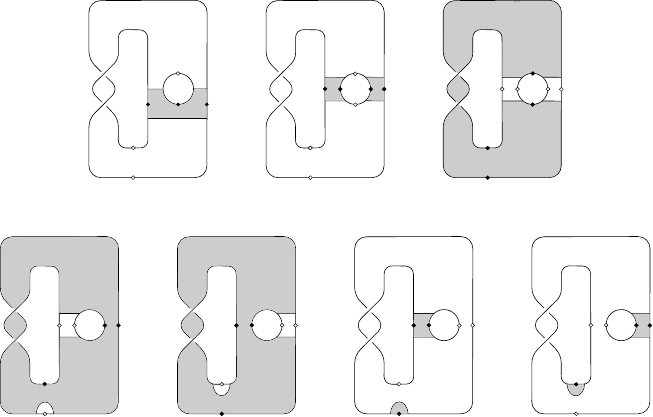}};
\node[label=right:$P$] at (3.1,5.5){};
\node[label=right:$P_z$] at (7.6,5.5){};
\node[label=right:$P_w$] at (12.1,5.5){};
\node[label=right:$P_1$] at (1.1,-0.4){};
\node[label=right:$P_1'$] at (5.5,-0.4){};
\node[label=right:$P_2$] at (9.7,-0.4){};
\node[label=right:$P_2'$] at (14.1,-0.4){};

     \end{tikzpicture}

    \caption{Different decorations on the pair of pants cobordism}
    \label{fig:pair-of-pants}
    
\end{figure}

%We denote the cobordism maps induced by $C=P,\ P_z,\ P_w,P_1,P_1',P_2, P_2'$ by $F_{*}$. 
%{\AA{there is a choice involved with $\Phi_i$ maps. Suggested change of notation: how about we use $P_{w1}$, $P_{w2}$ and $P_{z1}$, $P_{z2}$ instead?} 

\begin{lemma}\label{lem:def-Phi-i}
For $i=1,2$, we have $f_{P_i}\circ f_{B}\simeq f_{P_i'}\circ f_B$. 
\end{lemma}
\begin{proof}
As depicted in Figure \ref{fig:varphi-1}, cobordisms $P_1\circ B$ and $P_1'\circ B$ are the same, and so the cobordism maps are chain homotopic. The argument for $P_2$ and $P_2'$ is similar. 
\end{proof}

\begin{figure}[ht!]
    \centering
     \begin{tikzpicture}
\node[anchor=south west,inner sep=0] at (0,0)
    {\includegraphics[width=0.5\linewidth]{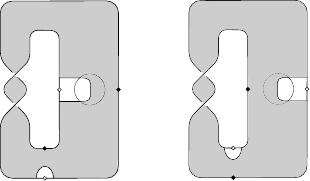}};
\node[label=right:{$=$}] at (3.6,2.5){};
\end{tikzpicture}
\caption{Left: Cobordism $P_1\circ B$, Right: Cobordism $P_1'\circ B$}
    \label{fig:varphi-1}
    \end{figure}

\begin{notation}
We denote the cobordism map $f_{P_i}\circ f_{B}$ by $\varphi_i$, and the map it induces on $\HFL$ or $\HFLhat$ by $\Phi_i$.
\end{notation}

\begin{lemma}\label{lem:map-relations}
The following relations hold between the aforementioned pair of pants cobordism maps:
\begin{enumerate}
\item $f_P\simeq f_{P_z}\circ T^+\simeq f_{P_w}\circ S^+$.
\item $f_{P_1}\simeq f_{P_w}\circ_1 T^+$, $f_{P_1'}\simeq f_{P_w}\circ_2 T^+$, $f_{P_2}\simeq f_{P_z}\circ_1 S^+$ and $f_{P_2'}\simeq f_{P_z}\circ_2 S^+$.

%\item $\Phi_1\simeq F_{P_w}\circ T^+$, $\Phi_2\simeq F_{P_z}\circ S^+$.
%\item $F_{P_1}\simeq(\Psi_z)_1\circ F_P$, $F_{P_1'}\simeq(\Psi_z)_2\circ F_P$, $F_{P_2}\simeq(\Phi_w)_1\circ F_P$ and $F_{P_2'}\simeq(\Phi_w)_2\circ F_P$.

%\item $\Phi_1\simeq(\Psi_z)_1\circ F_P$, $\Phi_1'\simeq(\Psi_z)_2\circ F_P$, $\Phi_2\simeq(\Phi_w)_1\circ F_P$ and $\Phi_2'\simeq(\Phi_w)_2\circ F_P$.
\end{enumerate}
In (2) subscripts $1$ and $2$ refer to two ways of composing $f_{P_w}$ (resp. $f_{P_z}$) with $T^+$ (resp. $S^+$).
%Similarly, in (3) The subscripts $1$ and $2$ refer to adding $\Psi_z$ or $\Phi_w$ on the components $K(n+1)_{n+1}$ and $K(n+1)_{n+2}$, respectively. 
%In number (2) similar relations hold of $\Phi_1'$ and $\Phi_2'$ as well.
\end{lemma}
%\BL{Finish reading and tailor in the end.}

\begin{proof}
See Figure \ref{fig:maps-relations} for a pictorial proof that $P$ is the composition $P_z\circ T^+$, while  $P_1$ and $P_1'$ are obtained from composing $P_w$ with $T^+$ in two different ways. Then, half of the relations follow from the functoriality of link Floer homology. The proof of the remaining relations is similar.
\begin{figure}[ht!]
    \centering
    \includegraphics[width=0.7\linewidth]{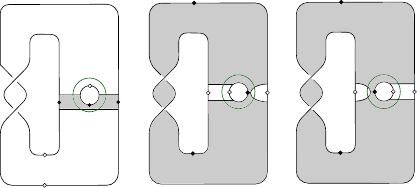}
    \caption{Inside the green circle in all three pictures is the cobordism $T^+$, while the outside cobordism is ${P_z}$ in the left and $P_{w}$ in the middle and right.}
    \label{fig:maps-relations}
\end{figure}
\end{proof}

\begin{lemma}\label{lem:PoP-grading} These pair of pants cobordism maps shift the Maslov and Alexander multi-gradings as in the following table. 

\begin{table}[ht!]
\centering
\begin{tabular}{|c|c|c|c|}
\hline
 & Maslov & $A_{i}$&$A_{n+1}$ \\ \hline
 $f_P$  & -1 &0&  0 \\ \hline
$f_{P_z}$ & 0&0 & 1/2 \\ \hline
  $f_{P_w}$ & -1&0 & -1/2 \\ \hline
$f_{P_1}$ & -2 &0& -1 \\ \hline
$f_{P_2}$ & 0 &0& 1 \\ \hline
\end{tabular}
\caption{Grading shifts}
\end{table}

Here, $1\le i\le n$, and the $A_{n+1}$-column records the shift $(A_{n+1}+A_{n+2})(f_{C}(x))-A_{n+1}(x)$ for $C$ equal to any of the given decorated pairs of pants. Furthermore, Maslov and Alexander multi-grading shifts for $f_{P_1'}$ and $f_{P_2'}$ are the same as $f_{P_1}$ and $f_{P_2}$, respectively.
\end{lemma}

%It is straightforward to check that we $P$ can be presented as the composition of a quasi-stabilization cobordism with $P_z$ or $P_w$ and so 
%\[F_P=F_{P_w}\circ S^+=F_{P_z}\circ S^+.\]

%\begin{lemma} \label{lem:PoP-grading}
%These maps shift the Maslov and Alexander grading (for the component corresponding to the pair of pants) as in the following table: {\AA{Can we briefly check the Alexander multi-grading?}}
% \begin{table}[h!]
% \centering
% \begin{tabular}{|c|c|c|}
% \hline
%  & Maslov & Alexander \\ \hline
%  $F_{P_z}$ & 0 & 1/2 \\ \hline
%   $F_{P_w}$ & -1/2 & -1 \\ \hline
% $F_P$  & -1 &  0 \\ \hline
% $\Phi_1$ & -2 & -1 \\ \hline
% $\Phi_2$ & 0 & 1 \\ \hline
% \end{tabular}
% \caption{Grading shifts}
% \end{table}
% \end{lemma}
\begin{proof}
The grading shifts for $f_{P_z}$ and $f_{P_w}$ are by \cite[Lemma 7.2]{Zemke2}. Consequently, the grading shifts for $f_P$, $f_{P_1}$ and $f_{P_2}$ can be computed from Lemmas \ref{lem:map-relations} and \ref{lem:quasi-shitfs}.
\end{proof}

Finally, let $-P$ denote the reverse of the cobordism $P$ (i.e. reverse the orientation on $S^3\times I$ and $\Sigma$) from $L'$ to $L\sqcup O$.

\begin{lemma}\label{lem:relations}
The cobordism map $f_{-P}$ satisfies the following relations:
%\[\Phi\circ\Phi_1\simeq U_{n+1}\ \ \ \Phi\circ \Phi_2\simeq V_{n+1},\ \ \Phi\circ F_P\simeq 0 \]
\[f_{-P}\circ f_{P_1}\simeq U_{n+1}\mathrm{id}\ \ \ f_{-P}\circ f_{P_2}\simeq V_{n+1}\mathrm{id},\ \ f_{-P}\circ f_P\simeq 0\]
%\AA{AA:I want to discuss this. The second line is the new version that is definitely true}
\end{lemma}

\begin{proof}
Note that the composition of $f_{P_w}$ and its reverse (resp. $f_{P_z}$ and its reverse) are chain homotopic to multiplication by $U_{n+1}$ (resp. $V_{n+1}$) \cite[Lemma 2.1]{AE} while $S^-T^+\simeq \mathrm{id}$, $T^-S^+\simeq \mathrm{id}$ and $S^-S^+\simeq 0$, by Lemma \ref{lem:stab-relations}. So the claim follows from these properties and Lemma \ref{lem:map-relations}.
\end{proof}

\subsection{X-shape cobordism}

Let $\gamma$ be a (framed) arc connecting $L_i$ and $L_j$ such that $\gamma\cap L=\{z_i,z_j\}$. Corresponding to $\gamma$ we define a cobordism $H_{\gamma}$ from $L$ to $L$ by first adding a band along $\gamma$ connecting $L_i$ to $L_j$ to get a component denoted by $L_{ij}$ with four basepoints. Next, compose it with $T^+T^-$ and then add a dual band to split $L_{ij}$ and get $L$, as in Figure \ref{fig:X-shape}. The cobordism $H_{\gamma}$ induces a map 
\[f_{H_\gamma}:\CFL\left(L;R/(U_i-U_j,V_i-V_j)\right)\to \CFL\left(L;R/(U_i-U_j,V_i-V_j)\right).\]
Abusing notation, we denote the induced maps on $\HFL(L;R/(U_i-U_j,V_i-V_j))$ and $\HFLhat(L)$ by $h_{\gamma}$.

\begin{figure}[ht!]
\includegraphics[width=0.6\linewidth]{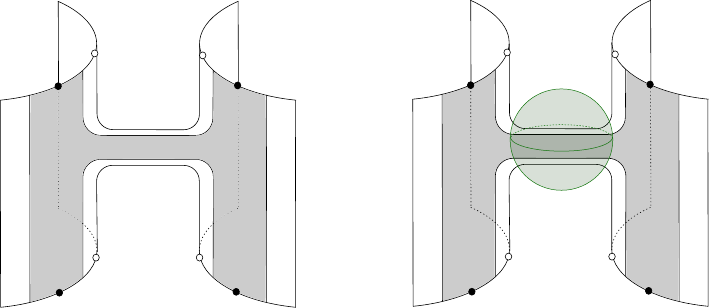}
\caption{Left: local picture of an X-shape cobordism, Right: splitting along the green sphere results in an annulus cobordism with horizontal dividing arcs from $\emptyset$ to $O(2)$, inside the green ball, composed with a cobordism from $(S^3,L)\sqcup (S^3,O(2))$ to $(S^3,L)$, outside the green ball. }\label{fig:X-shape}
\end{figure} 

%\begin{lemma}
%The cobordism map $h_{\gamma}$ is equal to the action of $\xi_i+\eta_j\in \HFLhat(O(n))$ on $\HFLhat(L)$.
%\end{lemma}
%\begin{proof}
%We can decompose $H_\gamma$ as $(W_i\sqcup W_j)\circ A$ where the cobordism $A$ from the empty set to $O(n)_i\cup O(n)_j$ is an annulus along $\gamma$ with two parallel dividing arcs, see Figure \ref{fig:cross-cob}. Then, $f_{A}(1)=\xi_i+\eta_j$ and so $h_{\gamma}$ is equal to the action of $\xi_i+\eta_j$ on $\HFLhat(L)$.
%\end{proof}
%Consequently, $h_{\gamma}$ is independent of the choice of $\gamma$ so we denote it by $h_{ij}$ or simply $h$ if the indices are clear from the context.

% \begin{lemma} 
% There is a chain homotopy $\widetilde{h}_{ij}$ between $U_iV_i$ and $U_jV_j$ in $\CFL(L)$ such that the induced maps by $\widetilde{h}_{ij}$ on $\HFL(L;R/(U_i-U_j,V_i-V_j))$ and $\HFLhat(L)$ are $h_{ij}$. (Note that setting $U_i=U_j$ and $V_i=V_j$ or $U_i=U_j=V_i=V_j=0$ will turn $\widetilde{h}_{ij}$ into a chain map.)
% \end{lemma}

\begin{example}\label{ex:simple-X}
Let $L'=L\sqcup O$ be a link obtained by putting an unknot with one pair of basepoints next to the $i$-th component of $L$ and consider a local  unknotted arc $\gamma$ connecting the $z$ basepoints on $L_i$ and $O$. Then, there is an isomorphism
\begin{equation}\label{eq:add unknot}
\HFL(L';R_{n+1}/(U_i-U_{n+1},V_i-V_{n+1}))\cong \HFL(L)\otimes \F\langle T,B \rangle.
\end{equation}
and under this isomorphism $h_{\gamma}(x\otimes T)=x\otimes B$ and $h_{\gamma}(x\otimes B)=0$. The same statement holds for $h_{\gamma}$ on $\HFLhat(L')$.

To show this, we use the decomposition of this cobordism as $P'\circ (-P')$ where $P'$ is a pair of pants cobordism from $L$ to $L\sqcup O$ with a similar decoration as $P$ in Figure \ref{fig:pair-of-pants}. Let 
\[f_1:\HFL(L)\otimes \F\langle T,B \rangle \to \HFL(L)\ \ \text{and}\ \ f_2:\HFL(L)\to \HFL(L)\otimes \F\langle T,B \rangle\]
be the cobordism maps induced by $-P'$ and $P'$. First, composing $-P'$ with the birth cobordism is the identity cobordism, so $f_1(a\otimes T)=a$. Second, we show that $f_1(a\otimes B)=0$. Let $(W,\mathcal{F}_1)$ be the connected sum cobordism (\cite[Proposition 5.1]{Zemke2019}) from $(S^3,L)\sqcup (S^3,O)$ to $(S^3,L)$ where $W$ from $S^3\sqcup S^3$ to $S^3$ is defined by attaching a $4$-dimensional $1$-handle, and $\mathcal{F}_1$ is a disjoint union of product cylinders with a pair of pants corresponding to taking the connected sum of $L_i$ and $O$ by a band that goes over the $1$-handle, and decorated analogous to $-P'$. Similarly, let $(W,\mathcal{F}_2)$ be the connected sum cobordism from $(S^3,L)\sqcup (S^3,O(2))$ to $(S^3,L\sqcup O)$ corresponding to taking the connected sum of $L_i$ and the second component of $O(2)$. Then, $(W,\mathcal{F}_1)\circ (-P'_O)=-P'\circ (W,\mathcal{F}_2)$, where $-P'_O$ is the disjoint union of the product cobordism $(S^3\times [0,1],L\times[0,1])$ with the pair of pants cobordism in $S^3\times [0,1]$ from $O(2)$ to $O$, as in Figure \ref{fig:sum-PoP} .

\begin{figure}[ht!]
\includegraphics[width=0.9\linewidth]{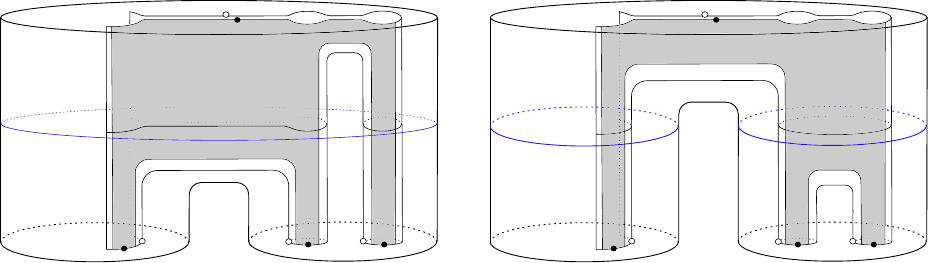}
\caption{Left: Cobordism $-P'\circ (W,\mathcal{F}_2)$, Right: Cobordism $(W,\mathcal{F}_1)\circ (-P'_O)$ }\label{fig:sum-PoP}
\end{figure} 

By the connected sum formula \cite[Proposition 5.1]{Zemke2019}, 
\[F_{(W,\mathcal{F}_2)}:\CFL(L)\otimes_{\F} \CFL(O(2))/(U_i-U_{n+2},V_i-V_{n+2})\to \CFL(L\sqcup O)\]
is a chain homotopy equivalence, and after identifying $U_{n+1}=U_{n+2}$ and $V_{n+1}=V_{n+2}$ it gives the identification \eqref{eq:add unknot} on homology. Similarly, after identifying $U_{i}=U_{n+1}$ and $V_{i}=V_{n+1}$, the map $F_{(W,\mathcal{F}_1)}$ is identity. Furthermore, 
\[F_{-P'_O}:\F[U,V]\langle T,B\rangle \to \F[U,V]\]
preserves Maslov and total Alexander grading, and its composition with the birth cobordism is identity. So, $F_{-P'_O}(T)=1$, while $F_{-P'_O}(B)$ has Maslov grading $-1$. Since every element of $\F[U,V]$ has even Maslov grading, $F_{-P'_O}(B)=0$. Consequently, functoriality of link Floer homology implies that $f_{1}(a\otimes B)=0$.

Finally, we need to show that $f_2(a)=a\otimes B$. Composing $P'$ with the death cobordism is the identity cobordism. So, $f_2(a)=a\otimes B+ \textit{other terms}$ where other terms are in $\HFL(L)\otimes \F\langle T\rangle$. A similar argument as in Lemma \ref{lem:relations} we have $f_1\circ f_2=0$. So, $f_1(f_2(a))=0$ and consequently $f_2(a)=a\otimes B$.

\end{example}

\begin{comment}

To show this, we use the decomposition of this cobordism as $P'\circ (-P')$ where $P'$ is a pair of pants cobordism from $L$ to $L\sqcup O$ with a similar decoration as $P$ in Figure \ref{fig:pair-of-pants}. Let 
\[f_1:\HFL(L)\otimes \F\langle T,B \rangle \to \HFL(L)\ \ \text{and}\ \ f_2:\HFL(L)\to \HFL(L)\otimes \F\langle T,B \rangle\]
be the cobordism maps induced by $-P'$ and $P'$. Note that composing $-P'$ with the birth cobordism is the identity cobordism, so $f_1(a\otimes T)=a$. Similarly, composing $P'$ with the death cobordism is also the identity cobordism. So, $f_2(a)=a\otimes B+ \textit{other terms}$ where other terms are in $\HFL(L)\otimes \F\langle T\rangle$. A similar argument as in Lemma \ref{lem:relations} we have $f_1\circ f_2=0$. So, $f_1(f_2(a))=0$ and consequently $f_2(a)=a\otimes B$.

Moreover, on $\HFLhat(L')$ we have $h_{\gamma}(x\otimes B)=0$. That is because the cobordism $H_{\gamma}\circ H_{\gamma}\circ B$ from $L$ to $L\sqcup O$ is the same as $H_{\gamma}\circ B_1$ as depicted in Figure \ref{fig:HHBirth}. However, it follows from \cite[Figure 17]{JM} that the induced map by $B_1$ i.e. $F_{B_1}=0$. Therefore, 
\[h_{\gamma}h_{\gamma}f_B(x)=h_{\gamma}(x\otimes B)=0.\]

\textcolor{red}{do we need to know that $f_1(a\otimes B)=0?$}

\begin{figure}[ht!]
%\def\svgwidth{8cm}
\includegraphics[width=0.7\linewidth]{figs/HHBirth.pdf}
\caption{Left: Cobordism $H_{\gamma}\circ H_{\gamma}\circ B$, Right: Cobordism $H_{\gamma}\circ B_1$ }\label{fig:HHBirth}
\end{figure} 
\end{comment}

\begin{definition}\label{def:dot map}
We call the X-shape cobordism of Example \ref{ex:simple-X} the \emph{dot} cobordism, and its induced map on $\HFLhat$ the \emph{dot map}, denoted by $\mathcal{B}$. Note that 
\[\mathcal{B}:\HFLhat(L\sqcup O)\to \HFLhat(L\sqcup O)\]
and $\mathcal{B}(x\otimes T)=x\otimes B$, $\mathcal{B}(x\otimes B)=0$.
\end{definition}

\begin{lemma}
\label{lem: UV homotopy arc}
For any framed arc $\gamma$ connecting $L_i$ and $L_j$, there exists a chain homotopy $\widetilde{h}_{ij}$ between $U_iV_i$ and $U_jV_j$ in $\CFL(L)$ such that the induced maps by $\widetilde{h}_{ij}$ on $\HFL(L;R/(U_i-U_j,V_i-V_j))$ and $\HFLhat(L)$ are equal to $h_{\gamma}$. (Note that setting $U_i=U_j$ and $V_i=V_j$ or $U_i=U_j=V_i=V_j=0$ will turn $\widetilde{h}_{ij}$ into a chain map.)
\end{lemma}

\begin{proof}
As depicted in Figure \ref{fig:X-shape}, the cobordism $H_{\gamma}$ is the composition of two cobordisms, $\mathcal{C}_1$ from the empty set to the $2$-component unlink $O(2)$ inside the green ball, and $\mathcal{C}_2$ from $(S^3,L)\sqcup (S^3,O(2))$ to $(S^3,L)$ outside of the green ball.
Recall that 
\[\CFL(O(2))=\left(B\xrightarrow{U_2V_2-U_1V_1} T\right),
\]
so
\[
\CFL\left(O(2);R_2/(U_2-U_1,V_2-V_1)\right)\cong R_2/(U_2-U_1,V_2-V_1)\langle T,B\rangle\] A similar grading argument as in the proof of \cite[Theorem 5.21]{JM} implies that $\mathcal{C}_1$ induces the chain map $F_{\mathcal{C}_1}(1)=B$ from $\F$ to $\CFL\left(O(2);R_2/(U_2-U_1,V_2-V_1)\right)$. Before identifying the variables, the map $F'_{\mathcal{C}_1}(1)=B$ from $\F$ to $\CFL(O(2))$, is not a chain map; however, it is the composition of the chain map $1\mapsto T$ from $\F$ to $\CFL(O(2))$ with the chain homotopy 
\[h:\CFL(O(2))\to \CFL(O(2)),\ \ h(T)=B,\ h(B)=0\]
between $U_1V_1$ and $U_2V_2$. Moreover, the cobordism $\mathcal{C}_2$ induces a chain map 
\[\CFL(L)\otimes_{R_2} \CFL(O(2))\to \CFL(L)\]
where $R_n$ is an $R_2$-module with the homomorphism that sends $U_1, V_1$ and $U_2,V_2$ to $U_i,V_i$ and $U_j,V_j$, respectively. Consequently, $F_{\mathcal{C}_2}\circ F'_{\mathcal{C}_1}$ gives a chain homotopy between $U_iV_i$ and $U_jV_j$ in $\CFL(L)$ that induces the chain map $F_{H_{\gamma}}$ after identifying $U_i=U_j$ and $V_i=V_j$. A similar argument works for the hat version as well. 
\end{proof}

\begin{lemma}\label{lem:h-PP-1}
Consider the pair of pants cobordisms $P$, $P_1$ and $P_2$ from $L=K_m(k_+,k_-)\sqcup O$ to $L'=K_{m}(k_++1,k_-+1)$. We have
\begin{enumerate}
\item $f_P\circ f_{-P}\simeq f_{H_{\gamma}}$.
\item $U_{n+1}f_P\simeq f_{H_{\gamma}}\circ f_{P_1}$
\item $V_{n+1}f_P\simeq f_{H_{\gamma}}\circ f_{P_2}$
\item $f_P\circ f_{H_{\gamma'}}\simeq f_{H_{\gamma''}}\circ f_P$ 
\end{enumerate}
where in (1), (2) and (3) $\gamma$ is a small arc on $P$ connecting $L'_{n+1}$ and $L'_{n+2}$. In (4), $\gamma'$ and $\gamma''$ are small local unknotted arcs connecting $L_n$ to $O$ in $L$ (as in Example \ref{ex:simple-X}) and $L'_n$ to $L'_{n+1}$ in $L'$, respectively.  
\end{lemma}
\begin{proof}
Part (1): It is easy to see that the decorated cobordism $P\circ (-P)$ is $H_\gamma$ for a small arc $\gamma$ connecting $L'_{n+1}$ and $L'_{n+2}$.

Part (2) and Part (3) follow from Lemma \ref{lem:relations} in a completely formal way:
\[U_{n+1}f_p=f_PU_{n+1}\simeq f_p\circ f_{-P}\circ f_{P_1}\simeq f_{H{\gamma}}f_{P_1}.\]
 As for part (4), Figure \ref{fig:x-pair-of-pants} shows that by sliding the feet of the handle we have $P\circ H_{\gamma'}=H_{\gamma''}\circ P$. 
\begin{figure}[ht!]
\includegraphics[width=1\linewidth]{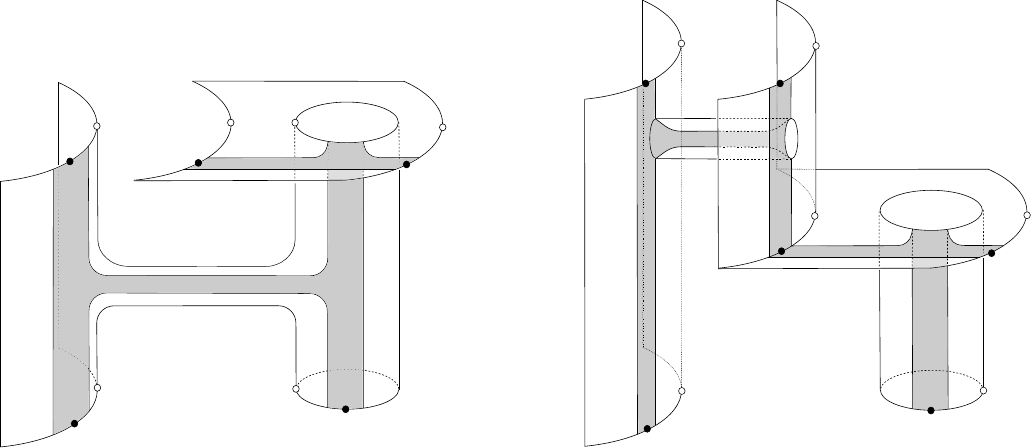}
\caption{Left: Cobordism $P\circ H_{\gamma'}$, Right: Cobordism $H_{\gamma''}\circ P$}\label{fig:x-pair-of-pants}
\end{figure} 
 
\end{proof}

Note that parts (2) and (3) of Lemma \ref{lem:h-PP-1} determine the map $f_P$ up to $(U_{n+1},V_{n+1})$-torsion in terms of $h=h_{\gamma}, f_{P_1}$ and $f_{P_2}$. In particular, if $L$ and $L'$ are  L-space links, Lemma \ref{lem:h-PP-1} determines the map $f_P$ completely.

\begin{corollary} The cobordism map $h_\gamma$ preserves Alexander grading and drops Maslov grading by $1$. 
\end{corollary}
\begin{proof}
Follows from part (1) of Lemma \ref{lem:h-PP-1} and Lemma \ref{lem:PoP-grading}.
\end{proof}

\begin{corollary}
Over $\HFLhat$ we have $F_P(-\otimes B)=h_{\gamma''}F_P(-\otimes T)$.
\end{corollary}
\begin{proof}
Follows from part (4) of Lemma \ref{lem:h-PP-1} and Example \ref{ex:simple-X}. 
\end{proof}

\section{Floer lasagna module}
\label{sec: lasagna}

We refer to \cite{Chen} for the general definition of the Floer lasagna module.  By 
\cite[Theorem 1.1]{Chen} we can  use ``cabled $\HFLhat$", defined below, as an alternative definition of the Floer lasagna module for 2-handlebodies.

Let $L=L_1\cup L_2\cup \cdots \cup L_r$ be a framed, oriented link in $S^3$. Fix a parametrization $\pi_i: S^1\times D^2 \rightarrow \text{nbd}(L_i)$ for each component $L_i$ of $L$. We will identify $S^1=\mathbb{R}/\Z$ and 
write the coordinates of some points on $S^1$ as real numbers modulo $\Z$.
Given any two $r$-tuples $k_+=(k^1_+, \cdots, k^r_+)$
 and $k_-=(k^1_-, \cdots, k^r_-)$ of positive integers, let $L(k_+, k_-)$ be the link obtained from $L$ by replacing each link component $L_i$ with $k^i_+$ many positively oriented and $k_-^i$ many negatively oriented parallel copies along the framing of $L_i$. That is, 
 $$L(k_+, k_-)=\bigcup_{i=1}^r \pi_i \left(S^1\times \left\{x^+_{i, 1}, \cdots, x^+_{i, k_{+}^{i}}, x^-_{i, 1}, \cdots, x^{-}_{i, k_{-}^{i}}\right\}\right)$$
 where $x^{-}_{i,  1}, \cdots, x^{-}_{i, k_{-}^{i}}, x^{+}_{i,1}, \cdots, x^{+}_{i, k_{+}^{i}}$ are distinct points in $D^2$. The link $L(k_+,k_{-})$ has 
$|k_{+}|+|k_{-}|=\sum_{i=1}^{r}(k_+^i+k_-^i)
 $
 components.
Denote the basepoints 
 $$w^\pm_{i,j}=\pi_i(0, x^{\pm}_{i,j}), \quad  z^{\pm}_{i,j}=\pi_i\left(\dfrac{1}{2}, x^{\pm}_{i,j}\right)$$
 on each component of $L(k_+,k_-)$. %\EG{What does $1/2$ mean here? Point $-1$ on $S^1$?} 
 Then we consider the link Floer homology of $(S^3,  L(k_+, k_-), \w, \z)$ denoted by $\HFLhat(L(k_+, k_-))$. Here, $\w=\cup_{i,j}w_{i,j}^{\pm}$ and $\z=\cup_{i,j}z_{i,j}^{\pm}$.

 \begin{definition}
 \label{def: cabled HF}
     The cabled link Floer homology of $L$ is defined as 
     $$\cHFL(L)=\bigoplus_{k_+, k_-\in \mathbb{N}^n} \HFLhat(L(k_+, k_-))\left[|k_+|+|k_-|\right]/ \sim$$
     where $\left[|k_+|+|k_-|\right]$ shifts the Maslov grading up by $|k_+|+|k_-|$. Here   $\sim$ denotes the equivalence relation which is linear closure of the following relations:
\begin{equation}
     \label{eq: lasagna 1}
        v\sim F_\tau(v),\ \tau\in \mathrm{Br}_{k^i_+}\times \mathrm{Br}_{k^i_-}
    \tag{\ref{def: cabled HF}.1}
 \end{equation}   
\begin{equation}
     \label{eq: lasagna 2}
         v\sim F_P(v\otimes B),\  0\sim F_P(v\otimes T)
    \tag{\ref{def: cabled HF}.2}
\end{equation} \begin{equation}
     \label{eq: lasagna 3}
     v\sim F_{m_{i,j}}(v)
    \tag{\ref{def: cabled HF}.3}
     \end{equation}
 \end{definition}

% \EG{Motivation: empty sphere, dotted sphere}

In relation \eqref{eq: lasagna 1}, $\tau\in \mathrm{Br}_{k^i_+}\times \mathrm{Br}_{k^i_-}$ is a braid group element and $F_{\tau}$ is the cobordism map that gives the action of $\tau$ on $\HFLhat(L (k_+, k_-))$. This relation holds for all $\tau$ and all $i$. In the third relation \eqref{eq: lasagna 3}, the map $F_{m_{i,j}}$ from $\HFLhat(L(k_+, k_-))$ to itself is the base point moving map \cite{ZemkeBasept} on the $j$-th parallel copy of $L_i$.  

Finally, in relation \eqref{eq: lasagna 2} the map
 \[F_P:\HFLhat(L(k_{+},k_{-})\sqcup O)\to \HFLhat(L( k_{+}+\delta_i, k_{-}+\delta_i))\]
is the pair of pants cobordism map as defined in Section \ref{sec:PoPCob}.
Here $\delta_i$ is the $r$-tuple with $1$ at the $i$-th position and $0$ elsewhere.
%Now we introduce the equivalence relations induced by the cobordism maps $F_\tau, F_P$ and $F_{m_{i,j}}$. $F_\tau$ is the braid group action induced by the symmetric group $S_{k^{-}_i,k^+_{i}}$, the map $F_{m_{i,j}}$ is the base moving map on the $j$-th parallel copy of $L_i$, that is,
 %$$F_{m_{i,j}}: \HFLhat(L, k^+, k^-)\rightarrow \HFLhat(L, k^+, k^-).$$
%The last map $F_P$ is the cobordim map induced by a band. We delete a disk from the band so that the cobordims is a pair of pants from the unknot to the a pair of parallel copy of $L_i$ with opposite orientation. 

%********************

%\textcolor{red}{Convention:}$\Phi_1, \Phi_2, \Phi$ are cobordism maps between $K(n)$ and $K(n+2)$. It is the composition between birth and $F_{P_1}$ or $F_{P_2}$ or $F_{-P}$. change the cobordism relation for this notation later. Use $\F$ instead of $\F_2$.

%********************

The gradings on $\cHFL(L)$ are defined as follows. The Alexander grading is just the total Alexander grading on $\HFLhat(L(k_+, k_-))$. The normalized Maslov grading is defined as $M+|k_{+}|+|k_{-}|$ where $M$ is the Maslov grading on $\HFLhat(L(k_+, k_-))$. Finally, the grading by $H_2(W)=\Z^r$ is given by 
$$
k_{+}-k_{-}=(k_{+}^1-k_{-}^{1},\ldots,k_{+}^{r}-k_{-}^{r})\in \Z^r.
$$
The maps $F_{\tau},F_{m_{i,j}}$ and $F_P(-\otimes B)$ preserve these gradings (the latter changes the normalized Maslov grading by $-2+2=0$), while $F_P(-\otimes T)$ preserves the Alexander and homological grading but shifts the normalized Maslov grading by $-1+2=+1$. This implies that equations \eqref{eq: lasagna 1}-\eqref{eq: lasagna 3} indeed yield well-defined gradings. 

\begin{theorem}\cite[Theorem 1.1]{Chen}
\label{thm: chen}
The cabled link Floer homology of $L$ is isomorphic to the Floer lasagna module of the 4-manifold $W(L)$ obtained by adding 2-handles along the components of $L$ with the corresponding framings:
$$
\FL(W(L))\simeq \cHFL(L).
$$
\end{theorem}

Thanks to Theorem \ref{thm: chen}, we will use the notations $\FL(W(L))$ and $\cHFL(L)$ interchangeably.

\section{Representations in characteristic 2}
\label{sec: rep theory}

Before computing some examples of cabled Heegaard Floer homology, we need to review some representation theory of $S_n$ over $\F$ which might not be familiar to a topologically minded reader (as it was not to the authors). 

Suppose that a prime number $p$ does not divide the order of a finite group $G$, and a field $F$ has characteristic $p$.   Then by Maschke's theorem any finite-dimensional representation of $G$ over $F$ can be uniquely decomposed as a direct sum of irreducible representations. In our situation $p=2$ and $|G|=|S_n|=n!$, so Maschke's theorem does not apply. Indeed, many  other useful facts from the standard representation theory course fail in this situation as well, in particular:
\begin{itemize}
\item A finite-dimensional representation might not be isomorphic to its dual.
\item The spaces of invariants and coinvariants in a given representation do not need to be isomorphic. 
\end{itemize}

Below we define invariants and coinvariants, and give some explicit examples of representations illustrating the above phenomena. In Section \ref{subsec: cube} we give a slightly more general family of examples relevant for topological applications. Throughout this section we work over $\F$.

\subsection{Invariants and coinvariants}

Recall that in characteristic zero $S_n$ has two one-dimensional representations: sign and trivial. In characteristic two, they coincide and it is not hard to prove that the only one-dimensional representation of $S_n$ over $\F$ is trivial. We will denote it by $\F$ by abuse of notation.

\begin{definition}
Let $V$ be a representation of $S_n$ over $\F$. We define the space of \emph{invariants} $\Inv(V)$ as the  subspace of $S_n$-invariant elements:
$$
\Inv(V)=\{v\in V:\ \sigma(v)=v\ \mathrm{for\ all}\ \sigma\in S_n\}.
$$
Furthermore, we define the space of \emph{coinvariants} $\Coinv(V)$ as the quotient
$$
\Coinv(V)=V/\sim,\quad v\sim \sigma(v)\ \mathrm{for\ all}\ \sigma\in S_n.
$$
\end{definition}

The space of coinvariants plays an important role in the study of Floer lasagna module, since it is closely related to the equivalence relation \eqref{eq: lasagna 1} in Definition \ref{def: cabled HF}. In characteristic zero, both spaces of invariants and coinvariants are isomorphic to the trivial direct summand in $V$. As we will see below, this is not true over $\F$, and $\Inv(V)$ and $\Coinv(V)$ might have different dimensions. Nevertheless, one can still check the following easy facts using definitions.

\begin{proposition}
\label{prop: inv and coinv}
\begin{enumerate}
\item[(a)] We have $\Inv(V)=\Hom_{S_n}(\F,V)$. There is a natural embedding $\Inv(V)\hookrightarrow V$.
\item[(b)] We have $\Coinv(V)=\Hom_{S_n}(V,\F)^{\ast}$. There is a natural projection $V\twoheadrightarrow \Coinv(V)$.
\item[(c)] Given an $S_n$-equivariant map $f:V\to W$, one has induced maps 
$$
\Inv(V)\to \Inv(W),\ \Coinv(V)\to \Coinv(W).
$$
\item[(d)] For the dual representation the invariants and coinvariants are swapped:
$$
\Inv(V^{\ast})=\Coinv(V)^{\ast},\ \Coinv(V^{\ast})=\Inv(V)^{\ast}.
$$
\end{enumerate}
\end{proposition}

\begin{remark}
Note that whenever $\dim V>1$ and either $\Inv(V)$ or $\Coinv(V)$ is nontrivial then $V$ is not irreducible.  Indeed, $V$  contains either $\Inv(V)$ or the kernel of the projection $V\to \Coinv(V)$, respectively.
\end{remark}

\subsection{Two examples}

Let $\F^n$ be the $n$-dimensional representation of $S_n$ with the  basis $e_i$ where $S_n$ permutes the indices. It is easy to check that $\F^n$ is self-dual, and 
$$
\Inv(\F^n)\simeq \Coinv(\F^n)\simeq \F.
$$
Now we define two interesting $(n-1)$-dimensional representations of $S_n$.

\begin{definition}
Consider the map $a: \F^n\to \F$ which sends $e_i$ to 1 for all $i$. We define $$W_1=\Ker(a).$$
\end{definition}

\begin{definition}
Consider the map $b: \F\to \F^n$ which sends $1$ to $e_1+\ldots+e_n$. We define $$W_2=\Coker(b)=\F^n/\Imm(b).$$
\end{definition}

\begin{lemma}
\label{lem: W1 W2}
Assume $n\ge 3$. We have the following:
\begin{enumerate}
\item[(a)] We have $\Coinv(W_1)=0$ and 
$$
\dim \Inv(W_1)=\begin{cases}
1 & n\ \mathrm{even}\\
0 & n\ \mathrm{odd}.
\end{cases}
$$
\item[(b)] We have $\Inv(W_2)=0$ and 
$$
\dim \Coinv(W_2)=\begin{cases}
1 & n\ \mathrm{even}\\
0 & n\ \mathrm{odd}.
\end{cases}
$$
\item[(c)] We have $W_2=W_1^{\ast}$.
\end{enumerate}
\end{lemma}

\begin{proof}
a) The only $S_n$-invariant vector in $\F^n$ is $\sum e_i$ and $a(\sum e_i)=n$. Therefore  $\sum e_i$ belongs to $\Ker(a)$ if and only if $n$ is even, this proves the statement about $\Inv(W_1)$. 

Next, $W_1$ is spanned by the vectors $e_i-e_j$, which are all equivalent under $S_n$ action. Since $(e_i-e_j)+(e_j-e_k)+(e_k-e_i)=0$, we have $3[e_i-e_j]=0$ and $\Coinv(W_1)=0$.

b) The space $W_2$ has the basis $e_2,\ldots,e_n$, and we have $(1\ 2)e_2=-(e_2+\ldots+e_n).$ If $v=a_2e_2+\ldots+a_ne_n$ is an invariant vector in $W_2$ then $a_2=\ldots=a_n$ and $v= (1\ 2)v$ implies
$$
a_2e_2+\ldots+a_ne_n=a_2(e_2+\ldots+e_n)+a_3e_3+\ldots+a_ne_n=a_2e_2+(a_2+a_3)e_3+\ldots+(a_2+a_n)e_n.
$$
We have $a_2+a_3=a_3$, so $a_2=0$ and all $a_i=0$, therefore $\Inv(W_2)=0$.

For $\Coinv(W_2)$, we have $e_i\sim e_j$ and $e_2\sim (e_2+\ldots+e_n)$, so $e_3+\ldots+e_n\sim (n-2)e_n\sim0$. If $n$ is even then this condition is vacuous and $\dim \Coinv(W_2)=1$.

c) This is clear from the fact that both $\F^n$ and $\F$ are self-dual.
\end{proof}

We conclude that when $n>2$ is even then $W_1$ and $W_2$ are not isomorphic and not irreducible. One can check (similarly to characteristic zero) that when $n$ is odd then actually $W_1$ and $W_2$ are   isomorphic and irreducible but we will not need this fact.

\begin{remark}
If $n=1$ then $W_1=W_2=0$. If $n=2$ then both $W_1$ and $W_2$ are trivial and have one-dimensional invariants and coinvariants.
\end{remark}

\subsection{More general examples}
\label{subsec: cube}

More generally, we can consider the \emph{cube complex} $C_n=(\bigoplus_{i=0}^{n}C_{n,i},\partial)$ defined as follows. 
The space $C_{n,i}$ is spanned by the generators $e_A$ for all $i$-element subsets $A\subset \{1,\ldots,n\}$. The differential is given by 
$$
\partial(e_A)=\sum_{a\in A} e_{A\setminus a}.
$$
We will write $\partial_i$ for the component of $\partial$ from $C_{n,i}$ to $C_{n,i-1}$.
We will often identify $C_n$ with the exterior algebra generated by $e_1,\ldots,e_n$, then $e_A$ corresponds to the product of $e_a$ for $a\in A$ (since we work over $\F$, the order does not matter).

We have $\partial^2=0$, and it is well known that the complex $(C_n,\partial)$ is acyclic. The complex $C_n$ has a natural action of $S_n$ which permutes $e_A$ and commutes with $\partial$. 

Similarly to the previous section one can check that 
$$
\Inv(C_{n,i})\simeq \Coinv(C_{n,i})\simeq \F
$$
and $C_{n,i}$ is self-dual as $S_n$ representation. Moreover, $C_{n,i}\simeq C_{n,n-i}$, the isomorphism sends $e_A$ to $e_{\overline{A}}$.

\begin{definition}
We define $V_{n,i}=\Ker \partial_i\subset C_{n,i}$. 
\end{definition}

Clearly, $V_{n,i}$ is an $S_n$-representation. Note that $V_{n,n}=0$, so below we will assume $i\le n-1$. Since $(C_n,\partial)$ is acyclic, we can also write
\begin{equation}
\label{eq: V as cokernel}
V_{n,i}=\Imm\partial_{i+1}\simeq \frac{C_{n,i+1}}{\Ker \partial_{i+1}}=\frac{C_{n,i+1}}{\Imm \partial_{i+2}}=\Coker\partial_{i+2}.
\end{equation}
The representations from the previous subsection can be written as
$$
W_1=\Ker \partial_1=V_{n,1},\ W_2=\Coker \partial_n\simeq V_{n,n-2}.
$$
In order to describe the properties of $V_{n,i}$ for general $i$, it is useful to introduce some special elements.

\begin{definition}
For $a\ge 2$ define $\xi_a=e_a-e_1$. For $A\subset \{2,\ldots,n\}$ we define 
$$
\xi_A=\prod_{a\in A}\xi_a=e_A+\sum_{a\in A}e_{(A\setminus a)\cup 1}.
$$
\end{definition}

\begin{proposition}
\label{prop: Sn action on xi}
The representation $V_{n,i}$ has a basis $\xi_A$ with $|A|=i$, and $\dim V_{n,i}=\binom{n-1}{i}$. The action of $S_n$ on $V_{n,i}$ in the basis $\xi_A$ can be described as following:
\begin{enumerate}
\item[(a)] $\sigma \xi_A=\xi_{\sigma(A)}$ if $\sigma(1)=1$.
\item[(b)] 
\begin{equation}
\label{eq: Sn action V}
(1\ 2)\xi_A=\begin{cases}
\xi_A & \mathrm{if}\ 2\in A\\
\xi_A+\sum_{a\in A}\xi_{(A\setminus a)\cup 2} & \mathrm{if}\ 2\notin A.
\end{cases}
\end{equation}
\end{enumerate}
\end{proposition}

\begin{proof}
Let us prove that $\xi_A\in \Ker \partial_i$. Indeed, 
$$
\partial_i(\xi_A)=\partial_i\left(e_A+\sum_{a\in A}e_{(A\setminus a)\cup 1}\right)=\sum_{a\in A}e_{A\setminus a}+\sum_{a\in A}\left(e_{A\setminus a}+\sum_{b\in (A\setminus a)}e_{(A\setminus \{a,b\})\cup 1}\right).
$$
The terms $e_{A\setminus a}$ cancel out, and the terms $e_{A\cup 1\setminus \{a,b\}}$ appear twice from $(a,b)$ and $(b,a)$ and cancel out as well.

The elements $\xi_A$ are clearly linearly independent, let us prove that they span $\Ker \partial_i$. Indeed, suppose $$c=\sum_{|A|=i} c_Ae_A\in \Ker \partial_i.$$ For $1\notin A$ the elements $\xi_A$ agree with $e_A$ modulo $\langle e_B:1\in B\rangle$, so we can consider the element 
$$
c'=c-\sum_{1\notin A}c_A\xi_A=\sum_{1\in B}c'_Be_B \in \Ker \partial_i.
$$
For $1\in B$ the boundary $\partial_i(e_B)$ has the term $e_{B\setminus 1}$ which cannot cancel with $e_{B'\setminus 1}$ for $B\neq B',1\in B'$. Therefore $c'=0$ and 
\begin{equation}
\label{eq: xi span}
c=\sum_{1\notin A}c_A\xi_A
\end{equation}
Finally, let us prove \eqref{eq: Sn action V}. If $2\in A$ then the action of $(1\ 2)$ swaps $e_A$ and $e_{(A\setminus 2)\cup 1}$, and fixes all other terms. If $2\notin A$ then 
$$
(1\ 2)\xi_A=e_A+\sum_{a\in A}e_{(A\setminus a)\cup 2}=\xi_A+\sum_{a\in A}\xi_{(A\setminus a)\cup 2},
$$
the last equation follows from \eqref{eq: xi span}.
\end{proof}

\begin{remark}
The elements $\xi_a$ are only defined when $n\ge 2$. If $n=1$ then $V_{1,1}=0$ and $V_{1,0}=\F$ with the basis vector $e_{\emptyset}=\xi_{\emptyset}$, so Proposition \ref{prop: Sn action on xi} holds in this case as well.
\end{remark}

\begin{remark}
A reader familiar with representation theory might recognize $V_{n,i}$ as the Specht module corresponding to the partition $\lambda=(n-i,1^i)$. It is a submodule of the permutation module $C_{n,i}$ corresponding to the same partition.
\end{remark}

Now we can compute invariants and coinvariants of $V_{n,i}$.

\begin{lemma}
\label{lem: Vni}
The representations $V_{n,i}$ have the following properties:
\begin{enumerate}
\item[(a)] $$
\dim \Inv(V_{n,i})=\begin{cases}
1 & \mathrm{if}\ n-i-1\ \mathrm{even\ or}\ i=0\\
0 & \mathrm{else}. 
\end{cases}
$$
\item[(b)] $$\dim \Coinv(V_{n,i})=\begin{cases}
1 & \mathrm{if}\ i\ \mathrm{even\ or}\ i=n-1\\
0 & \mathrm{else}.
\end{cases}
$$
\item[(c)] $V_{n,n-1-i}\simeq V_{n,i}^{\ast}$.
\end{enumerate}
\end{lemma}

\begin{proof}
a) The only nonzero $S_n$-invariant element  in $C_{n,i}$ is $\sum_{|A|=i}e_A$. We have 
$$
\partial\left(\sum_{|A|=i}e_A\right)=(n-i+1)\sum_{|B|=i-1}e_B,
$$
so $\sum_{|A|=i}e_A\in \Ker \partial_i$ if and only if $n-i+1$ is even or $i=0$.

b) Clearly, all $\xi_A$ are equivalent under the $S_n$ action. If $i\le n-2$ then we can find a subset $A$ not containing $1$ and $2$, then  by \eqref{eq: Sn action V} we get
$$
\xi_A\sim (1\ 2)\xi_A\sim (i+1)\xi_A,
$$
so $i\xi_A\sim 0$. If $i$ is even then this relation is vacuous, and if $i$ is odd then $\xi_A\sim 0.$

c) We have 
$$V_{n,i}^{\ast}=\left(\Ker\left[C_{n,i}\xrightarrow{\partial_i}C_{n,i-1}\right]\right)^{\ast}=\Coker\left[C_{n,n-i+1}\xrightarrow{\partial_{n-i+1}}C_{n,n-i}\right]\simeq V_{n,n-i-1}.
$$
Alternatively, one can check that there is a perfect pairing $V_{n,i}\otimes V_{n,n-i-1}\to V_{n,n-1}\simeq \F$.
\end{proof}

\begin{example}
For $i=1$ we get $\dim\Inv(V_{n,1})=1$ if and only if $n-2$ is even, and $\dim\Coinv(V_{n,1})=1$ is and only if $1=n-1$, so $n=2$. 

For $i=n-2$ we get $\dim\Inv(V_{n,n-2})=1$ if and only if $n-2=0$, and $\dim\Coinv(V_{n,n-2})=1$ if and only if $n-2$ is even. This agrees with Lemma \ref{lem: W1 W2} for $W_1=V_{n,1},W_2=V_{n,n-2}$.
\end{example}

\subsection{Deformed cube complex}

In what follows we will need a deformation of the cube complex. Fix an integer $1\le k\le n$, and let $\UU$ be a formal variable. We consider a complex of free $\F[\UU]$-modules $C_n[\UU]=C_{n}\otimes \F[\UU]$ with the basis $e_A$ as above and the differential
$$
\partial^{(k)}(e_A)=\begin{cases}
\sum_{a\in A}e_{A\setminus a} & |A|>k\\
\UU\sum_{a\in A}e_{A\setminus a} & |A|\le k
\end{cases}
$$
The group $S_n$ acts on $C_n[\UU]$ as before and the differential $\partial^{(k)}$ is $S_n$-equivariant. We can now determine the action of $S_n$ in the homology of $\partial^{(k)}$ (compare with \cite[Lemma 4.14]{GHom}).

\begin{lemma}
\label{lem: deformed cube}
We have an isomorphism of $S_n$-representations:
$$
H_i(\partial^{(k)})\simeq \begin{cases}
V_{n,i} & i<k\\
0 & i\ge k.
\end{cases}
$$
\end{lemma}

\begin{proof}
For all $i$ we get $\Ker\partial^{(k)}_i=\Ker \partial_i\otimes \F[\UU]=V_{n,i}[\UU]$.

For $i<k$ we get $\Imm \partial^{(k)}_{i+1}=\UU \Imm \partial_{i+1}\otimes \F[\UU]=\UU V_{n,i}[\UU],$
therefore
$$
H_i(\partial^{(k)})=\Ker\partial^{(k)}_i/\Imm \partial^{(k)}_{i+1}=V_{n,i}[\UU]/\UU V_{n,i}[\UU]=V_{n,i}.
$$
For $i\ge k$ we get $\Imm \partial^{(k)}_{i+1}= \Imm \partial_{i+1}\otimes \F[\UU]=  V_{n,i}[\UU],$
therefore
$H_i(\partial^{(k)})=0.$

All of the above isomorphisms commute with the $S_n$-action, and the result follows. %\AA{The statement does not say anything about the $S_n$ action, so do we need this sentence?}
\end{proof}

\section{0-framed unknot}
\label{sec: unknot}

In this section we compute the Floer lasagna module for the $0$-framed unknot $O_0$.

\subsection{Symmetrization}

%\EG{Rewrite this using $V_{n,i}$ notations}

%Let us try to compute the 0-framed unknot more carefully. 
First, we introduce some notations. Suppose that we have an unlink with   positively oriented components indexed by a subset $\mathcal{P}$ and  negatively oriented components indexed by a subset $ \mathcal{N}$. Here
$$
|\mathcal{P}|=k_{+},\ |\mathcal{N}|=k_{-},\ \mathcal{P}\sqcup \mathcal{N}=\{1,\ldots,n\}.
$$
Without loss of generality we will make the following assumption. 

\begin{assumption}
Throughout this section  we will assume that $k_{+}\ge 2$ and the components labeled by $1$ and $2$ are positively oriented.
\end{assumption}

%We  put a marked point on the first strand. 
The full Heegaard Floer complex $\CFL(O_0(k_+,k_-))$ is the exterior algebra over $R_n$ generated by $\xi_j, j\in \mathcal{P}\setminus \{1\}$ and $\eta_{k},\ k\in \mathcal{N}$ with differential
$$
\partial(\xi_j)=U_jV_j-U_1V_1\ (j\in \mathcal{P}),\ \partial(\eta_{k})=U_{k}V_{k}-U_1V_1\ (k\in \mathcal{N}).
$$
The hat Heegaard Floer homology $\HFLhat(O_0(k_+,k_-))$ is the exterior algebra over $\F$ generated by $\xi_j,\eta_{k}$ as above with zero differential. 
It has a basis
$$
\xi_{A_+}\eta_{A_-}:\ A_{+}\subset \mathcal{P}\setminus\{1\}, A_{-}\subset \mathcal{N}.
$$

Note that $\HFL(O_0(k_+,k_-))\cong R/(U_1V_1=\cdots=U_{n}V_n)$ and the chain complex $\CFL(O_0(k_+,k_-))$ is a free resolution of $\HFL(O_0(k_+,k_-))$. So to compute the braid group action $\Br_{k_+}\times \Br_{k_-}$ on $\HFLhat(O_0(k_+,k_-))$ one can lift the braid group action on $\HFL(O_0(k_+,k_-))$ to $\CFL(O_0(k_+,k_-))$, which is unique up to chain homotopy by Lemma \ref{lem: lifting to resolutions}, and then compute the induced map on $\HFLhat(O_0(k_+,k_-))$. The action of $(\sigma_+,\sigma_-)\in \Br_{k_+}\times \Br_{k_-}$ on $R/(U_1V_1=\cdots=U_{n}V_n)$ is the isomorphism induced by permuting the indices, i.e. $(\sigma_+,\sigma_-)(1)=1$, 
\[(\sigma_+,\sigma_-)(U_i)=\begin{cases}
U_{\sigma_+(i)}& i\in \mathcal{P}\\ U_{\sigma_-(i)}& i\in \mathcal{N}.\end{cases}\]
The action $(\sigma_+,\sigma_-)(V_i)$ is similar. We lift this action to $\CFL(O_0(k_+,k_-))$ such that it factors through the action of $S_{k_{+}}\times S_{k_{-}}$.  This is generated by the transposition $(1\ 2)$, the action of $S_{k_{+}-1}$ which permutes the strands in $\mathcal{P}\setminus \{1\}$ and the action of $S_{k_{-}}$ which permutes the strands in  $\mathcal{N}$. The action of $S_{k_{+}-1}$ permutes the indices of $\xi_j$ and the action of $S_{k_{-}}$ permutes the indices of $\eta_{k}$. 
The action of the transposition $(1\ 2)$ is the most interesting and is given as follows:
$$
(1\ 2)\xi_j=\begin{cases}
\xi_2 & j=2\\
\xi_j+\xi_2 & j\neq 2,
\end{cases}
\quad
(1\ 2)\eta_{k}=\eta_{k}+\xi_2.
$$
For the action of $(1\ 2)$ on $\xi_{A_+}\eta_{A_-}$ we have two cases. If $2\in A_+$ then $(1\ 2)\xi_{A_+}\eta_{A_-}=\xi_{A_+}\eta_{A_-}$. If $2\notin A_+$ then
$$
(1\ 2)\xi_{A_+}\eta_{A_-}=\prod_{j\in A_{+}}(\xi_j+\xi_2)\prod_{k\in A_{-}}(\eta_k+\xi_2)=
$$
\begin{equation}
\label{eq: transposition action unlink}
\xi_{A_+}\eta_{A_-}+\sum_{j\in A_+}\xi_{(A_+\setminus j)\cup 2}\eta_{A_{-}}+\sum_{k\in A_{-}}\xi_{A_{+}\cup 2}\eta_{A_{-}\setminus k}.
\end{equation}
Compare with Proposition \ref{prop: Sn action on xi} which corresponds to the case $|A_-|=0$.

\begin{lemma}\label{lem: 0 framed relations}
The quotient of $\HFLhat(O_0(k_+,k_{-}))$ by the equivalence relation \eqref{eq: lasagna 1} is spanned by the elements 
$$
g(a,b),\ 0\le a\le k_{+}-1,\ 0\le b\le k_{-}
$$
modulo relations
\begin{equation}
\label{eq: equivalence}
g(a,b)\sim (a+1)g(a,b)+bg(a+1,b-1)\ \mathrm{for}\  a<k_{+}-1.
\end{equation}

\end{lemma}

\begin{proof}
Under the action of $S_{k_{+}-1}\times  S_{k_{-}}$ the monomials $\xi_{A_+}\eta_{A_-}$ are equivalent as long as they have the same $|A_+|$ and $|A_-|$. We denote the representative in this equivalence class by $g(a,b)$ for $|A_+|=a,|A_-|=b$. If $2\in A_+$, in particular when $a=k_{+}-1$, we do not impose any additional relations.

If $a<k_{+}-1$ then we can always choose a representative in an equivalence class with $A_+$ not containing 2, and \eqref{eq: transposition action unlink} implies \eqref{eq: equivalence}.
\end{proof}

Consequently, we can describe a symmetrized basis for $\HFLhat(O_0(k_{+},k_{-}))$ as follows.

\begin{lemma}
\label{lem: symmetrized basis unknot}
The quotient of $\HFLhat(O_0(k_{+},k_{-}))$ by the equivalence relation  \eqref{eq: lasagna 1} has a basis over $\F$ consisting of equivalence classes $[g(a,b)]$ where
\begin{itemize}
\item $a$ even and $b$ arbitrary.
\item $a=k_{+}-1$ odd and $b$ odd. 
\item $a=k_+-1$ odd and $b=k_-$ even.
\end{itemize}
In all cases the Maslov grading of $g(a,b)$ equals $-a-b$.
\end{lemma} 

\begin{proof}
Based on the parity of $a$ and $b$ we have the following cases for \eqref{eq: equivalence}:

\begin{enumerate}[(a)]
\item If $a$ and $b$ even then \eqref{eq: equivalence} is vacuous.
\item If $a$ odd and $b$ even then \eqref{eq: equivalence} implies $g(a,b)\sim 0$.
\item If $a$ even and $b$ odd then \eqref{eq: equivalence} implies $g(a+1,b-1)\sim 0$. Note that $a+1$ is odd and $b-1$ is even, so we do not get any new relations, unless $a+1=k_+-1$. In that case, $g(k_+-1,b)\sim 0$ if $k_+-1$ is odd and $b$ is even.
\item If $a$ and $b$ odd then \eqref{eq: equivalence} implies $g(a,b)\sim g(a+1,b-1)$.
\end{enumerate}
Consequently, by (a), (b) and (c) the equivalence classes $[g(a,b)]$ where $a$ is even and $b$ is arbitrary are all non-zero and distinct. Moreover, any $[g(a,b)]$ with $a<k_+-1$ is equal to one of these classes. Suppose $a=k_+-1$ and odd. By (c) and (d) the equivalence classes $[g(k_+-1,b)]$ where $b$ is either odd or equal $k_-$ give rise to new distinct classes that survive.
\end{proof}

\begin{comment}

\begin{lemma}
\label{lem: symmetrized basis unknot}
The quotient of $\HFLhat(O_0(k_{+},k_{-}))$ by the equivalence relation  \eqref{eq: lasagna 1} has the following basis over $\F$:
\begin{itemize}
\item $g(a,b)$, $a$ even and $b$ arbitrary.
\item $g(a,b)$, $a=k_{+}-1$ odd and $b$ odd. 
\item $g(k_{+}-1,k_{-})$ in case both $k_{+}$ and $k_{-}$ are even. 
\end{itemize}
In all cases the Maslov degree of $g(a,b)$ equals $-a-b$.
\end{lemma} 

\begin{proof}
We have the following cases:
\begin{enumerate}[(a)]
\item If $a$ is even and $b$ is even then \eqref{eq: equivalence} is vacuous and $g(a,b)$ survives.
\item If $a$ is odd and $a<k_+-1$, and $b$ is even then \eqref{eq: equivalence} implies $g(a,b)=0$. 
\item If $a=k_{+}-1$  is odd and $b<k_{-}$ is even, then \eqref{eq: equivalence} for $a-1$ and $b+1$ implies that $g(a,b)=0$. 
\item If $a=k_{+}-1$  is odd but $b=k_{-}$ is even then \eqref{eq: equivalence} is vacuous.
\item If $a$ is even and $b$ is odd then \eqref{eq: equivalence} (if applicable) implies $g(a+1,b-1)=0$. Note that $a+1$ is odd and $b-1$ is even, so we do not get any new relations but $g(a,b)$ survives. 
\item If $a$ is odd and $b$ is odd then \eqref{eq: equivalence} (if applicable) implies $g(a,b)=g(a+1,b-1)$. Note that in this case $a+1$ and $b-1$ is even, so we reduce to case (a).
%\item Assume $a=k_{+}-1$  is odd but $b<k_{-}$ is even. By applying the argument in case (c) to $a-1$ and $b+1$, we get $g(a,b)=0$. 
%\item If $a=k_{+}-1$  is odd but $b=k_{-}$ is even then there is no relation.
\end{enumerate}

\end{proof}

\end{comment}

\begin{remark}
The case $b=0$ corresponds to computing the coinvariants $\Coinv V_{k_{+},a}$. In this case Lemma \ref{lem: symmetrized basis unknot} agrees with Lemma \ref{lem: Vni}.  
\end{remark}

%\textcolor{red}{If $b=k_{-}$, there seems no relations for $a=k_{+}-1.$}

%Overall, we get the following result.
%\AA{This proposition looks like the proof of the next Lemma, how about we just move this to the proof of the Lemma. Ignore this comment if it is going to disapear after rewriting}

We illustrate the basis $g(a,b)$  for $(k_{+},k_{-})=(3,2),(4,3),(5,4)$ in Figure \ref{fig: g ab}.

\begin{figure}[ht!]
\begin{tikzpicture}
\draw [dotted,->] (0,0)--(2.5,0);
\draw [dotted,->] (0,0)--(0,2.5);
\draw (2.5,-0.2) node {$a$};
\draw (-0.2,2.5) node {$b$};

\draw[dotted] (0,0)--(2,0)--(2,2)--(0,2)--(0,0);
\draw[dotted] (1,0)--(1,2);
\draw[dotted] (0,1)--(2,1);
\draw (0,0) node {$\bullet$};
\draw (0,1) node {$\bullet$};
\draw (0,2) node {$\bullet$};
\draw (2,0) node {$\bullet$};
\draw (2,1) node {$\bullet$};
\draw (2,2) node {$\bullet$};
\end{tikzpicture}
\qquad 
\begin{tikzpicture}
\draw [dotted,->] (0,0)--(3.5,0);
\draw [dotted,->] (0,0)--(0,3.5);
\draw (3.5,-0.2) node {$a$};
\draw (-0.2,3.5) node {$b$};

\draw[dotted] (0,0)--(3,0)--(3,3)--(0,3)--(0,0);
\draw[dotted] (1,0)--(1,3);
\draw[dotted] (2,0)--(2,3);
\draw[dotted] (0,1)--(3,1);
\draw[dotted] (0,2)--(3,2);
\draw (0,0) node {$\bullet$};
\draw (0,1) node {$\bullet$};
\draw (0,2) node {$\bullet$};
\draw (0,3) node {$\bullet$};
\draw (2,0) node {$\bullet$};
\draw (2,1) node {$\bullet$};
\draw (2,2) node {$\bullet$};
\draw (2,3) node {$\bullet$};
\draw (3,1) node {$\bullet$};
\draw (3,3) node {$\bullet$};
\end{tikzpicture}
\qquad
\begin{tikzpicture}
\draw [dotted,->] (0,0)--(4.5,0);
\draw [dotted,->] (0,0)--(0,4.5);
\draw (4.5,-0.2) node {$a$};
\draw (-0.2,4.5) node {$b$};

\draw[dotted] (0,0)--(4,0)--(4,4)--(0,4)--(0,0);
\draw[dotted] (1,0)--(1,4);
\draw[dotted] (2,0)--(2,4);
\draw[dotted] (3,0)--(3,4);
\draw[dotted] (0,1)--(4,1);
\draw[dotted] (0,2)--(4,2);
\draw[dotted] (0,3)--(4,3);
\draw (0,0) node {$\bullet$};
\draw (0,1) node {$\bullet$};
\draw (0,2) node {$\bullet$};
\draw (0,3) node {$\bullet$};
\draw (0,4) node {$\bullet$};
\draw (2,0) node {$\bullet$};
\draw (2,1) node {$\bullet$};
\draw (2,2) node {$\bullet$};
\draw (2,3) node {$\bullet$};
\draw (2,4) node {$\bullet$};
\draw (4,0) node {$\bullet$};
\draw (4,1) node {$\bullet$};
\draw (4,2) node {$\bullet$};
\draw (4,3) node {$\bullet$};
\draw (4,4) node {$\bullet$};
\end{tikzpicture}
\caption{The basis $g(a,b)$ for $(k_{+},k_{-})=(3,2)$ (left), $(4,3)$ (middle) and $(5,4)$ (right).}
\label{fig: g ab}
\end{figure}
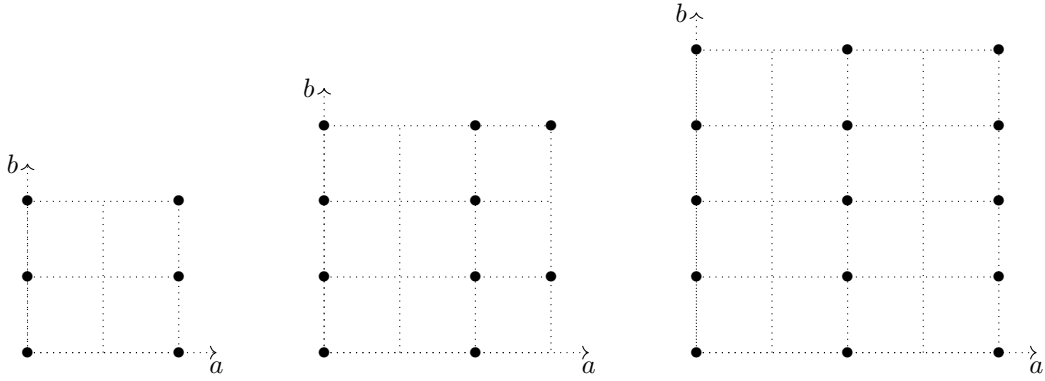

\subsection{Cabled link Floer homology}

Next, we consider the cobordism maps between $L=O_0(k_{+},k_{-})$ and $L'=O_0(k_{+}+1,k_{-}+1)$ induced by the pair of pants.

\begin{lemma}
Suppose that the new $(n+1)$-st component $L'_{n+1}$ is oriented positively and $(n+2)$-nd component $L'_{n+2}$ is oriented negatively. Then,
\begin{eqnarray}
\label{eq: pants zero framing}
F_P(v\otimes T)=v\otimes (\xi_{n+1}+\eta_{n+2}),\\
F_P(v\otimes B)=v\otimes (\xi_{n+1}\eta_{n+2}). 
\end{eqnarray}
Note that after identification $U_{n+1}=U_{n+2}$ and $V_{n+1}=V_{n+2}$ we get $\partial(\xi_{n+1}+\eta_{n+2})=0$. 
\end{lemma}
\begin{proof}
Let $D_{n+1}$ (resp. $D_{n+2}$) be the death cobordism from $L'$ to $O_{0}(k_+,k_-+1)$ (resp. $O_{0}(k_++1,k_-)$) defined by capping off $L'_{n+1}$ (resp. $L'_{n+2}$) by a disk. Then, both $D_{n+1}\circ P$ and $D_{n+2}\circ P$ are isotopic to the product cobordism. So, the cobordism map
\[F_{D_{n+1}}\circ F_P: \HFLhat(L\sqcup O)\cong \HFLhat(L)\otimes\langle T,B\rangle\to \HFLhat(O_{0}(k_+,k_-+1))\]
after identifying $T$ and $B$ with $1$ and $\eta_{n+2}$ in our algebraic model is equal identity. Similarly, $F_{D_{n+2}}\circ F_p$ is equal identity after identifying $T$ and $B$ with $1$ and $\xi_{n+1}$, respectively. Consequently, the claim follows from these identities and the fact that $F_P$ drops Maslov grading by $1$ (Lemma \ref{lem:PoP-grading}).
\end{proof}

\begin{theorem}
\label{thm: zero framed unknot}
For  $k_+-k_-=\alpha$ the cabled link Floer homology  $\cHFL(O_0)$  for the $0$-framed unknot $O_0$ (equivalently, the Floer lasagna module for $W_0(O)=S^2\times D^2$) has the following generators:
\begin{enumerate}
\item[(a)] For $\alpha$ odd, one generator of each odd normalized Maslov grading $1,3,\ldots$.
\item[(b)] For $\alpha$ even, one generator of each even normalized Maslov grading $2,4,\ldots$ and an additional generator of grading $1$.
\end{enumerate}
\end{theorem}

\begin{proof}

The basepoint moving map for the unlink is trivial, so the equivalence relation \eqref{eq: lasagna 3} is vacuous. The quotient by the equivalence relation \eqref{eq: lasagna 1} is given by Lemma \ref{lem: symmetrized basis unknot}, we will denote the respective generators by $g_n(a,b)$ for $n=k_{+}+k_{-}.$ 

It remains to consider the  equivalence relation \eqref{eq: lasagna 2}. It has the form 
\begin{equation}
\label{eq: B equivalence}
F_P(g_n(a,b)\otimes B)=g_{n+2}(a+1,b+1)\sim g_n(a,b)
\end{equation}
and
\begin{equation}
\label{eq: T equivalence}
F_P(g_n(a,b)\otimes T)=g_{n+2}(a,b+1)+g_{n+2}(a+1,b)\sim 0.
\end{equation}
From \eqref{eq: T equivalence} we get that the classes $g_{n}(a,b)$ for $(a,b)$ on the same diagonal $a+b=d$ are all equivalent. This agrees with case (d) in the proof of Lemma \ref{lem: symmetrized basis unknot}. Note that these classes have normalized Maslov grading $-d+k_{+}+k_{-}$ which ranges from $1$ (for $d=k_{+}-1+k_{-}$)  to $k_{+}+k_{-}$ (for $d=0$). 

The equivalence relation \eqref{eq: B equivalence} preserves the normalized Maslov grading, and sends an equivalence class of $g_n(a,b)$ with $a+b=d$ to the equivalence class of $g_{n+2}(a+1,b+1)$ with $(a+1)+(b+1)=d+2$. 

 If $a$ is odd and $b$ is even then $g_{n}(a,b)=0$ by the case (b) in the proof of Lemma \ref{lem: symmetrized basis unknot}. If $a$ is even and $b$ is odd then $g_n(a,b)\sim g_n(a+1,b-1)\sim 0$, unless $(a,b)=(k_{+}-1,k_{-})$.
 
Therefore, the cabled homology $\cHFL(O_0)$ is spanned by the equivalence classes of $g_n(a,b)$ with $a+b$ even, and the equivalence class of $g_n(k_{+}-1,k_{-})$.
This yields classes with positive normalized Maslov gradings of the same parity as $k_{+}+k_{-}=\alpha\ (\mathrm{mod}\ 2)$ (and any such class appears for sufficiently large $k_{+},k_{-})$, and an additional class of normalized Maslov grading 1.
\end{proof}

\section{Homology of $(n,n)$ torus links}
\label{sec: Tnn}

\subsection{Full homology}

In this subsection we recall the description of the full Heegaard Floer homology $\HFL(T(n,n))$ from \cite{BLZ}. This is a plumbed L-space link, so the results from Section \ref{sec: L space background} apply.

We first assume that all strands are oriented the same way. In this case $\HFL(T(n,n))$  is generated over $R_n=\F[U_1,\ldots,U_n,V_1,\ldots,V_n]$ by the elements $X_j$ for $j=1,\ldots,n$ of Alexander grading 
\begin{equation}
\label{eq: Alex Xj}
A(X_j)=\left(\frac{n+1}{2}-j,\ldots,\frac{n+1}{2}-j\right)
\end{equation}
 and Maslov grading $M(X_j)=-j(j-1)$. The relations have the form
\begin{equation}
\label{eq: HFL full relations}
U_AX_{j}=V_{\overline{A}}X_{j+1},\ U_iV_iX_j=U_kV_kX_j
\end{equation}
for all subsets $A\subset \{1,\ldots,n\},\ |A|=j$, and all $i\neq k$. Here $\overline{A}=\{1,\ldots,n\}\setminus A$.
Furthermore, Theorem \ref{thm: resolution} implies  the following.

\begin{theorem}
\label{thm: BLZ}\cite[Theorem 6.1]{BLZ}
The chain complex $\CFL(T(n,n))$ is homotopy equivalent to a free resolution of $\HFL(T(n,n))$ over $R_n$. 
\end{theorem}

\begin{corollary}\label{cor:T(n,n)hat}
The generators in the  minimal free resolution of $\HFL(T(n,n))$ give a basis for the hat homology $\HFLhat(T(n,n))$.
\end{corollary}

See Figure \ref{fig: T33} for the $n=3$ example.

\begin{figure}[h]
\begin{tikzcd}[row sep=1.4cm,column sep=.9cm,labels=description]
&&&[.8cm]&
e
	\ar[dllll,"\scU_1"]
	\ar[dlll,"\scU_2"]
	\ar[dll,"\scU_3"]
	\ar[drr,"\scV_1"]
	\ar[drrr,"\scV_2"]
	\ar[drrrr,"\scV_3"]
&&[.8cm]&&&\\
c_1
	\ar[d,"\scU_2", pos=.7]
	\ar[drrr,dashed,"\scV_1", pos=.75]
	\ar[dr,"\scU_3", pos=.7]
&
c_2
	\ar[dl,crossing over,"\scU_1", pos=.7]
	\ar[dr,crossing over,"\scU_3", pos=.7]
	\ar[drrrr,dashed,"\scV_2"]
&c_3
	\ar[d,crossing over,"\scU_2", pos=.7]
	\ar[dl,crossing over,"\scU_1", pos=.7]%This was a typo
	\ar[drrr,dashed,"\scV_3",crossing over,pos=.45]
	\ar[dr,dashed,"\scV_3",crossing over, pos=.58]
&&&&
d_1
	\ar[dlll,dashed,crossing over,"\scU_1", pos=.45]
&d_2
	\ar[dll,dashed,crossing over,"\scU_2", pos=.6]
&d_3
	\ar[d,"\scV_2", pos=.7,crossing over]
	\ar[dlll,dashed,"\scU_3",pos=.7, crossing over]
	\ar[dlllll,dashed,"\scU_3",pos=.5, crossing over]
	\\[1cm]
b_3
	\ar[drrrr,dashed,"\scV_1\scV_2", pos=.65]
	\ar[dr,"\scU_3",crossing over]
&
b_2
	\ar[drrr,dashed,"\scV_1\scV_3"]
	\ar[d,"\scU_2",crossing over]
&
b_1
	\ar[drr,dashed,"\scV_2\scV_3", pos=.35]
	\ar[dl,"\scU_1",crossing over]
&Z_1
	\ar[dr,dashed,"\scU_2\scV_2+\scU_3\scV_3" near start]
&
&Z_2
	\ar[dl,dashed,"\scU_1\scV_1+\scU_3\scV_3" near start]
&
B_3
	\ar[dll,dashed,"\scU_1\scU_2", pos=.35]
	\ar[from=u,crossing over, "\scV_2", pos=.7]
	\ar[from=ur,crossing over, "\scV_1", pos=.7]
&B_2
	\ar[dlll,dashed,"\scU_1\scU_3"]
	\ar[from=ul,crossing over, "\scV_3", pos=.7]
	\ar[from=ur,crossing over, "\scV_1", pos=.7]	
&B_1
	\ar[dllll,dashed,"\scU_2\scU_3", pos=.65]
	\ar[from=ul, "\scV_3",crossing over, pos=.7]
\\[1cm]
&X_1
	\ar[from=u, crossing over,"\scU_2"]
	%\ar[from=ur, crossing over]
&&&X_2&&&
X_3
	\ar[from=ul,crossing over,"\scV_3"]
	\ar[from=u, crossing over,"\scV_2"]
	\ar[from=ur, crossing over,"\scV_1"]
\end{tikzcd}
\caption{The complex $\CFL(T(3,3))$ as a free resolution.}
\label{fig: T33}
\end{figure}
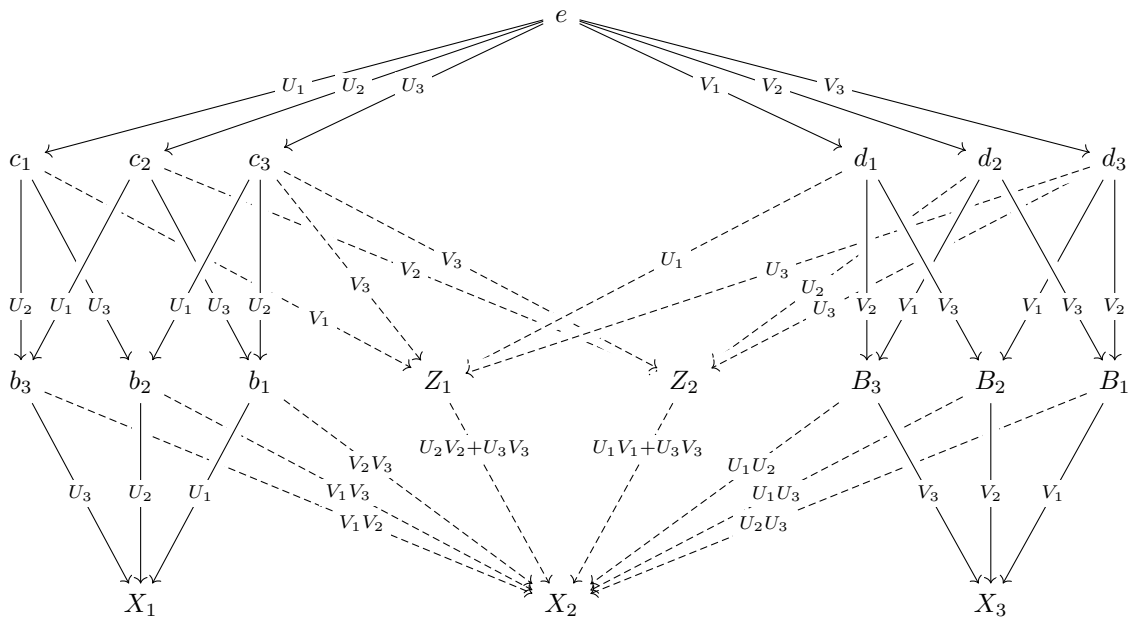

As a consequence of Corollary \ref{cor:T(n,n)hat}, the hat homology $\HFLhat(T(n,n))$ acquires an additional {\bf level} grading (see Section \ref{sec: L space background}).
For example, in Figure \ref{fig: T33} the classes $X_{1},X_{2},X_{3}$ all have level zero, while the class $e$ has level 3.

If we identify some of the variables on the chain complex level then we can create higher homology. However, we have the following.

\begin{lemma}
\label{lem: identify U not V}
Suppose we identify $U_i=U_j$ in $\CFL(T(n,n))$ but keep all other variables free. Then the resulting complex is still a resolution of its homology, that is, it only has homology at level $0$.
\end{lemma}

\begin{proof}
We need to prove that the homology of $\CFL(T(n,n))\otimes_{R} R/(U_i-U_j)$ is concentrated in level zero. By definition of the $\mathrm{Tor}$ functor, we can compute $\mathrm{Tor}\left(\HFL(T(n,n)),R/(U_i-U_j)\right)$ in two ways: first, we consider the resolution $\CFL(T(n,n))$ of $\HFL(T(n,n))$, tensor it with $R/(U_i-U_j)$ and take homology. Second, we consider the resolution $R\xrightarrow{U_i-U_j}R$ of $R/(U_i-U_j)$, tensor it with $\HFL(T(n,n))$ and take homology. 

So we need to prove that the homology of the two-term complex $$
\HFL(T(n,n))\xrightarrow{U_i-U_j}\HFL(T(n,n))
$$
is concentrated in level zero, or equivalently the multiplication by $U_i-U_j$ is injective in $\HFL(T(n,n))$. But this is clear since the multiplications by $U_i$ and $U_j$ are both injective and shift the Alexander grading in different directions, so there is no cancellation.
\end{proof}

%\textcolor{red}{add the formal definition of $\Phi$ in Section 3.}

\begin{remark}
Note that if $\CFL(T(n,n))$ is a minimal free resolution of $\HFL(T(n,n))$, then  
\[\CFL(T(n,n))\otimes_{\F}\left(B\xrightarrow{U_{n+1}V_{n+1}-U_1V_1}T\right)\]
is chain homotopic to $\CFL(T(n,n)\sqcup O)$ and a minimal free resolution of $\HFL(T(n,n)\sqcup O)$. Furthermore, 
\[\lvl(x\otimes T)=\lvl(x)\ \ \text{and}\ \ \lvl(x\otimes B)=\lvl(x)+1.\]
\end{remark}

\begin{proposition}\label{prop:level-preserving}
a) Up to a homotopy, the maps $f_{-P}$ and $\varphi_1,\varphi_2$ in $\CFL$ are level-filtered  while $f_P(-\otimes T)$ and $f_P(-\otimes B)$ increase the level at most by 1 and 2 respectively.

b) The induced maps $F_{-P}$  and $\Phi_1,\Phi_2$ in $\HFLhat$ are level-filtered while $F_P(-\otimes T)$ and $F_P(-\otimes B)$ increase the level at most by 1 and 2 respectively.
\end{proposition}

\begin{proof}
Part (a): By Lemma \ref{lem: identify U not V} identifying $U_{n+1}$ and $U_{n+2}$ in $\CFL(T(n+2,n+2))$ we still get a free resolution of its homology. Both $\varphi_1$ and $\varphi_2$ are lifts of $\Phi_1$ and $\Phi_2$, respectively, which  are homotopic to level filtered chain maps by Corollary \ref{cor: homotopic to level filtered}.

Applying Corollary \ref{cor: homotopic to level filtered}, $h$ can be chosen to raise level at most 1.  The statements for $F_P(-\otimes T)$ and $F_P(-\otimes B)$ now follow from Lemma \ref{lem:h-PP-1}, since $B$ and $h$ both increase the level by 1.

Part (b) follows from (a), Corollary  \ref{cor:T(n,n)hat} and Lemma \ref{lem:h-PP-1}. 
\end{proof}

\subsection{Hat homology}

We start by recalling the description of $\HFLhat(T(n,n))$ with all strands oriented the same way from \cite{GHom}.  

\begin{theorem}\cite[Theorem 3]{GHom}
\label{thm: GH} The homology $\HFLhat(T(n,n))$ has the following form:

a) In the Alexander grading $A(X_j)$ for $1\le j\le \frac{n+1}{2}$ the homology is isomorphic to 
$$
\bigoplus_{i=0}^{j-1}\binom{n-1}{i}\F_{-j(j-1)-i}\oplus \bigoplus_{i=0}^{j-2}\binom{n-1}{i}\F_{-j(j-1)-n+2+i}.
$$

b) In the Alexander grading $A(X_j)$ for $\frac{n+1}{2}\le j\le n$ the homology is isomorphic to 
$$
\bigoplus_{i=0}^{n-j}\binom{n-1}{i}\F_{-j(j-1)-i}\oplus \bigoplus_{i=0}^{n-j-1}\binom{n-1}{i}\F_{-j(j-1)-n+2+i}.
$$

c) Suppose that the Alexander grading has $s$ coordinates equal to $\frac{n+1}{2}-j-1$ and $n-s$ coordinates equal to $\frac{n+1}{2}-j$, for $1\le j\le n-1$ and $1\le s\le  n-1$. Then the  homology is isomorphic to 
$$
\binom{n-2}{j-1}\F_{-j^2-s+1}.
$$
 The (hat) homology in all other Alexander gradings is zero. 
\end{theorem}

Some comments and notations are in order. We call the homology in (a) and (b) {\bf diagonal} and the homology in (c) {\bf off-diagonal}. For $1\le j\le n$ we call the sum 
$$\bigoplus_{i=0}^{\min(j-1,n-j)}\binom{n-1}{i}\F_{-j(j-1)-i}
$$
the {\bf bottom diagonal block} at $j$, and the sum 
$$
\bigoplus_{i=0}^{\min(j-2,n-j-1)}\binom{n-1}{i}\F_{-j(j-1)-n+2+i}
$$
the {\bf top diagonal block} at $j$, so that the diagonal part of the homology is the direct sum of the top and bottom diagonal blocks.

\begin{example}
\label{ex: T33}
In Figure \ref{fig: T33} the bottom diagonal blocks are spanned by $X_1$ (for $j=1$), $X_2,Z_1,Z_2$ (for $j=2$) and $X_3$ (for $j=3$),  while the only class in the top diagonal block is $e$ for $j=2$.

The off-diagonal classes are organized as follows:
\begin{itemize}
\item The classes $b_i$ correspond to $j=1,s=1$ and have Maslov grading $-1$.
\item The classes $c_i$ correspond to $j=1,s=2$ and have Maslov grading $-2$.
\item The classes $d_i$ correspond to $j=2,s=1$ and have Maslov grading $-4$.
\item The classes $B_i$ correspond to $j=2,s=2$ and have Maslov grading $-5$.
\end{itemize}
\end{example}

\begin{convention}
\label{conv: complex}
Following Corollary \ref{cor:T(n,n)hat} we will always consider the model for the chain complex $\CFL(T(n,n))$ with generators corresponding to some basis in $\HFLhat(T(n,n))$. We will classify these generators as off-diagonal, top diagonal and bottom diagonal. 
\end{convention}

The following result determines the level grading for diagonal blocks in $\HFLhat(T(n,n))$. 

\begin{lemma}
\label{lem: levels}
a) The subspace $\binom{n-1}{i}\F_{-j(j-1)-i}$ in the bottom diagonal block has level $i$.

b) The subspace $\binom{n-1}{i}\F_{-j(j-1)-n+2+i}$ in the top diagonal block has level $n-i$. 
\end{lemma}
 
\begin{proof}
We follow the computation in \cite{GHom} which considered two spectral sequences from Lemma \ref{lem: spectral sequences}. 
Each of the spectral sequences has an additional ``cube" grading defined in \cite{GHom} and the level grading is the sum of the two cube gradings. 

By symmetry we can assume that $j\le \frac{n+1}{2}$. Then by \cite[Theorem 5]{GHom} the terms in (a) have cube grading $i$ in the spectral sequence $\HFL(T(n,n))\Rightarrow \HFLminus(T(n,n))$ and by \cite[Theorem 4.22]{GHom} it has cube grading $0$ in the spectral sequence $\HFLminus(T(n,n))\Rightarrow \HFLhat(T(n,n))$, so the level equals $i$.

Similarly, the terms in (b) have cube grading $j-2-i$ in the spectral sequence $\HFL(T(n,n))\Rightarrow \HFLminus(T(n,n))$ and the cube grading $n-j+2$ in the spectral sequence $\HFLminus(T(n,n))\Rightarrow \HFLhat(T(n,n))$, so the level equals $(j-2-i)+(n-j+2)=n-i$.
\end{proof}

For each $j$ we can schematically draw the diagonal blocks as in Figure \ref{fig: diagonal blocks}. The vertical axis denotes the level grading which runs from $0$ to $n$.  The bottom diagonal block runs from level $0$ to 
$$
m_j=\min(j-1,n-j)
$$
and its dimension at level $i$ equals $\binom{n-1}{i}$. The top diagonal block runs from level
$$
M_j=n-\min(j-2,n-j-1)=\max(n-j+2,j+1)
$$
to level $n-0=n$, and the dimension at level $n-i$ equals $\binom{n-1}{i}$.

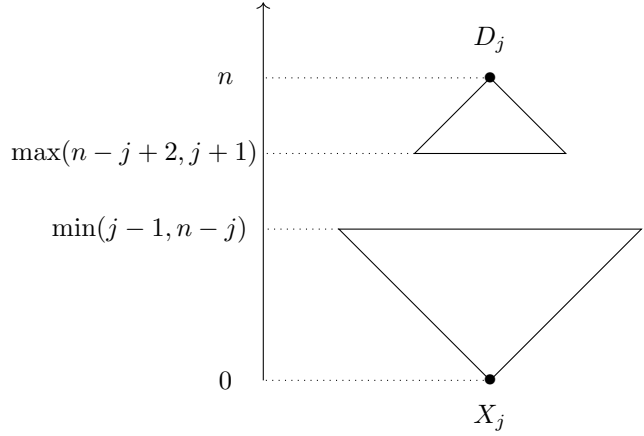
\begin{figure}[ht!]
\begin{tikzpicture}
\draw (0,0)--(-2,2)--(2,2)--(0,0);
\draw (0,4)--(-1,3)--(1,3)--(0,4);
\draw[->] (-3,0)--(-3,5);
\draw [dotted] (0,0)--(-3,0);
\draw [dotted] (-2,2)--(-3,2);
\draw [dotted] (-1,3)--(-3,3);
\draw [dotted] (0,4)--(-3,4);
\draw (-3.5,0) node {$0$};
\draw (-4.5,2) node {$\min(j-1,n-j)$};
\draw (-4.7,3) node {$\max(n-j+2,j+1)$};
\draw (-3.5,4) node {$n$};
\draw (0,0) node {$\bullet$};
\draw (0,4) node {$\bullet$};
\draw (0,-0.5) node {$X_{j}$};
\draw (0,4.5) node {$D_{j}$};
\end{tikzpicture}
\caption{Diagonal blocks at Alexander grading $A(X_j)$.}
\label{fig: diagonal blocks}
\end{figure}

\begin{remark}
\label{rem: M-m}
We have $M_j-m_j\ge 2$ and $M_j-m_j=2$ if and only if $j=\frac{n+1}{2}$.
\end{remark}

\begin{corollary}
\label{cor: Maslov separation}
Suppose that $x$ and $y$ are two diagonal classes with Alexander grading $A(X_j)$ and the same Maslov grading $M(x)=M(y)$, and $\lvl(x)\ge \lvl(y)$. Then either $\lvl(x)=\lvl(y)$ or $j=\frac{n+1}{2},\lvl(x)=M_j$ and $\lvl(y)=m_j$.
\end{corollary}

\begin{proof}
If $x$ and $y$ are both in the bottom (resp. top) diagonal block at $A(X_j)$ then they have the same Maslov grading if and only if they have the same level. 

Assume now that $x$ is at the top diagonal block at level $n-i$ and $y$ is at the bottom diagonal block at level $i'.$ We have
$$
-j(j-1)-i'=-j(j-1)-n+2+i,
$$
so $i'=(n-i)-2$ and $\lvl(y)=\lvl(x)-2$. By Remark \ref{rem: M-m} this is only possible for $j=\frac{n+1}{2},\lvl(x)=M_j$ and $\lvl(y)=m_j$.
\end{proof}

The level structure for the off-diagonal homology is more complicated.

\begin{lemma}
\label{lem: levels off diagonal}
The off-diagonal homology for a given $j,s$  is supported in levels 
$$
\lvl(j,s,c)=j-s+2c-1
$$
where $c$ runs over the interval
\begin{equation}
\label{eq: c j inequalities}
\max(1,s-j+1)\le c\le \min(s,n-j).
\end{equation}
For a fixed $c$, the rank of the off-diagonal summand equals
$$
\binom{s-1}{s-c}\binom{n-s-1}{j+c-s-1}=\binom{s-1}{c-1}\binom{n-s-1}{n-j-c}.
$$
\end{lemma}

\begin{proof}
This follows from the proof of \cite[Theorem 4.25]{GHom}. For a fixed $c$, the corresponding summand has cube grading $j+c-s-1$ in the spectral sequence $\HFL\Rightarrow \HFLminus$, and cube grading $c$ in the spectral sequence $\HFLminus\Rightarrow \HFLhat$, so the total level equals $(j+c-s-1)+c=j+2c-s-1$.
\end{proof}

\begin{remark}
Observe that 
$$
\sum_{c}\binom{s-1}{c-1}\binom{n-s-1}{n-j-c} =\binom{n-2}{n-j-1}=\binom{n-2}{j-1},
$$
in agreement with Theorem \ref{thm: GH}(c).
\end{remark}

\begin{example}
\label{ex: extremal off-diagonal 1}
Consider the case $s=1$. By \eqref{eq: c j inequalities} this forces $c=1$, so all the off-diagonal summands in the corresponding Alexander grading are supported at the level 
$$
j-1+2-1=j.
$$
\end{example}

\begin{example}
\label{ex: extremal off-diagonal}
Consider the case $s=n-1$. By \eqref{eq: c j inequalities} this forces $c=n-j$, so all the off-diagonal summands in the corresponding Alexander grading are supported at the level 
$$
j-(n-1)+2(n-j)-1=n-j.
$$
\end{example}

\subsection{Changing orientations on components}
\label{sec: change orientation Tnn}

Let $O_1(k_{+},k_{-})$ be the $(n,n)$ torus link with $k_{+}$ strands in positive direction and $k_{-}$ strands in negative direction. It is a cable of the $1$-framed unknot $O_1$, and  $n=k_{+}+k_{-}$.  By Proposition \ref{prop: reversing orientation} the generators in $\HFLhat(O_1(k_{+},k_{-}))$ are in bijection with those of $\HFLhat(T(n,n))$ after some grading shifts which we determine below. In particular, the diagonal blocks  (resp. off-diagonal) in $\HFLhat(T(n,n))$ correspond to some blocks in $\HFLhat(O_1(k_{+},k_{-}))$ which by abuse of notation we also call diagonal (resp. off-diagonal).  

\begin{lemma}
\label{lem: change orientation n n}
a) After changing the orientation on $k_{-}$ strands, the generator $X_j$ for $\HFL(T(n,n))$ corresponds to the generator $X'_j$ for $\HFL(O_1(k_{+},k_{-}))$ of Maslov grading
$$
-(j-k_{-})(j-k_{-}-1).
$$

b) More generally, the bottom diagonal block at Alexander grading $A(X_j)$ in $\HFLhat(T(n,n))$ corresponds to the block 
$$\bigoplus_{i=0}^{\min(j-1,n-j)}\binom{n-1}{i}\F_{-(j-k_{-})(j-k_{-}-1)-i}
$$
 in $\HFLhat(O_1(k_{+},k_{-}))$ with Alexander grading $A(X'_j)$.

c) The top  diagonal block at $A(X_j)$  in $\HFLhat(T(n,n))$ corresponds to the block 
$$
\bigoplus_{i=0}^{\min(j-2,n-j-1)}\binom{n-1}{i}\F_{-(j-k_{-})(j-k_{-}-1)-n+2+i}
$$
in $\HFLhat(O_1(k_{+},k_{-}))$ with Alexander grading $A(X'_j)$.
\end{lemma}

\begin{proof}
Parts (b) and (c) follow from (a) since for a fixed Alexander grading the Maslov grading is shifted by the same amount.

To prove (a), we use Corollary \ref{cor: change orientation many components}. We change orientation on $k_{-}$ components, each of them has linking number $1$ with each of $k_{+}$ remaining ones. Therefore  
$$
M(X'_j)=M(X_j)-2k_{-}\left(\frac{n+1}{2}-j\right)+k_{+}k_{-}=-j(j-1)-k_{-}\left(k_{-}+k_{+}+1-2j\right)+k_{+}k_{-}=
$$
$$
-j^2+(2k_{-}+1)j-k_{-}(k_{-}+1)=-(j-k_{-})(j-k_{-}-1).
$$
\end{proof}

\begin{remark}
\label{rem: jj}
It is useful to change the variables $\jj=\frac{n+1}{2}-j$. Then all diagonal terms at $j$ will have total Alexander  grading $\jj(k_{+}-k_{-})=\jj\alpha$. We have
$$
j-k_{-}=\frac{n+1}{2}-\jj-k_{-}=\frac{k_{+}-k_{-}+1}{2}-\jj=\frac{\alpha+1}{2}-\jj,
$$
so in the bottom diagonal block we have Maslov gradings
$$
-(j-k_{-})(j-k_{-}-1)-i=-\left(\frac{\alpha+1}{2}-\jj\right)\left(\frac{\alpha-1}{2}-\jj\right)-i,\ 
 $$
 and in the top diagonal block we have Maslov gradings
 $$
-(j-k_{-})(j-k_{-}-1)-n+2+i=-\left(\frac{\alpha+1}{2}-\jj\right)\left(\frac{\alpha-1}{2}-\jj\right)-n+2+i.
$$
\end{remark}

Next, we compute the Maslov grading for the off-diagonal terms after changing orientation.

\begin{lemma}
\label{lem: off diagonal change orientation}
Given a class $a$ in $\HFLhat(T(n,n))$ with $s$ components of the Alexander grading equal to $\frac{n+1}{2}-j-1$ and $n-s$ equal to $\frac{n+1}{2}-j$, let us change the orientation on $s_{-}$ strands in the former group and $k_{-}-s_{-}$ strands in the latter group. The resulting class $a'$ has Maslov grading 
\begin{equation}
\label{eq: off diagonal change orientation}
-(j-k_{-})^2-s-k_{-}+2s_{-}+1.
\end{equation}
\end{lemma}

\begin{proof}
The original class $a$ has grading $M(a)=-j^2-s+1$. Similarly to Lemma \ref{lem: change orientation n n} we get
$$
M(a')=M(a)-2s_{-}\left(\frac{n+1}{2}-j-1\right)-2(k_{-}-s_{-})\left(\frac{n+1}{2}-j\right)+k_{+}k_{-}=$$
$$
-j^2-s+1-2k_{-}\left(\frac{n+1}{2}-j\right)+2s_{-}+k_{+}k_{-}=
$$
$$
-j^2-s+1-k_{-}^2-k_{+}k_{-}-k_{-}+2jk_{-}+2s_{-}+k_{+}k_{-}=
$$
$$
-(j-k_{-})^2-s-k_{-}+2s_{-}+1.
$$
\end{proof}

We will often use the following consequence of naturality for cobordism maps.
Recall that in our conventions the cobordism maps $\Phi_1,\Phi_2,F_{-P}$ and $F_P$ only affect the components indexed by $n+1$ and $n+2$. In particular, it is sufficient to assume that the first $n$ components are oriented positively, and we consider these maps between $O_1(n,0)$ and $O_1(n+1,1)$.

\begin{proposition}
\label{prop: cob commute with orientation reversal}
The cobordism maps $\Phi_1,\Phi_2,F_{-P}$ and $F_P$ commute with the isomorphisms from Proposition \ref{prop: reversing orientation} which reverse orientation on the strands $1,\ldots,n$.
\end{proposition}

\subsection{Braid group action}

\begin{lemma}
\label{lem: Sn action diagonal blocks}
The action of the braid group $\Br_n$ in the diagonal blocks of $\HFLhat(T(n,n))$ is level-filtered and in fact level-preserving unless $j=\frac{n+1}{2}$ and $\lvl=M_j$. The level-preserving component of the braid group action factors through $S_n$, and yields the following $S_n$ representations:

a) In the bottom diagonal block at $A(X_j)$, we have the representation
$$
\bigoplus_{i=0}^{\min(j-1,n-j)}V_{n,i}
$$

b) In the top diagonal block at $A(X_j)$, we have the representation 
$$
\bigoplus_{i=0}^{\min(j-2,n-j-1)}V_{n,n-1-i}.
$$
\end{lemma}

\begin{proof}
By Corollary \ref{cor: homotopic to level filtered} the action of $\Br_n$ is level-filtered. It preserves the Maslov grading so by Corollary \ref{cor: Maslov separation} it is level-preserving unless $j=\frac{n+1}{2}.$

The rest of the proof follows from the results of \cite{GHom}. First, since $T(n,n)$ is an L-space link we have 
$$
\Am(L,v)\simeq \F[\UU][-2h(v)]
$$
where $h(v)$ is the $h$-function computed in \cite[Theorem 4.3]{GHom}. We denote the generator of $\Am(L,v)$ by $z_v$.

The $h$-function is clearly symmetric with respect to the permutation of components of $v$. Given a braid $\beta\in \Br_n$ with the corresponding permutation $\sigma\in S_n$, the action of the cobordism map $F_{\beta}$ yields an isomorphism  
$$
F_{\beta}:\Am(L,v)\to \Am(L,\sigma(v))
$$
and therefore $F_{\beta}(z_v)=z_{\sigma(v)}$. We conclude that the action of $\Br_n$ on $\HFL(T(n,n))$ factors through $S_n$. 

Next, as in Lemma \ref{lem: levels} one can use the spectral sequences $\HFL(L)\Rightarrow \HFLminus(L)$ and   $\HFLminus(L)\Rightarrow \HFLhat(L)$ to determine the latter homology. All these spectral sequences are built from $\Am$ by purely algebraic methods and hence commute with the $S_n$ (resp. braid group) action as well. On the other hand, their $E_{\infty}$ pages only compute the associated graded of the corresponding representations with respect to level filtration.
Thus the level-preserving component of the braid group action factors through $S_n$, but one needs to compute the corresponding representations.

By Lemma \ref{lem: deformed cube} and \cite[Theorem 5]{GHom} the diagonal block in $\HFLminus(T(n,n))$ is isomorphic to 
$
\bigoplus_{i=0}^{j-1}V_{n,i}.
$

For $j\le \frac{n+1}{2}$ the bottom diagonal block in $\HFLhat(T(n,n))$ is isomorphic to the corresponding block in $\HFLminus(T(n,n))$.
Similarly, using \cite[Theorem 4.22]{GHom} and Lemma \ref{lem: Vni} one can check that the top diagonal block  in $\HFLhat(T(n,n))$ is given by 
$
\bigoplus_{i=0}^{j-2}V_{n,n-1-i}.
$
See also Example \ref{ex: second block rep} below.

For $j\ge \frac{n+1}{2}$ the result follows from the symmetry of $\HFLhat$ which exchanges $j\leftrightarrow n+1-j$ and commutes with the action of $S_n$.
\end{proof}

\begin{example}
\label{ex: second block rep}
Following the proof of \cite[Theorem 4.22]{GHom}, one can determine the representations for diagonal blocks for small $j$. We assume that $n$ is sufficiently large.

For $j=1$ the bottom diagonal block in $\HFLhat$ agrees with the diagonal block in $\HFLminus$ and is isomorphic to $\F$, while the top diagonal block vanishes.

For $j=2$ the bottom diagonal block in $\HFLhat$ (resp. diagonal block in $\HFLminus$) is isomorphic to $\F\oplus V_{n,1}$ while the top diagonal block is isomorphic to the diagonal block in $\HFLminus$ at $j=1$, that is, $\F$.

For $j=3$  the bottom diagonal block in $\HFLhat$ (resp. diagonal block in $\HFLminus$) is isomorphic to $\F\oplus V_{n,1}\oplus V_{n,2}$. The top diagonal block is isomorphic as $S_n$-representation to the cokernel 
$$
\Coker\left[\F\oplus V_{n,1}\to \F^n\oplus \F^n\right]=\Coker\left[\F\to \F^n\right]\oplus \Coker[V_{n,1}\to \F^n]\simeq V_{n,n-2}\oplus \F.
$$
Here $\F\oplus V_{n,1}$ in the first equation corresponds to the diagonal block in $\HFLminus$ at $j=2$.
\end{example}

\subsection{Basepoint moving maps}
We recall Zemke's formula \cite[Theorem B]{ZemkeBasept} for the basepoint moving map around $L_i$
$$
F_{m_i}=1+\Phi_{w_i}\Psi_{z_i},\ i=1,\ldots,n
$$
 where $\Phi_{w_i}$ and $\Psi_{z_i}$ are respectively the partial derivatives of the differential $\partial$ with respect to $U_i$ and $V_i$.  
The map $F_{m_i}$ has Maslov and Alexander grading zero.

\begin{lemma}
\label{lem: basepoint moving}
Assume $j\neq \frac{n+1}{2}$. Then $F_{m_i}=1$ in the diagonal part of $\HFLhat(T(n,n))$.
\end{lemma}

\begin{proof}
We need to check that $\Phi_{w_i}\Psi_{z_i}=0$. The map $\Phi_{w_i}\Psi_{z_i}$ preserves the Alexander multi-grading and Maslov grading $M$, and decreases the level by $2$. Therefore it vanishes on both diagonal blocks, but may in principle send the top block to the bottom one. By Remark \ref{rem: M-m} this is only possible when $j=\frac{n+1}{2}$.
\end{proof}

\begin{remark}
\label{rem: basepoint}
If $j=\frac{n+1}{2}$ then at least for $n=3$ the composition $\Phi_{w_i}\Psi_{z_i}$ is not zero.  Indeed, in the notation of Figure \ref{fig: T33} we get $\Psi_{z_1}(e)=d_1$ and $\Phi_{w_1}(d_1)=Z_1$, so $F_{m_1}(e)=e+Z_1$.
\end{remark}

\subsection{Cobordism maps $\Phi_1,\Phi_2,F_{-P}$: full homology}
Let $L=O_1(n,0)$ and $L'=O_1(n+1,1)$ and assume the $(n+2)$-nd component of $L'$ is oriented negatively. 
%Furthermore,
%assume $\HFL(L')$ be the link Floer homology of the $L'$ with the coefficient ring that identify $U_{n+1}=U_{n+2}$ and $V_{n+1}=V_{n+2}$.

%\AA{For $\Phi_1$ and $\Phi_2$ we don't want to identify both $U_{n+1}=U_{n+2}$ and $V_{n+1}=V_{n+2}$, right?}
\begin{proposition}
\label{prop: Phi1 full homology}
The cobordism maps 
\[\Phi_1:\HFL(O_1(n,0))\to \HFL\left(O_1(n+1,1);R_{n+2}/(U_{n+2}-U_{n+1})\right)\]
\[\Phi_2:\HFL(O_1(n,0))\to \HFL\left(O_1(n+1,1);R_{n+2}/(V_{n+2}-V_{n+1})\right)\]
 are given by 
\begin{equation}
\label{eq: phi guess T}
\Phi_1(X'_{j,n})=U_{n+1}X'_{j+1,n+2},\ \Phi_2(X'_{j,n})=V_{n+1}X'_{j+1,n+2}.
\end{equation}
Note that we denote the generators by $X'_{j,n}$, resp. $X'_{j,n+2}$ to keep track of the number of strands. 
\end{proposition}

\begin{proof}
It follows from Lemma \ref{lem: identify U not V} that 
\[\HFL\left(O_1(n+1,1);R_{n+2}/(U_{n+2}-U_{n+1})\right)\cong \HFL(O_1(n+1),1)/(U_{n+1}-U_{n+2})\] in level $0$. Since 
 $O_1(n,0)$ and $O_1(n+1,1)$ are both L-space links, the map $\Phi_1$ is completely determined by its Alexander and Maslov grading shifts, up to a scalar, and is nonzero by Lemma \ref{lem:relations}. 
We can use Lemma \ref{lem: change orientation n n} to verify that the gradings indeed agree. The proof for $\Phi_2$ is analogous. 
\end{proof}

Note that if we identify both $U_{n+1}=U_{n+2}$ and $V_{n+1}=V_{n+2}$ in $\CFL(O_1(n+1,1))$, we will have non-trivial homology in levels higher than $0$. However, since $\Phi_1$ and $\Phi_2$ are both level preserving (Proposition \ref{prop:level-preserving}) they are given by Equation \eqref{eq: phi guess T}.

\begin{proposition}
\label{prop: Phi full homology} After identifying $U_{n+1}=U_{n+2}$ and $V_{n+1}=V_{n+2}$ we have 
\begin{equation}
\label{eq: def Phi hat}
F_{-P}(X'_{j,n+2})=\begin{cases}
V_1\cdots V_n X'_{1,n}\otimes T & j=1,\\
X'_{j-1,n}\otimes T & 1<j<n+2\\
U_1\cdots U_nX'_{n,n}\otimes T & j=n+2.
\end{cases}
\end{equation}
\end{proposition}
\begin{proof}
By Lemma \ref{lem:relations} the compositions of $F_{-P}$ with $\Phi_1$ and $\Phi_2$ are given by $U_{n+1}F_{B}$ and $V_{n+1}F_{B}$, where $F_{B}$ is the birth map. So, by \eqref{eq: phi guess T} for any $1<j<n+2$ we have
\[F_{-P}(U_{n+1}X'_{j,n+2})=U_{n+1}X'_{j-1,n}\otimes T.\]
Since $\HFL(O_1(n,0))$ has no $U_{n+1}$-torsion we have $F_{-P}(X'_{j,n+2})=X'_{j-1,n}\otimes T$. 

Furthermore, for $j=1$ and $j=n+2$, using the relations of $\HFL(O_1(n+1,1))$ and \eqref{eq: phi guess T} we have
\[F_{-P}(U_{n+1}X'_{1,n+2})=F_{-P}(V_1\cdots V_{n}U_{n+1}X'_{2,n+2})=U_{n+1}(V_1\cdots V_{n}X'_{1,n}\otimes T)\]
\[F_{-P}(U_{n+1}X'_{n+2,n+2})=F_{-P}(U_1\cdots U_{n+1}X'_{n+1,n+2})=U_{n+1}(U_1\cdots U_n X'_{n,n}\otimes T).\]
Again, the claim follows from the fact that $\HFL(O_1(n,0))$ has no $U_{n+1}$-torsion.\end{proof}

\begin{remark}
One can check algebraically that the maps \eqref{eq: phi guess T} and \eqref{eq: def Phi hat} preserve all relations \eqref{eq: HFL full relations}.
\end{remark}

\section{$(-1)$-framed unknot}
\label{sec: -1 framed unknot}

\subsection{Duality}

Given a link $L$, we define $\overline{L}$ as the mirror of $L$ with the orientations of all components reversed. In particular, we have $\overline{T(n,n)}=T(n,-n)$ and 
$$
\overline{K_{m}(k_{+},k_{-})}=\overline{K}_{-m}(k_{+},k_{-})
$$
for any oriented knot $K$. For the link $\overline{K}_{-m}(k_{+},k_{-})$, we have $k_+$ link components with the same orientation as $\overline{K}$ and $k_-$ link components with opposite orientation. 

Recall that by \cite[Theorem 1.1]{JZ} we have a canonical isomorphism
\begin{equation}
\label{eq: duality HFL}
\HFLhat(\overline{L})\simeq \HFLhat(L)^{\vee}[1-n]
\end{equation}
where $N^{\vee}$ denotes the graded dual of a vector space $N$ with all gradings reversed. More explicitly, 
$$\HFLhat_{M}(L,s)^{\vee}\simeq \HFLhat_{1-n-M}(\overline{L},-s).$$
For example, 
$$
\HFLhat(T(n,-n))=\HFLhat(T(n,n))^{\vee}[1-n].
$$

\begin{remark}
In \cite{BaldwinGrigsby} and \cite[Lemma 5.4]{McCulloch} the authors used a slightly different isomorphism
$$
\HFLhat_{M}(L,s)\simeq \HFLhat_{1-n-M+2|s|}(\overline{L},s)
$$
which agrees with ours after applying the symmetry
$$
\HFLhat_{1-n-M+2|s|}(\overline{L},s)\simeq \HFLhat_{1-n-M}(\overline{L},-s).
$$
However, it will be important for us to distinguish between the vector spaces and their duals, so we use the variant from \cite{JZ} instead. 
\end{remark}

The duality interacts well with the cobordism maps. Let $$\mathcal X=(W,\mathcal F):(S^3, L_1, \mathbf w_1, \mathbf z_1)\longrightarrow(S^3, L_2, \mathbf w_2, \mathbf z_2)$$ be a decorated link cobordism with $\mathcal F=(\Sigma,\mathcal A)$. The  dividing set $\mathcal A$ decomposes $\Sigma$ into the
subsurfaces $\Sigma_{\mathbf w}$ and $\Sigma_{\mathbf z}$. The turned-around decorated cobordism is
\[
\mathcal X^\vee=(W^\vee,\mathcal F^\vee):
(-S^3,- L_2, \mathbf w_2, \mathbf z_2)\longrightarrow(-S^3,- L_1, \mathbf w_1, \mathbf z_1),
\]
where $W^\vee$ is the same oriented $4$-manifold as $W$, with the
incoming and outgoing ends interchanged, and $\mathcal F^\vee=(\Sigma,\mathcal A).$
In particular, the orientation of $\Sigma$, the dividing set
$\mathcal A$, the subsurfaces $\Sigma_{\mathbf w}$ and
$\Sigma_{\mathbf z}$, and the labels of the $w$- and $z$-basepoints
are unchanged. Only the roles of the incoming and outgoing ends are
reversed. By~\cite[Theorem~1.3]{JZ}, the induced map
is the dual map:
\[
F_{\mathcal X^\vee}=F_{\mathcal X}^{\vee}.
\]
Fix an orientation-reversing diffeomorphism
$\rho:S^3\to S^3$, and define $\overline{L}:=\rho(-L)$. Let 
\[
\iota:S^3\times[0,1]\longrightarrow S^3\times[0,1],
\qquad
\iota(x,t)=(\rho(x),1-t),
\]
and set $\overline\Sigma=\iota(\Sigma), \quad \overline{\mathcal A}=\iota(\mathcal A).$
We transport the orientation of $\Sigma$, the decomposition into
the $\mathbf w$- and $\mathbf z$-regions, and all basepoint labels
under $\iota$. Then $\overline{\mathcal X}=\bigl(S^3\times[0,1],
(\overline\Sigma,\overline{\mathcal A})\bigr)$  is a decorated cobordism from $(S^3, \overline{L}_2)$ to $(S^3, \overline{L}_1)$, and 
\[
F_{\overline{\mathcal X}}:
\HFLhat(\overline L_2)
\longrightarrow
\HFLhat(\overline L_1)
\]
corresponds to $F_{\mathcal X}^{\vee}$ under the duality
isomorphisms.

Then we have the following  commutative diagram
\begin{equation}
\label{eq: duality maps}
\begin{tikzcd}
\HFLhat(L_1)^{\vee}[1-n_1] \arrow{d}{\simeq}& \HFLhat(L_2)^{\vee}[1-n_2] \arrow{d}{\simeq} \arrow[swap]{l}{(F_{\mathcal{X}})^{\vee}}\\
\HFLhat(\overline{L}_1) & \arrow{l}{F_{\overline{\mathcal X}}} \HFLhat(\overline{L}_2).
\end{tikzcd}
\end{equation}

This implies the following. 

\begin{proposition}
\label{prop: dual cobordisms}
Let $K$ be a framed knot, fix $k_{+},k_{-}$ and $n=k_{+}+k_{-}$. We have the following dualities between the elementary cobordisms:

a) The pair-of-pants cobordism 
$$
P:K_m(k_{+},k_{-})\sqcup O\to K_m(k_{+}+1,k_{-}+1),
$$
is dual to the reverse pair-of-pants cobordism
$$
-P: \overline{K}_{-m}(k_{+}+1,k_{-}+1)\to \overline{K}_{-m}(k_{+},k_{-})\sqcup O.
$$

b) The birth cobordism
$$
B:K_{m}(k_{+},k_{-})\to K_{m}(k_{+},k_{-})\sqcup O
$$
is dual to the death cobordism 
$$
D:\overline{K}_{-m}(k_{+},k_{-})\sqcup O\to \overline{K}_{-m}(k_{+},k_{-}).
$$

c) The dot cobordism for $K_{m}(k_{+},k_{-})\sqcup O$ is dual to the dot cobordism for $\overline{K}_{-m}(k_{+},k_{-})\sqcup O$.
%{\AA{Added the definition of the dot map to section 3, definition \ref{def:dot map}}}
%Here the dot map $\mathcal B: \widehat{HFL}(L\sqcup O)\rightarrow \widehat{HFL}(L\sqcup O)$ is defined as 
%$$\mathcal B(x\otimes T)=x\otimes B, \quad \mathcal B(x\otimes B)=0. $$
\end{proposition}

% \begin{proposition}
% \label{prop: dual cobordisms}
% Let $K$ be a framed knot, fix $k_{+},k_{-}$ and $n=k_{+}+k_{-}$. We have the following dualities:

% a) The pair-of-pants cobordism
% $$
% f_P:\CFL(K_m(k_{+},k_{-})\sqcup O)\to \CFL(K_m(k_{+}+1,k_{-}+1))\otimes R_{n+1},
% $$
% is dual to the reverse pair-of-pants cobordism
% $$
% f_{-P}: \CFL(\overline{K}_{-m}(k_{+}+1,k_{-}+1))\otimes R_{n+1}\to \CFL(\overline{K}_{-m}(k_{+},k_{-})\sqcup O).
% $$
% where $R_{n+1}=R_{n+2}/(U_{n+2}-U_{n+1},V_{n+2}-V_{n+1})$.

% b) The birth cobordism
% $$
% f_B:\CFL(K_{m}(k_{+},k_{-}))\otimes R_{n+1}\to \CFL(K_{m}(k_{+},k_{-})\sqcup O)
% $$
% is dual to the death cobordism 
% $$
% f_D:\CFL(\overline{K}_{-m}(k_{+},k_{-})\sqcup O)\to \CFL(\overline{K}_{-m}(k_{+},k_{-}))\otimes R_{n+1}.
% $$

% c) The dot map $\mathcal B$ on 
% $
% \HFLhat\left(K_{m}(k_{+},k_{-})\sqcup O\right)
% $
% is dual to the dot map on $\widehat{HFL}\left(\overline{K}_{-m}(k_{+},k_{-})\sqcup O\right).$ {\AA{Added the definition of the dot map to section 3, definition \ref{def:dot map}}}
% \end{proposition}

\begin{lemma}
We have commutative diagrams
\begin{equation}
\label{eq: duality Phi undotted}
\begin{tikzcd}
\HFLhat(T(n+2,n+2))^{\vee}[-1-n] \arrow{d}{\simeq}& &
\HFLhat(T(n,n))^{\vee}[1-n] \arrow{d}{\simeq} \arrow[swap]{ll}{[F_D\circ F_{-P}]^{\vee}}\\
\HFLhat(T(n+2,-(n+2))) & & \arrow{ll}{F_P(-\otimes T)} \HFLhat(T(n,-n)).
\end{tikzcd}
\end{equation}
\begin{equation}
\label{eq: duality Phi}
\begin{tikzcd}
\HFLhat(T(n+2,n+2))^{\vee}[-1-n] \arrow{d}{\simeq}& &\HFLhat(T(n,n))^{\vee}[1-n] \arrow{d}{\simeq} \arrow[swap]{ll}{\Phi^{\vee}}\\
\HFLhat(T(n+2,-(n+2))) & & \arrow{ll}{F_P(-\otimes B)} \HFLhat(T(n,-n)).
\end{tikzcd}
\end{equation}
Here, $\Phi=F_D\circ \mathcal{B}\circ F_{-P}$ is the composition of $F_{-P}$ and the dotted death cobordism. 
\end{lemma}

\begin{proof}
This follows from Proposition \ref{prop: dual cobordisms} and the decompositions
$$
F_P(-\otimes T)=F_P\circ F_B,\ F_P(-\otimes B)=F_P\circ \mathcal{B}\circ  F_B.
$$
\end{proof}

The duality allows us to define a level grading in homology of $T(n,-n)$, and more generally mirrors of plumbed L-space links.
Let
$F_\bullet \simeq \CFL(T(n,n))$
be the minimal graded free resolution of
\(\HFL(T(n,n))\) over
$R_n=\mathbb{F}[U_1,\ldots,U_n,V_1,\ldots,V_n]$,
where \(F_k\) is placed in level \(k\).  The duality gives a
graded chain homotopy equivalence
\[
\CFL(T(n,-n))
\simeq
F_{\bullet}^{\vee}=\Hom_{R_n}(F_\bullet,R_n),
\]
up to the above Maslov and Alexander grading changes.  Since the
dual differential maps
\[
F_k^\vee \longrightarrow F_{k+1}^\vee,
\]
we place \(F_k^\vee\) in level \(n-k\).   After setting all the variables equal to zero we get 
\[
\HFLhat(T(n,-n))_{\lvl=n-k}
\simeq
\left(
\HFLhat(T(n,n))_{\lvl=k}
\right)^\vee[1-n].
\]
This implies the following:

\begin{corollary}
The braid group action in the diagonal part of $\HFLhat(T(n,-n))$ is level-filtered and level-preserving unless $j=\frac{n+1}{2}$ and $\lvl=n-M_j$. The level-preserving part of the action factors through $S_n$ and yields the following representations:

a) The dual bottom block decomposes into $V_{n,i}^{\vee}\simeq V_{n,n-1-i}$  in level $n-i$ and Maslov grading
$$
\left(\frac{\alpha+1}{2}-\jj\right)\left(\frac{\alpha-1}{2}-\jj\right)+i+1-n.
$$

b) The dual top block decomposes into $V_{n,n-i-1}^{\vee}\simeq V_{n,i}$  in level $i$ and Maslov grading
$$
\left(\frac{\alpha+1}{2}-\jj\right)\left(\frac{\alpha-1}{2}-\jj\right)-1-i.
$$

All these classes have total Alexander grading $-\jj\alpha$. 
\end{corollary}

\begin{proof}
Choose a homogeneous basis in $\HFLhat(T(n,n))$ and the dual basis in $\HFLhat(T(n,-n))$. Suppose that $\tau\in \Br_n$ and $x$ and $y$ are two basis vectors in $\HFLhat(T(n,n))$ such that $y$ appears in $\tau(x)$. Dually, $x^{\vee}$ appears in $\tau(y^{\vee})$.

If $\lvl(x)=\lvl(y)$ then $\lvl(x^{\vee})=\lvl(y^{\vee})$. Otherwise by Corollary \ref{cor: homotopic to level filtered} we can assume $\lvl(x)>\lvl(y)$, so 
$\lvl(x^{\vee})< \lvl(y^{\vee})$ and the action of $\tau$ on $\HFLhat(T(n,-n))$ is level-filtered. Furthermore, by Corollary \ref{cor: Maslov separation} this is only possible when $j=\frac{n+1}{2}$, $\lvl(x)=M_j$ and $\lvl(y)=m_j$, so 
$\lvl(x^{\vee})=n-M_j$ and $\lvl(y_j)=n-m_j$.

The rest of the proof follows from Remark \ref{rem: jj}, Lemma \ref{lem: Sn action diagonal blocks} and Lemma \ref{lem: Vni}(c). 
\end{proof}

Furthermore, we have the following.

\begin{lemma}
\label{lem: levels for dual maps}
With the above level gradings, the maps $F_P(-\otimes T)$ and $F_P(-\otimes B)$ for $T(n,-n)$ increase the levels at most by 1 and 2 respectively.
\end{lemma}

\begin{proof}
We choose a homogeneous basis in $\HFLhat(T(n,n))$ and the dual basis in $\HFLhat(T(n,-n))$. Suppose that $x$ is a class in $\HFLhat(T(n+2,n+2))$, $y_1$ is a basis vector appearing in $F_D\circ F_{-P}(x)$ and $y_2$ is a basis vector appearing in $F_D\circ \mathcal{B}\circ F_{-P}(x)$. Since $F_{-P}$ is level-filtered, $F_D$ decreases the level by 1 and $\mathcal{B}$ increases the level by 1, we get
$$
\lvl(y_1)\le \lvl(x)-1,\ \lvl(y_2)\le \lvl(x).
$$
In the dual basis, we have 
$$
\lvl(y_1^{\vee})=n-\lvl(y_1),\ \lvl(y_2^{\vee})=n-\lvl(y_2),\ \lvl(x^{\vee})=n+2-\lvl(x),
$$
so
$$
\lvl(x^{\vee})=n+2-\lvl(x)\le n+2-\lvl(y_1)-1=\lvl(y_1^{\vee})+1
$$
and similarly
$$
\lvl(x^{\vee})\le \lvl(y_2)+2.
$$
By \eqref{eq: duality Phi undotted} and \eqref{eq: duality Phi} $x^{\vee}$ appears in $F_P(y_1\otimes T)$ and $F_P(y_2\otimes B)$, and the result follows.
\end{proof}

\subsection{Cobordism maps for the bottom diagonal block}

In this subsection we compute the action of the cobordism map $F_{-P}$ on the bottom diagonal block for $O_1(k_++1,k_-+1)$ and the dual map $F_{P}$ for $O_{-1}(k_+,k_-)$. 

First we define some convenient basis in the bottom diagonal block of $\CFL(T(n,n))$. Recall that by Lemma \ref{lem: UV homotopy arc}, for any $a\in \{2, \cdots, n\}$, there is a chain homotopy $h_{1,a}$ between $U_a V_a$ and $U_1 V_1$ in $\CFL(T(n,n))$, that is 
$$[\partial_T, h_{1,a}]=(U_a V_a-U_1V_1) \mathrm{id}.$$

\begin{lemma}
\label{lem: unlink action bottom diagonal}
The bottom diagonal block at $A(X_j)$ at level $0\le i\le \min(j-1,n-j)$ has basis $Z_A^j$ indexed by subsets $A\subset \{2,\ldots,n\}$ with $|A|=i$. The level $i=0$ corresponds to $A=\emptyset$ and $Z^j_{\emptyset}=X_j$.
The action of the homotopies $h_{1,a}$ is given by 
$$
h_{1,a}Z^j_A=\begin{cases}
Z^j_{A\cup a} & \mathrm{if}\ a\notin A\\
0 & a\in A.
\end{cases}
$$
The action of $S_n$ permutes $Z^j_A$ similarly to \eqref{eq: Sn action V}.
\end{lemma}

\begin{proof}
For $a\in \{2, \cdots, n\}$, recall that $\HFLhat(O(n))\cong \Lambda(\xi_{ 2}, \cdots, \xi_{n})$ while in $\CFL(O(n))$ one has $\partial \xi_{a}=U_aV_a-U_1V_1$ for $a\in \{2, \cdots, n\}$. Consider the homomorphism $\mathfrak{F}_j: \HFL(O(n))\rightarrow \HFL(T(n,n))$ which sends $1$ to $X_j$. By Lemma \ref{lem: lifting to resolutions}, there is a level-preserving lift $\mathfrak{f}_j: \CFL(O(n))\rightarrow \CFL(T(n,n))$, and the lift is unique up to chain homotopy. %For the bottom diagonal of $\widehat{\HFL}\bigl(T_{n,n}(k_+,k_-)\bigr)$
We define
$$
Z^j_A=\widehat{\mathfrak{F}}_j (\xi_A)
$$
where $\widehat{\mathfrak{F}}_j$ denotes the induced map by $\mathfrak{f}_j$ from $\HFLhat(O(n))$ to $\HFLhat(T(n,n))$.
On the other hand, define $m_a: \CFL(O(n))\rightarrow \CFL(O(n))$ with $m_a(z)=\xi_{a} z$. Then again we have
$$[\partial_O, m_a]=U_aV_a-U_1V_1.$$
Hence, for  $D:=h_{1,a} \mathfrak{f}_j +\mathfrak{f}_j m_a$ we have $\partial_T D+D\partial_O=0$. By Corollary \ref{cor: homotopic to level filtered}, the component $D_1$ of $D$, which raises level by one,   is null-homotopic, that is, $h_{1,a} \mathfrak{f}_j$ is homotopic to $\mathfrak{f}_j m_a$. Setting all the variables to be 0, we get
$\widehat{h}_{1,a}\widehat{\mathfrak{F}}_j=\widehat{\mathfrak{F}}_jm_a$ 
and
$$
\widehat{h}_{1,a}Z^j_A=\begin{cases}
Z^j_{A\cup a} & \mathrm{if}\ a\notin A\\
0 & a\in A.
\end{cases}
$$
  
Finally, we need to prove that after specializing to $\HFLhat$ these classes are in the bottom diagonal block for  $|A|\le \min(j-1,n-j)$.
Recall that by \cite[Theorem 5]{GHom} the bottom diagonal block generates the diagonal block in $\HFLminus(T(n,n))$. Furthermore, the spectral sequence $\HFL(T(n,n))\Rightarrow \HFLminus(T(n,n))$ in  Alexander grading $A(X_j)$ is given by the deformed cube complex as in Lemma \ref{lem: deformed cube}, compare with the proofs of Lemmas \ref{lem: levels} and \ref{lem: Sn action diagonal blocks}.  
By Lemma \ref{lem: unknot action ss} the action of $h_{1,a}$ in $\HFLminus$ corresponds to the multiplication by $\xi_a$ in the setup of Lemma \ref{lem: deformed cube}, and the result follows.
\end{proof}

\begin{remark}\label{rmk:full-twist}
    The homomorphism $\mathfrak{F}_j$ is the same as the full twist map $\phi_{j-1}^n: \HFL(O(n))\rightarrow \HFL(T(n,n))$ considered in \cite[Example 3.12]{AGL2}. 
\end{remark}

Under the change of orientation, we get a similar basis for the bottom diagonal block of $\CFL(O_1(k_+,k_-))$. Specifically, if $\mathcal{P}$ and $\mathcal{N}$ denote the positively and negatively oriented components of $O_1(k_+,k_-)$, we assume $1\in \mathcal{P}$ and denote $Z^j_{A}=\widehat{\mathfrak{F}}_j(\xi_{A_+}\eta_{A_-})$, where $A_+=A\cap\mathcal{P}$ and $A_-=A\cap \mathcal{N}$.
%\AA{$B$ has been used for the bottom generator, so I swiched from $A$, $B$ to $A_+$ and $A_-$}
 
%\EG{add the range of $j$.}

\begin{lemma}
\label{lem: F minus P}
The reverse pair-of-pants map 
$$
F_{-P}:\HFLhat(O_1(k_{+}+1,k_{-}+1))\to \HFLhat(O_1(k_{+},k_{-})\sqcup O)
$$ 
is given on the bottom diagonal block by the equations
\begin{align*}
& F_{-P}\left(Z_A^{j+1}\right)=Z_A^j\otimes T\\
& F_{-P}\left(Z_{A\cup \{n+1\}}^{j+1}\right)=  F_{-P}\left(Z_{A\cup \{n+2\}}^{j+1}\right)=Z_A^j\otimes B\\
& F_{-P}\left(Z_{A\cup \{n+1, n+2\}}^{j+1}\right)=0.
\end{align*}
Here $A\subset \{2,\ldots,n\}$ and $1\le j\le n$, so that $2\le j+1\le n+1$.
\end{lemma}

\begin{remark}
Lemma \ref{lem: F minus P} does not apply to $F_{-P}\left(Z_{A}^{1}\right)$ and $F_{-P}\left(Z_{A}^{n+2}\right)$. However, for $j=1$ and $j=n+2$ the inequality $|A|\le \min(j-1,n+2-j)$ implies $A=\emptyset$, and Proposition \ref{prop: Phi full homology} implies that $F_{-P}(X'_{1,n+2})$ and $F_{-P}(X'_{n+2,n+2})$ vanish in $\HFLhat$.
\end{remark}

\begin{proof}
Let $\mathcal{P}$ and $\mathcal{N}$ denote the positively and negatively oriented components of $O_1(k_+,k_-)$. Suppose $\mathcal{P}\sqcup\{n+1\}$ and $\mathcal{N}\sqcup\{n+2\}$ are the positively and negatively oriented components of $O_1(k_++1,k_-+1)$, respectively. Consider $O_0(k_+,k_-)$ (resp. $O_0(k_++1,k_-+1)$) such that its positively oriented components are exactly $\mathcal{P}$ (resp. $\mathcal{P}\sqcup\{n+1\}$).

Consider the full chain complex 
$$
\CFL(O_1(k_{+},k_{-})\sqcup O)\simeq \CFL(O_1(k_{+},k_{-}))\otimes \left[B\xrightarrow{U_{n+1}V_{n+1}-U_1V_1} T\right]. 
$$
%Recall that we identify $U_{n+1}=U_{n+2}$ and $V_{n+1}=V_{n+2}$. 
We extend the chain map 
\[\mathfrak{f}^n_j:\CFL(O_0(k_+,k_-))\to\CFL(O_1(k_+,k_-))\] from the proof of Lemma \ref{lem: unlink  action bottom diagonal} to a chain map 
\[
\mathfrak{f}_j^{n,o}:\CFL(O_0(k_++1,k_-+1)/(U_{n+2}-U_{n+1},V_{n+2}-V_{n+1})\to \CFL(O_1(k_{+},k_{-})\sqcup O)
\]
by setting
\begin{align*}
& \mathfrak{f}_j^{n,o}(1)=X'_{j,n}\otimes T,  \quad     \mathfrak{f}_j^{n,o}\left(\xi_{A_+}\eta_{A_-}\right)=\mathfrak{f}_j(\xi_{A_+}\eta_{A_-})\otimes T\\
& \mathfrak{f}_j^{n,o}\left(\xi_{A_+\cup\{n+1\}}\eta_{A_-}\right)=\mathfrak{f}_j^{n,o}\left(\xi_{A_+}\eta_{A_-\cup\{n+2\}}\right)=\mathfrak{f}_j(\xi_{A_+}\eta_{A_-})\otimes B\\
& \mathfrak{f}_{j}^{n,o}\left(\xi_{A_+\cup\{n+1\}}\eta_{A_-\cup\{n+2\}}\right)=0.
\end{align*}
for  $A\subset \{2,\ldots,n\}$. 
%\AA{Why is this a chain map?}
We also have the chain map 
$$
\mathfrak{f}^{n+2}_{j+1}:\CFL(O_0(k_++1,k_-+1)\to \CFL(O_1(k_{+}+1,k_{-}+1))
$$
which is a lift of $\mathfrak{F}^{n+2}_{j+1}: \HFL(O_0(k_++1,k_-+1)\rightarrow \HFL(O_1(k_{+}+1,k_{-}+1))$ sending $1$ to $X'_{j+1, n+2}$.
By Proposition \ref{prop: Phi full homology} 
the chain maps $f_{-P}\circ \mathfrak{f}_{j+1}^{n+2}$ and $\mathfrak{f}_{j}^{n,o}$ induce the same map in homology, so the level-preserving components are homotopic. Note that the level $i$ elements in the bottom block of $X'_{j,n}$ have Maslov grading $M(X'_{j,n})-i$, and the level $r$ elements in the bottom (resp. top) block of $X'_{j+1, n+2}$ have Maslov grading $M(X'_{j,n})-r$ (resp. $M(X'_{j,n})-r+2$). Since $F_{-P}$ preserves the Maslov grading and is a level-filtered map, its image has to be level $i$ in the bottom block of $X'_{j+1, n+2}$. Hence $F_{-P}$ is level-preserving. So $f_{-P}\circ \mathfrak{f}_{j+1}^{n+2}$ and $\mathfrak{f}_{j}^{n,o}$ are homotopic and induce the same map on $\widehat{HFL}$, thus the result follows. 
\end{proof}

\begin{remark} 
One can use Remark \ref{rmk:full-twist} and compute $F_{-P}$ on the bottom diagonal block topologically.  By functoriality of link Floer homology we have a commutative diagram
\[\begin{tikzcd}
\HFLhat(O_0(k_++1,k_-+1)) \arrow[r, "F_{-P}"] \arrow[d, "\widehat{\mathfrak{F}}_{j+1}"'] & \HFLhat(O_0(k_+,k_-)\sqcup O) \arrow[d, "\widehat{\mathfrak{F}}_{j}\otimes \mathrm{id}_V"] \\
\HFLhat(O_1(k_++1,k_-+1)) \arrow[r, "F_{-P}"']                 & \HFLhat(O_1(k_+,k_-)\sqcup O)
\end{tikzcd}\]
Note that by \cite[Figure 17]{JM} the map $F_{-P}$ in the top row is given by 
\[F_{-P}(\xi_{A_+}\eta_{A_-})=\begin{cases} \xi_{A_+}\eta_{A_-}\otimes T&n+1\notin A_+,\ n+2\notin A_-\\
\xi_{A_+}\eta_{A_-}\otimes B&n+1\in A_+,\ n+2\notin A_-\\
\xi_{A_+}\eta_{A_-}\otimes B&n+1\notin A_+,\ n+2\in A_-\\
0&n+1\in A_+,\ n+2\in A_-\\
\end{cases}\]
Thus, the Lemma \ref{lem: F minus P} follows from the commutativity of the diagram and the definition of the vertical maps.
\end{remark}

Let $\left(Z^j_A\right)^{\vee}$ denote the dual basis to $Z^j_A$ in the dual bottom block for $O_{-1}(k_{+},k_{-})$.

\begin{lemma}
\label{lem: Fp for dual bottom}
For $A\subset \{2,\ldots,n\}$ and $1\le j\le n$ we have
$$
F_{P}\left(\left(Z_A^j\right)^{\vee}\otimes T\right)=\left(Z_{A\cup \{n+1\}}^{j+1}\right)^{\vee}+\left(Z_{A\cup \{n+2\}}^{j+1}\right)^{\vee},
$$
modulo the off-diagonal and dual top diagonal block, and
$$
F_{P}\left(\left(Z_A^j\right)^{\vee}\otimes B\right)=\left(Z_{A}^{j+1}\right)^{\vee}.
$$
This matches the formulas for the 0-framed unknot after identifying $\left(Z_{A}^j\right)^{\vee}$ in $O_{-1}(k_{+},k_{-})$  (resp. $\left(Z_{B}^{j+1}\right)^{\vee}$ in $O_{-1}(k_{+}+1,k_{-}+1)$) with $\xi_{\mathcal{P}\setminus A_+}\eta_{\mathcal{N}\setminus A_-}$
%$\xi_{\{2,\ldots,n\}\setminus A}$ 
(resp.  $\xi_{\mathcal{P}'\setminus B_+}\eta_{\mathcal{N}'\setminus B_-}$), where $\mathcal{P}'=\mathcal{P}\sqcup \{n+1\}$, $\mathcal{N}'=\mathcal{N}\sqcup  \{n+2\}$. 
\end{lemma}

\begin{proof}
After applying the death cobordism to $F_{-P}$ computed in Lemma \ref{lem: F minus P}, we get
\begin{align*}
& F_{D}\circ F_{-P}\left(Z_A^{j+1}\right)=0\\
& F_{D}\circ F_{-P}\left(Z_{A\cup \{n+1\}}^{j+1}\right)=F_{D}\circ F_{-P}\left(Z_{A\cup \{n+2\}}^{j+1}\right)=Z^j_A\\
& F_{D}\circ F_{-P}\left(Z^{j+1}_{A\cup \{n+1, n+2\}}\right)=0.
\end{align*}
Similarly, after applying the dotted death cobordism we get
\begin{align*}
& F_D\circ \mathcal{B}\circ F_{-P}\left(Z^{j+1}_A\right)=Z^j_A\\
& F_D\circ \mathcal{B}\circ F_{-P}\left(Z^{j+1}_{A\cup \{n+1\}}\right)=F_D\circ \mathcal{B}\circ F_{-P}\left(Z^{j+1}_{A\cup \{n+2\}}\right)=0\\
& F_D\circ \mathcal{B}\circ F_{-P}\left(Z^{j+1}_{A\cup \{n+1,n+2\}}\right)=0.
\end{align*}
Now the result follows from \eqref{eq: duality Phi undotted} and \eqref{eq: duality Phi}.
\end{proof}

\subsection{Cabled link Floer homology}

We now compute the cabled  link Floer homology of the
$(-1)$-framed unknot. 
\begin{lemma}
\label{lem: Fp from off-diagonal}
The map
$$
F_P(-\otimes B):\HFLhat(O_1(k_{+},k_{-}))\to \HFLhat(O_1(k_{+}+1,k_{-}+1))
$$
vanishes on the off-diagonal part of $\HFLhat$ and the image of 
$$
\Phi:\HFLhat(O_1(k_{+}+1,k_{-}+1)
)\to \HFLhat(O_1(k_{+},k_{-}))
$$
does not contain any off-diagonal terms.
%vanishes on the off-diagonal part of $\HFLhat$. \AA{Do you mean the image does not contain any off-diagonal terms?}
\end{lemma}

\begin{proof}
Let $x$ be a class in the off-diagonal part of $\HFLhat(O_1(k_{+},k_{-}))$, with parameters $j$ and $s$. Up to a sign, its Alexander grading has $s$ components equal to 
$$
\frac{n+1}{2}-j-1=\frac{(n+2)+1}{2}-(j+1)-1
$$
and $n-s$ components equal to 
$$
\frac{n+1}{2}-j=\frac{(n+2)+1}{2}-(j+1).
$$
As in Lemma \ref{lem: off diagonal change orientation}, let $s_-$ denote the number of components with grading $-(\frac{n+1}{2}-j-1)$.

%we change signs for $s_{-}$ and $k_{-}-s_{-}$ of these respectively.

The map $F_P(-\otimes B)$ preserves the $n$-component Alexander grading $(A_1,\ldots,A_n)$ and has {\bf total} grading zero for the two new components: $A_{n+1}+A_{n+2}=0$. Thus, $F_P(x\otimes B)$ is off-diagonal and for any term $y$ in $F_P(x\otimes B)$ we have two cases for $A_{n+1}(y)$ and $A_{n+2}(y)$.

{\bf Case 1:} 
$$A_{n+1}(y)=-A_{n+2}(y)=\frac{(n+2)+1}{2}-(j+1)-1.$$ This means that $y$ is in an off-diagonal block with parameters $(j+1,s+2,s_-+1)$ and so 
%$$
%(j,k_{-},s,s_{-})\mapsto (j+1,k_{-}+1,s+2,s_{-}+1)
%$$
and  by \eqref{eq: off diagonal change orientation}, we have 
\[M(y)=-(j-k_-)^2-(s+2)-(k_-+1)+2(s_-+1)+1=M(x)-1.\]
%$$
%-2-1+2=-1.
%$$
However, $F_P(-\otimes B)$ drops Maslov grading by $2$, so this is not possible. 
 
{\bf Case 2:}  
$$
A_{n+1}(y)=-A_{n+2}(y)=\frac{(n+2)+1}{2}-(j+1).$$ 
So $y$ is in an off-diagonal block with parameters $(j+1,s,s_-)$ and therefore by \eqref{eq: off diagonal change orientation}, we have 

%$$
%(j,k_{-},s,s_{-})\mapsto (j+1,k_{-}+1,s,s_{-})
%$$
%and by \eqref{eq: off diagonal change orientation}, its Maslov grading changes by 
\[M(y)=-(j-k_-)^2-s-(k_-+1)+2s_-+1=M(x)-1.\]

%$$
%-0-1+0=-1.
%$$
Again, since $F_P(-\otimes B)$ drops Maslov grading by $2$, this is not possible. 

Consequently, $F_{P}(-\otimes B)$ vanishes on the off-diagonal part. The proof for $\Phi$ is similar.
\end{proof}

Now we are ready to prove Theorem \ref{thm: unknot negative intro}. We first show that only the dual-bottom blocks contribute to the
cabled homology.

\begin{lemma}\label{lem:negative-nonbottom-vanishes}
The off-diagonal and dual-top parts represent zero in the cabled hat
link Floer homology of the $(-1)$-framed unknot, that is for $r\gg 0$ the $r$-fold iterate of the pair of pants map $F_P(-\otimes B)$ vanishes on the dual-top and off-diagonal part of $\HFLhat(O_{-1}(k_+,k_-))$.
\end{lemma}

\begin{proof}
Let
\[
G_B:=F_P(-\otimes B): \HFLhat(O_{-1}(k_+,k_-))\longrightarrow \HFLhat(O_{-1}(k_{+}+1, k_{-}+1))
\]
By \eqref{eq: duality Phi}, the map $G_B$ is identified with
$\Phi^\vee$.

By Lemma \ref{lem: Fp from off-diagonal}, the image of $\Phi$ does not contain any off-diagonal terms. Therefore, the dual map $G_B$ vanishes on the dual off-diagonal part. 

Next, we consider the dual-top block. Let $x$ be a class in the dual-top block with Alexander grading $-A(X'_{j,n})$. Since $G_B$ preserves $A_1,\cdots, A_n$, every diagonal term in $G_B(x)$ must have Alexander grading $-A(X'_{j+1,n+2})$. The lowest level of the top diagonal block of $\HFLhat(O_1(k_+, k_-))$ at $A(X_{j,n}')$ is  $M_{j,n}=\max (n-j+2, j+1)$.  For $O_1(k_{+}+1, k_{-}+1)$, the lowest level of the top diagonal block at $A(X_{j+1,n+2}')$ is $M_{j+1,n+2}=M_{j,n}+1$. Therefore, the highest level for the dual top block with Alexander grading $-A(X_{j,n}')$ equals $n-M_{j,n}$, while the highest level for the dual top block with Alexander grading $-A(X_{j+1,n+2}')$ is equal to 
$$
(n+2)-M_{j+1,n+2}=n-M_{j,n}+1,
$$
see Figure \ref{fig: pants}. 

By Lemma \ref{lem: levels for dual maps} the map $G_B$ increases the level at most by $2$. We claim that on the dual top block it increases the level exactly by 2. Indeed, the lowest level in the dual bottom block at  $-A(X_{j+1,n+2}')$ equals 
$$
n+2-m_{j+1,n+2}=n+1-m_{j,n}\ge n+3-M_{j,n}.
$$
Therefore if $x$ is in the dual top block then $G_B(x)$ cannot be in the dual bottom block and must be in the dual top block. By considering Maslov gradings and applying Corollary \ref{cor: Maslov separation} we get $\lvl(G_B(x))=\lvl(x)+2$.

So if $\lvl(x)=\ell$ then $r$-fold iterate $G^r_B(x)$ is in the level $\ell+2r$. The highest level of the dual top diagonal block at $-A(X'_{j+r,n+2r})$ is
 $$
 (n+2r)-M_{j+r,n+2r}=(n+2r)-(M_{j,n}+r)=n+r-M_{j,n}.
 $$ 
 For $r\gg 0$, we have 
 $$ 
 \ell+2r>n+r-M_{j,n},
 $$
and so if $G_{B}^r(x)$ is non-zero then it has to be in the off-diagonal block. Since, $G_B$ vanishes on the off-diagonal block, $G_{B}^r(x)=0$ for $r\gg 0$.

\end{proof}

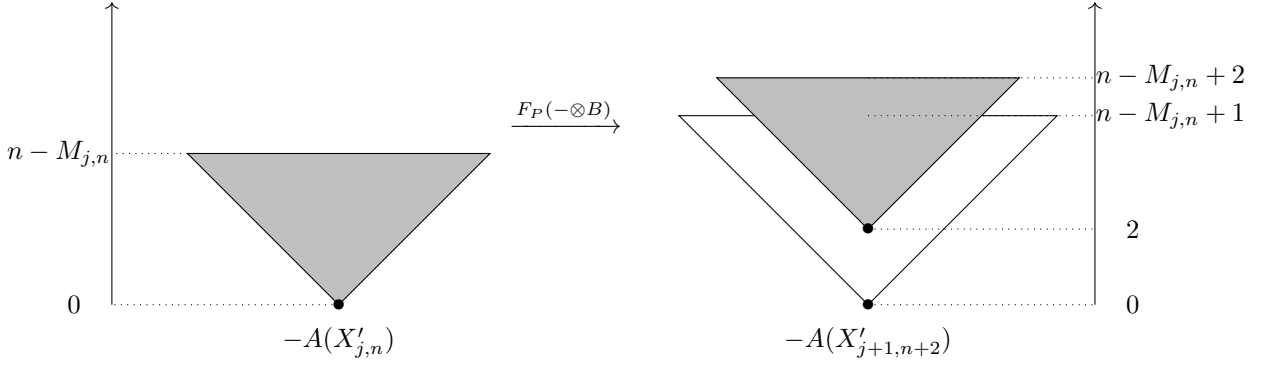
\begin{figure}[ht!]
\begin{tikzpicture}
\filldraw[color=lightgray] (0,0)--(-2,2)--(2,2)--(0,0);
\draw (0,0)--(-2,2)--(2,2)--(0,0);
%\filldraw[color=lightgray] (0,4)--(-1,3)--(1,3)--(0,4);
%\draw (0,4)--(-1,3)--(1,3)--(0,4);
\draw[->] (-3,0)--(-3,4);
\draw [dotted] (0,0)--(-3,0);
\draw [dotted] (-2,2)--(-3,2);
%\draw [dotted] (-1,3)--(-3,3);
%\draw [dotted] (0,4)--(-3,4);
\draw (-3.5,0) node {$0$};
\draw (-3.7,2) node {$n-M_{j,n}$};
%\draw (-3.5,3) node {$M_j$};
%\draw (-3.5,4) node {$n$};
\draw (0,0) node {$\bullet$};
%\draw (0,4) node {$\bullet$};
\draw (0,-0.5) node {$-A(X'_{j,n})$};
%\draw (0,4.5) node {$D'_{j,n}$};

\draw (3,2.5) node {$\xrightarrow{F_P(-\otimes B)}$};

\draw (7,0)--(4.5,2.5)--(9.5,2.5)--(7,0);
\filldraw[color=lightgray] (7,1)--(5,3)--(9,3)--(7,1);
\draw (7,1)--(5,3)--(9,3)--(7,1);

%\draw (7,5)--(5.5,3.5)--(8.5,3.5)--(7,5);
%\filldraw[color=lightgray] (7,5)--(6,4)--(8,4)--(7,5);
%\draw (7,5)--(6,4)--(8,4)--(7,5);
 
\draw (7,0) node {$\bullet$};
\draw (7,1) node {$\bullet$};
%\draw (7,5) node {$\bullet$};
\draw (7,-0.5) node {$-A(X'_{j+1,n+2})$};
%\draw (7,5.5) node {$D'_{j+1,n+2}$};

\draw[->] (10,0)--(10,4);

\draw[dotted] (7,0)--(10,0);
\draw (10.5,0) node {$0$};
\draw[dotted] (7,1)--(10,1);
\draw (10.5,1) node {$2$};
\draw[dotted] (7,2.5)--(10,2.5);
\draw (11,2.5) node {$n-M_{j,n}+1$};
\draw[dotted] (7,3)--(10,3);
\draw (11,3) node {$n-M_{j,n}+2$};
%\draw[dotted] (7,3.5)--(10,3.5);
%\draw (10.7,3.5) node {$M_j+1$};
%\draw[dotted] (7,4)--(10,4);
%\draw (10.7,4) node {$M_j+2$};
%\draw[dotted] (7,5)--(10,5);
%\draw (10.7,5) node {$n+2$};

\end{tikzpicture}
\caption{Action of $F_P(-\otimes B)$ on dual top diagonal blocks, here $M_{j,n}=\max(n-j+2,j+1)$.  
}
\label{fig: pants}
\end{figure}

 Fix $\alpha=k_+-k_-$ and let $n=k_++k_-$, then $n$ has the same parity as $\alpha$.  We follow the notation of Remark \ref{rem: jj} and index the diagonal
blocks by
\[
\jj=\dfrac{n+1}{2}-j\in\mathcal J_\alpha,
\qquad
\mathcal J_\alpha=
\begin{cases}
\mathbb Z, & \alpha\text{ odd},\\[2mm]
\mathbb Z+\frac12, & \alpha\text{ even},
\end{cases}
\]
so that the total Alexander grading of the block labeled by $\jj$ equals
$\jj\alpha$ for $O_1(k_{+},k_{-})$ and $-\jj\alpha$ for $O_{-1}(k_{+},k_{-})$.  Set
\[
Q_\alpha(\jj)
:=
\left(\frac{\alpha+1}{2}-\jj\right)
\left(\frac{\alpha-1}{2}-\jj\right).
\]

\begin{proof}[Proof of Theorem \ref{thm: unknot negative intro}] By Lemma \ref{lem:negative-nonbottom-vanishes} we only need to consider the dual bottom diagonal block. Moreover, the Alexander and level gradings imply that the sum of the off-diagonal and dual-top parts is preserved by the map
$F_P(-\otimes T)$.  In particular, the relations arising from these parts do not impose any additional relation on the dual-bottom quotient.

We first consider the relation \eqref{eq: lasagna 3} with the basepoint moving maps $F_{m_i}$. If $v$ is a class in the dual bottom block then similar to the proof of  Lemma \ref{lem: basepoint moving} the only nontrivial terms in $F_{m_i}(v)-v$ could be in the dual top diagonal block and vanish in the limit. 

Next, we consider the equivalence relations \eqref{eq: lasagna 1} and \eqref{eq: lasagna 2}. The action of the braid group on the dual bottom block factors through $S_n$ (modulo dual top block) and as $S_n$ representation it has summands $V_{n,i}^{\vee}\simeq V_{n,n-1-i}$  in level $n-i$ and Maslov grading
$
Q_\alpha(\jj)+i+1-n,
$
so the normalized Maslov grading equals
$
Q_\alpha(\jj)+i+1.
$
Note that $0\le i\le \min(j-1,n-j)$, so by increasing $n$ we can assume that $i$ is arbitrarily large. 
Lemma \ref{lem: Fp for dual bottom} shows that, up to the overall grading shift, the system of dual-bottom blocks is identical to the truncation of the cube system used in the computation of the $0$-framed unknot in
Section \ref{sec: unknot}.  %We may assume that
%$k_+>0$ and $k_->0$.  Choose a distinguished positively oriented
%strand, and use the basis
%\[
%\xi_2,\ldots,\xi_{k_+},
%\eta_1,\ldots,\eta_{k_-}
%\]
We use the basis of monomials $\xi_{A_+}\eta_{A_-}$ with $|A_+|+|A_-|=n-i-1$
as in Section \ref{sec: unknot}.  
Modulo the action of $S_{k_+}\times S_{k_-},$ a monomial is represented by a class $g_n(a,b)$, where $a=|A_+|, b=|A_-|$   and $a+b=n-i-1$. These classes satisfy
\eqref{eq: equivalence}.
The two maps in Lemma \ref{lem: Fp for dual bottom} become
\begin{align}
F_P\bigl(g_n(a,b)\otimes B\bigr)
&=
g_{n+2}(a+1,b+1),
\label{eq:negative-g-B}\\
F_P\bigl(g_n(a,b)\otimes T\bigr)
&=
g_{n+2}(a,b+1)
+
g_{n+2}(a+1,b).
\label{eq:negative-g-T}
\end{align}
Notice that \eqref{eq:negative-g-B} preserves $i$, while
\eqref{eq:negative-g-T} increases $i$ by one.

The braid relations, together with $F_P(g_n(a,b)\otimes T)=0$ and $g_n(a,b)\sim G_B(g_n(a,b)),$ are therefore exactly the relations appearing in the proof of
Theorem \ref{thm: zero framed unknot}. Hence the proof of Theorem \ref{thm: zero framed unknot} applies verbatim, and nontrivial homology classes in cabled Floer homology correspond to $g(a,b)$ where $n-i-1$ is even or $n-i-1=n-1$.
For $n$ even this implies $i\in \{0,1,3,\ldots\}$ and for $n$ odd we get $i\in \{0,2,4,\ldots\}$.
\end{proof}

\section{1-framed unknot}
\label{sec: 1 framed unknot}

\subsection{Cabled link Floer homology}

To formulate our main result about the 1-framed unknot, we recall the notation $\jj=\frac{n+1}{2}-j$ from Remark \ref{rem: jj}. The map $F_P$ increases $n$ by $2$ and $j$ by $1$, so it does not change $\jj.$
In Appendix \ref{sec: appendix A}, we give a conjectural algebraic model for the complex $\CFL(T(n,n))$, see Conjecture \ref{conj: algebraic complex}. Assuming this conjecture, we can compute the Floer lasagna module of $\mathbb{CP}^2$ as follows.

\begin{theorem}
\label{thm: 1 framed unknot}
1) The cabled Heegaard Floer homology of the 1-framed unknot vanishes in off-diagonal and bottom diagonal blocks.

2) Assuming Conjecture \ref{conj: algebraic complex}, the cabled Heegaard Floer homology of the 1-framed unknot  does not vanish in the top diagonal block, and is described as follows:
\begin{enumerate}
\item[(a)]  For $\alpha$ odd and $\jj\in \Z$ the space $\FL(\mathbb{CP}^2,\emptyset;\alpha)$ has generators in Alexander grading $\jj\alpha$ and normalized Maslov gradings 
$$
-\left(\frac{\alpha+1}{2}-\jj\right)\left(\frac{\alpha-1}{2}-\jj\right)+2+\{0,2,4,\ldots\}
$$
\item[(b)] For $\alpha$ even and $\jj\in \Z+\frac{1}{2}$ the space $\FL(\mathbb{CP}^2,\emptyset;\alpha)$ has generators in Alexander grading $\jj\alpha$ and normalized Maslov gradings  
$$
-\left(\frac{\alpha+1}{2}-\jj\right)\left(\frac{\alpha-1}{2}-\jj\right)+2+\{0,1,3,\ldots\}
$$
\end{enumerate}
Here $\alpha=k_{+}-k_{-}$ corresponds to the grading by $H_2(W_1(O))=H_2(\mathbb{CP}^2)$.
\end{theorem}

\begin{proof}
We consider separately the off-diagonal terms, the top diagonal block and the bottom diagonal block. Recall that we have the equivalence relation $v\sim F_P(v\otimes B)$.

1) The off-diagonal terms vanish after applying $F_P(-\otimes B)$ by Lemma \ref{lem: Fp from off-diagonal}. 

2) For the bottom diagonal block, we argue similarly to Lemma \ref{lem:negative-nonbottom-vanishes}. Note that the map $F_P(-\otimes B)$ increases $j$ by 1 and the level at most by 2.  Similarly to Lemma \ref{lem:negative-nonbottom-vanishes}, we can prove that it increases the level exactly by 2 on the bottom diagonal block, and can be presented similarly to Figure \ref{fig: pants}. If we start with a class $v$ at level $\ell$ and apply $F_P(-\otimes B)$ $r$ times then we get to level $\ell+2r$ and Alexander grading  $A(X'_{j+r,n+2r})$, but for $r\gg 0$ we have $\ell+2r>j+r>\min(j+r-1,n-j+r)$, and such a generator vanishes. Therefore after applying $F_P(-\otimes B)$ to any generator in the bottom block sufficiently many times we get zero. 
 
3) For the top diagonal block, we first consider the relation \eqref{eq: lasagna 3} with the basepoint moving maps $F_{m_a}$. If $v$ is a class in the top diagonal block then by Lemma \ref{lem: basepoint moving} the only nontrivial terms in $F_{m_a}(v)-v$ could be in the bottom diagonal block and vanish in the limit by the above.

Next, we consider \eqref{eq: lasagna 1} and \eqref{eq: lasagna 2}.
By Theorem \ref{thm: Fp top diagonal} for each $\jj$ the maps $F_P(-\otimes T)$ and $F_P(-\otimes B)$ coincide with the respective maps for the $0$-framed unknot (modulo the bottom diagonal block),  except possibly for the lowest level $n-\min(j-2,n-j-1)$. 
Note that Theorem \ref{thm: Fp top diagonal} only describes $F_{P}(-\otimes B)$ modulo the image of $F_{P}(-\otimes T)$, but the equivalence relation \eqref{eq: lasagna 2} allows one to take quotient by the  image of $F_{P}(-\otimes T)$.

Here we identify the level $n-i$ of the top diagonal block of $T(n,n)$ with the level $n-i-1$ for the $n$-component unlink and $i<\min(j-2,n-j-1)$. As we increase $n$ to $n+2$ and $j$ to $j+1$, the bound $\min(j-2,n-j-1)$ increases by 1 and hence in the limit $i$ could be arbitrary large. Recall that by Lemma \ref{lem: Sn action diagonal blocks} the $S_n$-representation at level $n-i$ is isomorphic to $V_{n,n-i-1}$ and hence by part (b) of Lemma \ref{lem: Vni} it has nontrivial (and one-dimensional) coinvariants for $n-i-1$ even or $i=0$.

By Theorem \ref{thm: zero framed unknot} we get exactly one generator in $\cHFL$ for $n-i-1$ even or $i=0$.
For $n$ odd (equivalently, $\alpha$ odd) this implies $i=0,2,4,\ldots$ while for $n$ even we get $i=0,1,3,\ldots$. 

By Remark \ref{rem: jj} the normalized Maslov grading of such generators equals 
$$
-\left(\frac{\alpha+1}{2}-\jj\right)\left(\frac{\alpha-1}{2}-\jj\right)+2+i.
$$
After applying $F_P(-\otimes B)$ to the lowest level it increases by 2 and becomes
$$
\max(n-j+2,j+1)+2>\max(n-j+2,j+1)+1=\max((n+2)-(j+1)+2,(j+1)+1).
$$
Therefore the new level is no longer lowest and the above arguments apply. This completes the proof.
\end{proof}

\section{More general L-space knots}
\label{sec: L space knots}

In this section we generalize these results to other L-space knots. Let $K$ be an L-space knot with framing $m>\max\{0, 2g(K)-1\}$. Then by \cite{GHom} for all $n$ the cable $K_{n,mn}$ is an L-space link. We furthermore assume that $K$ is a plumbed knot. Then $K_{n,mn}$ is a plumbed L-space link and the results of \cite{BLZ} apply to it. In the next subsections we explain how the above computations for $T(n,n)$ generalize to this case.

\subsection{Full homology}

Each component $L_i$ of $K_{n,mn}$ is a copy of $K$, and each pair of components has linking number $m$, so that $\ell_i=m(n-1)$. 
We denote 
\begin{equation}
\label{eq: def gn}
g_n(K)=\frac{m(n-1)}{2}+g(K)+1.
\end{equation}
The generators of $\HFL(K_{n,mn})$ are given by $Y_j, 1\le j\le 2g_n(K)-1$ with the Alexander grading
$$
A(Y_j)=\left(g_n(K)-j,\ldots,g_n(K)-j\right).
$$ 
We will often write 
$$
j=qm+r,\ 1\le r\le m.
$$
Note that $1\le j\le 2g_n(K)-1=m(n-1)+2g(K)+1$, so for $m>2g(K)$ we have $0\le q\le n-1$ while for $m=2g(K)$ there is a unique $j=2g_n(K)-1$ with $q=n$.

\begin{example}
For $n=1$ we get $g_1(K)=g(K)+1$, so in our normalization we have $1\le j\le 2g(K)+1$. The corresponding Alexander gradings are equal to $g(K),\ldots,-g(K)$.
\end{example}

\begin{example}
For $m=1$ and $K=O$ the unknot we have $g_n(O)=\frac{n+1}{2}$ and $1\le j\le n$, as before.
\end{example}

To describe the Maslov grading of $Y_j$, we need to introduce some notations. Let $\Delta_K(t)$ denote the Alexander polynomial of $K$, with the usual normalization such that it is symmetric around the origin. Since $K$ is an L-space knot, by \cite{OSlens} we get that the lowest grading term in $\Delta_K(t)$ equals $t^{-g(K)}$ and the highest grading term in $\Delta_K(t)$ equals $t^{g(K)}$. We define 
$$
\widetilde{H}(t)=\sum_{j}\widetilde{h}(j)t^j=\frac{t^{g(K)+1}\Delta_K(t)}{(1-t)^2}=\frac{t^{g(K)+1}}{1-t}\cdot \frac{\Delta_K(t)}{(1-t)}
$$
and 
$$
H_{n}(t)=\sum_{j}h_n(j)t^j=t\widetilde{H}(t)\frac{1-t^{mn}}{1-t^m}=t\widetilde{H}(t)\left(1+t^m+\ldots+t^{m(n-1)}\right).
$$
Equivalently, we have 
$$h_n(j)=\widetilde{h}(j-1)+\widetilde{h}(j-1-m)+\ldots+\widetilde{h}(j-1-m(n-1)).
$$

We summarize some properties of $\widetilde{h}(j)$ and $h_n(j)$ as follows.
\begin{lemma}
\label{lem: hn properties}
Assume $m>2g(K)-1$. Then the following hold:

a) We have $\widetilde{h}(j)=0$ for $j\le 0$ and $\widetilde{h}(j)=j-g(K)$ for $j\ge 2g(K)$.  Furthermore,
\begin{equation}
\label{eq: h tilde shift 0 or 1}
0\le \widetilde{h}(k+1)-\widetilde{h}(k)\le 1,
\end{equation}
so in particular $\widetilde{h}(k)\ge 0$ for $k\ge 0.$

b) The coefficients $\widetilde{h}(j)$ are symmetric: 
$$
\widetilde{h}(2g(K)-j)=\widetilde{h}(j)-j+g(K).
$$

c) Suppose $j=mq+r, 1
\le r\le m$, and $q\le n-1$. Then
$$
h_n(j)=m\frac{q(q+1)}{2}+q(r-g(K)-1)+\widetilde{h}(r-1).
$$

d) We have  %\EG{check if we need $-1$}
$$
h_n(2g_n(K)-j)=h_n(j)+n\left(g_n(K)-j\right).
$$
\end{lemma}

\begin{proof}
Equation \eqref{eq: h tilde shift 0 or 1} follows from the main result of \cite{OSlens}, the rest of
parts (a) and (b) is clear from definitions. Let us prove (c).
For $i<q$ we have $j-1-mi\ge r+m-1\ge 2g(K)$, so 
$
\widetilde{h}(j-1-mi)=j-1-mi-g(K).
$
Similarly, for $i>q$ we have $j-mi-1<0$ and $\widetilde{h}(j-mi-1)=0$. Therefore
$$
h_n(j)=\widetilde{h}(j-1)+\widetilde{h}(j-m-1)+\ldots+\widetilde{h}(j-m(n-1)-1)=$$
$$
(j-g(K)-1)+\ldots+(j-(q-1)m-g(K)-1)+\widetilde{h}(r-1)=
$$
$$
(r+qm-g(K)-1)+\ldots+(r+m-g(K)-1)+\widetilde{h}(r-1)=
$$
$$
m\frac{q(q+1)}{2}+q(r-g(K)-1)+\widetilde{h}(r-1).
$$
Part (d) follows from (b), see \cite[Lemma 4.4]{GHom}.
\end{proof}

\begin{lemma}
\label{lem: hn increases}
Suppose that $j\le j'$, then we have 
$$
0\le h_n(j')-h_n(j)\le n(j'-j).
$$
\end{lemma}

\begin{proof}
It is sufficient to prove this for $j'=j+1$, in which case it follows from the inequality \eqref{eq: h tilde shift 0 or 1}. 
\end{proof}

\begin{definition}
We define 
$$
\beta_n(j)=h_n(j+1)-h_n(j)-1.
$$ 
\end{definition}

By Lemma \ref{lem: hn increases} we get $-1\le \beta_n(j)\le n-1$.
The symmetry in Lemma \ref{lem: hn properties}(d) translates to
\begin{equation}
\label{eq: beta symmetry}
\beta_n(j)+\beta_n(2g_n(K)-j-1)=n-2.
\end{equation}
See also \cite[Lemma 4.23]{GHom}.

%We give some examples of the above phenomena.

\begin{example}
Suppose that $K=O$ is the unknot. 
For $m=1$ we get $r=1,q=j-1$ and 
$$
h_n(j)=j(j-1)/2,\ \beta_n(j)=j-1.
$$
For $m>1$ we get 
$$
h_n(j)=m\frac{q(q+1)}{2}+(q+1)(r-1),\ \beta_n(j)=q.
$$
compare with \cite[Proof of Theorem 4.1]{AGL2}.
\end{example}
 
\begin{example}
\label{ex: trefoil h}
If $K=T(2,3)$ is a positive trefoil knot then  
$$
\widetilde{h}(j)=\begin{cases}
0 & j\le 0\\
1 & j=1\\
j-1 & j\ge 2.
\end{cases}
$$
Similarly, for $m\ge 2$ we get
$$
h_n(j)=m\frac{q(q+1)}{2}+q(r-2)+\begin{cases}
0 & r=1\\
1 & r=2\\
r-2 & r\ge 3.
\end{cases}
$$
For $m>2$ this gives sequences
$$
h_n:\ \underbrace{0,1,1,\ldots,m-2}_{m},\underbrace{m-1,m+1,m+2,m+4,\ldots,3m-4}_{m},3m-2\ldots
$$
and 
$$
\beta_n:\ 
\underbrace{0,-1,0,\ldots,0}_{m},\underbrace{1,0,1,\ldots,1}_{m},\ldots
$$
For $m=2$ we get sequences
$$
h_n:\ \underbrace{0,1}_{2},\underbrace{1,3}_{2},4\ldots\ \mathrm{and}\ 
\beta_n:\ 
\underbrace{0,-1}_{2},\underbrace{1,0}_{2},\ldots
$$
\end{example}

\begin{proposition}
The full homology $\HFL(K_{n,mn})$ is generated by $Y_j$ with Alexander gradings as above, and Maslov gradings 
\begin{equation}
\label{eq: degree Yj}
M(Y_j)=-2h_n(j).
\end{equation}
The relations between $Y_i$ are spanned by the relations 
\begin{equation}
U_AY_j=V_{\overline{A}}Y_{j+1},\ U_iV_iY_j=U_kV_kY_j
\end{equation}
for all subsets $A\subset \{1,\ldots,n\}, |A|=\beta_n(j)+1$.
\end{proposition}

\begin{proof}
The proof is similar to \cite[Theorem 7.3]{BLZ} and \cite[Theorem 4.1]{AGL2}.
\end{proof}

\subsection{Hat homology}

The hat homology of the cable link $K_{n,mn}$ is given by the following result.

\begin{theorem}\cite[Theorem 3]{GHom}
\label{thm: GH general}
Assume $m>2g(K)-1$. The homology $\HFLhat(K_{n,mn})$ has the following form:

a) In the Alexander grading $A(Y_j)$ for $\beta_n(j)+\beta_n(j-1)\le n-2$ the homology is isomorphic to 
$$
\bigoplus_{i=0}^{\beta_n(j)}\binom{n-1}{i}\F_{-2h_n(j)-i}\oplus \bigoplus_{i=0}^{\beta_n(j-1)}\binom{n-1}{i}\F_{-2h_n(j)-n+2+i}.
$$

b) In the Alexander grading $A(Y_j)$ for $\beta_n(j)+\beta_n(j-1)\ge n-2$ the homology is isomorphic to 
$$
\bigoplus_{i=0}^{n-2-\beta_n(j-1)}\binom{n-1}{i}\F_{-2h_n(j)-i}\oplus \bigoplus_{i=0}^{n-2-\beta_n(j)}\binom{n-1}{i}\F_{-2h_n(j)-n+2+i}.
$$

c) Suppose that $1\le j<2g_n(K)-1$, and the Alexander grading has $s$ coordinates equal to $A(Y_j)-1=A(Y_{j+1})$ and $n-s$ coordinates equal to $A(Y_j)$, for  $1\le s\le  n-1$. Then the  homology is isomorphic to 
$$
\binom{n-2}{\beta_n(j)}\F_{-2h_n(j)-\beta_n(j)-s}.
$$
The hat homology in all other Alexander gradings is zero.  
\end{theorem}

\begin{remark}
In case (a), if $\beta_n(j)=-1$ or $\beta_n(j-1)=-1$ then the corresponding summand vanishes. In case (b), if $\beta_n(j)=n-1$ or $\beta_n(j-1)=n-1$ then  the corresponding summand vanishes as well.

In particular, we can formally assume that $\beta_{n}(0)=-1$. Since $\beta_n(1)=0$,  for $j=1$ the first summand in Theorem \ref{thm: GH general}(a) is one-dimensional and the second summand vanishes.

By symmetry \eqref{eq: beta symmetry} we have $\beta_n(2g_n(K)-1)=n-1$ and $\beta_n(2g_n(K)-2)=n-2$, so for $j=2g_n(K)-1$ the first summand in Theorem \ref{thm: GH general}(b) is one-dimensional and the second summand vanishes. 
\end{remark}

\begin{remark}
Note that by \eqref{eq: beta symmetry} the cases (a) and (b) in Theorem \ref{thm: GH general} are exchanged by the involution
$
j\leftrightarrow 2g_n(K)-j.
$
\end{remark}

As before, we refer to
\begin{equation}
\label{eq: def mj L space}
\bigoplus_{i=0}^{m_j}\binom{n-1}{i}\F_{-2h_n(j)-i},\quad  m_j=\min(\beta_n(j),n-2-\beta_n(j-1))
\end{equation}
as the bottom diagonal block at $A(Y_j)$, and to 
\begin{equation}
\label{eq: def Mj L space}
\bigoplus_{i=0}^{n-M_j}\binom{n-1}{i}\F_{-2h_n(j)-n+2+i},\quad  
M_j=\max(n-\beta_n(j-1),\beta_n(j)+2)
\end{equation}
as to the top diagonal block at $A(Y_j)$. If $\beta_n(j)=-1$ or $\beta_n(j-1)=n-1$ then $m_j$ is not defined, and if $\beta_n(j)=n-1$ or $\beta_n(j-1)=-1$ then $M_{j}$ is not defined. The results of \cite{GHom} and the proofs of Lemma \ref{lem: levels} and Lemma \ref{lem: Sn action diagonal blocks} generalize to this case verbatim, and we get the following.

\begin{lemma}
a) Assume $m_j$ is defined, then the bottom diagonal block is supported in levels $0,\ldots,m_j.$ The corresponding $S_n$-representation at level $i$ is isomorphic to $V_{n,i}$.

b) Assume $M_j$ is defined, then the top diagonal block is supported in levels $M_j,\ldots,n.$ The corresponding $S_n$-representation at level $n-i$ is isomorphic to $V_{n,n-1-i}$.
\end{lemma}

Lemma \ref{lem: levels off diagonal} generalizes as follows.
 
\begin{lemma}
\label{lem: levels off diagonal L space} 
The off-diagonal homology for a given $j,s$ is supported in levels 
$$
\lvl(j,s,c)=\beta_n(j)-s+2c
$$
where $c$ runs over the interval
\begin{equation}
\label{eq: c j inequalities L space}
\max(1,s-\beta_n(j))\le c\le \min(s,n-1-\beta_n(j)).
\end{equation}
For a fixed $c$, the rank of the off-diagonal summand equals
$$
\binom{s-1}{s-c}\binom{n-s-1}{\beta_n(j)+c-s}=\binom{s-1}{c-1}\binom{n-s-1}{n-1-\beta_n(j)-c}.
$$
\end{lemma}

\begin{remark}
Observe that 
$$
\sum_{c}\binom{s-1}{c-1}\binom{n-s-1}{n-1-\beta_n(j)-c} =\binom{n-2}{n-\beta_n(j)-2}=\binom{n-2}{\beta_n(j)},
$$
in agreement with Theorem \ref{thm: GH general}(c).
\end{remark}

\begin{example}
\label{ex: extremal off-diagonal 1 L space}
Consider the case $s=1$. By \eqref{eq: c j inequalities L space} this forces $c=1$, so all the off-diagonal summands in the corresponding Alexander grading are supported at the level 
$$
\beta_n(j)-1+2=\beta_n(j)+1.
$$
\end{example}

\begin{example}
\label{ex: extremal off-diagonal L space}
Consider the case $s=n-1$. By \eqref{eq: c j inequalities L space} this forces $c=n-1-\beta_n(j)$,  so all the off-diagonal summands in the corresponding Alexander grading are supported at the level 
$$
\beta_n(j)-(n-1)+2(n-1-\beta_n(j))=n-1-\beta_n(j).
$$
\end{example}

\subsection{Changing orientation}

As in Section \ref{sec: change orientation Tnn}, let us change the orientation on $k_{-}$ components of the link $K_{n,mn}$ to get the link $K_{m}(k_{+},k_{-})$. 
By Corollary \ref{cor: change orientation many components} the generator $Y_j$ becomes $Y'_j$ with 
$$
M(Y'_j)=M(Y_j)-2k_{-}\left(\frac{m(n-1)}{2}+g(K)+1-j\right)+mk_{+}k_{-}.
$$
since each pair of components has linking number $m$. We can rewrite this as
\begin{equation}
\label{eq: Maslov Y' L space}
M(Y'_j)=-2h_n(j)-mk_{-}(k_{-}-1)-2k_{-}(g(K)+1-j).
\end{equation}
More generally, we have the following gradings for the diagonal blocks:
\begin{lemma}
\label{lem: change orientation L space diagonal}
The bottom diagonal block in $\HFLhat\left(K_m(k_{+},k_{-})\right)$ has the form
$$
\bigoplus_{i=0}^{m_j}\binom{n-1}{i}\F_{M(Y'_j)-i} 
$$
and the top diagonal block in $\HFLhat\left(K_m(k_{+},k_{-})\right)$  has the form 
$$
\bigoplus_{i=0}^{n-M_j}\binom{n-1}{i}\F_{M(Y'_j)-n+2+i},
$$
where $M(Y'_j)$ is given by \eqref{eq: Maslov Y' L space} and $m_j,M_j$ are given by \eqref{eq: def mj L space},\eqref{eq: def Mj L space}.
\end{lemma}
For the off-diagonal blocks we get the following:
\begin{lemma}
\label{lem: change orientation L space off diagonal}
Given a class $a$ in $\HFLhat(K_{n,mn})$ with $s$ components of the Alexander grading equal to $A(Y_j)-1$ and $n-s$ equal to $A(Y_j)$, let us change the orientation on $s_{-}$ strands in the former group and $k_{-}-s_{-}$ strands in the latter group. The resulting class $a'$ has Maslov grading 
\begin{equation}
\label{eq: off diagonal change orientation L space}
M(a')=M(Y'_{j})-\beta_n(j)-s+2s_{-}.
\end{equation}
\end{lemma}

\begin{proof}
The original class $a$ has grading 
$
M(a)=-2h_n(j)-\beta_n(j)-s.
$
By Corollary \ref{cor: change orientation many components}, after changing orientation we get new Maslov grading
$$
M(a')=M(a)-2s_{-}\left(\frac{m(n-1)}{2}+g(K)-j\right)-2(k_{-}-s_{-})\left(\frac{m(n-1)}{2}+g(K)+1-j\right)+mk_{+}k_{-}=
$$
$$
M(a)-2k_{-}\left(\frac{m(n-1)}{2}+g(K)+1-j\right)+mk_{+}k_{-}+2s_{-}=
$$
$$
M(Y'_{j})-\beta_n(j)-s+2s_{-}.
$$

\end{proof}

\subsection{Adding strands}

We would like to understand the effect of adding an oppositely oriented pair of strands. Note that 
$$
g_{n+2}(K)=g_n(K)+m.
$$

\begin{lemma}
\label{lem: h and beta change}
For $m>2g(K)-1$ and $1\le j<2g_n(K)-1$ we have 
$$
h_{n+2}(j+m)=h_{n}(j)+j+m-g(K)-1,\ \beta_{n+2}(j+m)=\beta_n(j)+1.
$$
\end{lemma}

\begin{proof}
Suppose that $j=mq+r,1\le r\le m$, then $0\le q\le n-1$. Then
$$
h_n(j)=\widetilde{h}(j-1)+\widetilde{h}(j-m-1)+\ldots+\widetilde{h}(j-m(n-1)-1)= 
$$
$$
(j-g(K)-1)+\ldots+(j-m(q-1)-g(K)-1)+\widetilde{h}(r-1).
$$
Now $j+m=m(q+1)+r$ and 
$$
h_{n+2}(j+m)=(j+m-g(K)-1)+ 
(j-g(K)-1)+\ldots+(j-m(q-1)-g(K)-1)+\widetilde{h}(r-1)=
$$
$$
h_n(j)+(j+m-g(K)-1).
$$
Finally,
$$
\beta_{n+2}(j+m)=h_{n+2}(j+m+1)-h_{n+2}(j+m)-1=
$$
$$
\beta_n(j)+(j+m+1-g(K)-1)-(j+m-g(K)-1)=\beta_n(j)+1.
$$
\end{proof}

\begin{corollary}
\label{cor: size blocks grows L space}
Assume that $1<j<2g_n(K)-1$. Then $m_{j+m,n+2}$ and $M_{j+m,n+2}$ are defined. Whenever $m_{j,n}$ or $M_{j,n}$ are defined, we have
$m_{j+m,n+2}=m_{j,n}+1$ and
$M_{j+m,n+2}=M_{j,n}+1$.
\end{corollary}

\begin{proof}
Since $1<j<2g_n(K)-1$, one has 
$1\le j-1<2g_n(K)-1$. Therefore Lemma \ref{lem: h and beta change} applies both to $j$ and $j-1$, so that 
$$
\beta_{n+2}(j+m)=\beta_{n}(j)+1,\ \beta_{n+2}(j+m-1)=\beta_{n}(j-1)+1.
$$
We have $-1\le \beta_{n}(j),\beta_n(j-1)\le n-1$, so
$$
0\le \beta_{n+2}(j+m),\beta_{n+2}(j+m-1)\le n, 
$$
so $m_{j+m,n+2}$ and $M_{j+m,n+2}$ are defined. Now
$$
m_{j+m,n+2}=\min(\beta_{n+2}(j+m),(n+2)-2-\beta_{n+2}(j+m-1))=
$$
$$
\min(\beta_{n}(j)+1,n-2-\beta_{n}(j-1)+1)=m_{j,n}+1
$$
and
$$
M_{j+m,n+2}=\max((n+2)-\beta_{n+2}(j+m-1),\beta_{n+2}(j+m)+2)=
$$
$$\max(n-\beta_{n}(j-1)+1,\beta_{n}(j)+2+1)=M_{j,n}+1.
$$
\end{proof}

\begin{corollary}
\label{cor: M stable L space}
Assume that $1\le j<2g_n(K)-1$.
Then $M(Y'_{j,n})=M(Y'_{j+m,n+2})$.
\end{corollary}

\begin{proof}
By \eqref{eq: Maslov Y' L space} we get 
$$
M(Y'_{j,n})=-2h_n(j)-mk_{-}(k_{-}-1)-2k_{-}(g(K)+1-j).
$$
After changing $n$ to $n+2$, $j$ to $j+m$ and $k_{-}$ to $k_{-}+1$, we get
$$
M(Y'_{j+m,n+2})=-2h_{n+2}(j+m)-m(k_{-}+1)k_{-}-2(k_{-}+1)(g(K)+1-j-m)=
$$
$$
M(Y'_{j,n})-2(j+m-g(K)-1)-2mk_{-}-2(g(K)+1-j-m)+2mk_{-}=M(Y'_{j,n}).
$$
 \end{proof}

\subsection{Floer lasagna module: large negative case}

The large negative case is similar to Section \ref{sec: -1 framed unknot}, as we can use duality between homology of $K_{m}(k_{+},k_{-})$ and $\overline{K}_{-m}(k_{+},k_{-})$. In particular, we have the following description of the dual homology at the diagonal.

\begin{lemma}
\label{lem: dual diagonal L space}
Let $K$ be an L-space knot, $\overline{K}$ its mirror and $m>2g(K)-1$. Then the diagonal part of $\HFLhat(\overline{K}_{-m}(k_{+},k_{-}))$ at $j$ has two blocks which decompose as bigraded $S_n$-representations as follows:
\begin{enumerate}
\item[(a)] The dual bottom block decomposes into $V_{n,i}^{\vee}\simeq V_{n,n-1-i}$ in level $n-i$ and Maslov grading 
$$
1-n-M(Y'_{j,n})+i.
$$
\item[(b)] The dual top block decomposes into $V_{n,n-i-1}^{\vee}\simeq V_{n,i}$ in level $i$ and Maslov grading 
$$
1-n-M(Y'_{j,n})+n-2-i=-M(Y'_{j,n})-1-i.
$$
\end{enumerate}
All these classes have total Alexander grading $-\alpha(g_n(K)-j)$.
\end{lemma}

\begin{lemma}
\label{lem: Fp from off-diagonal L space}
The map 
$$
F_P(-\otimes B):\HFLhat(K_m(k_{+},k_{-}))\to \HFLhat(K_m(k_{+}+1,k_{-}+1))
$$
vanishes on the off-diagonal part of $\HFLhat$ and the image of
$$
\Phi:\HFLhat(K_m(k_{+}+1,k_{-}+1))\to \HFLhat(K_m(k_{+},k_{-}))
$$
does not contain any off-diagonal terms of $\HFLhat$. 
\end{lemma}

\begin{proof}
The proof is similar to the proof of Lemma \ref{lem: Fp from off-diagonal}. We change $n$ to $n+2$, $k_{-}$ to $k_{-}+1$, $j$ to $j+m$.
By Corollary \ref{cor: M stable L space} $M(Y_j')$ does not change, and by Lemma \ref{lem: h and beta change} $\beta_n(j)$ increases by $1$.  Note that for the off-diagonal blocks we have $1\le j<2g_n(K)-1$, so Lemma \ref{lem: h and beta change} applies.

Now we use the formula \eqref{eq: off diagonal change orientation L space}.
If $(s,s_{-})\mapsto (s,s_{-})$ then overall Maslov grading \eqref{eq: off diagonal change orientation L space} changes by $(-1)$. If $(s,s_{-})\mapsto (s+2,s_{-}+1)$ then overall Maslov grading \eqref{eq: off diagonal change orientation L space} changes by $-1-2+2=-1$.

Therefore any map of Maslov grading $(-2)$ or $0$ between these components vanishes.
\end{proof}

Consider the equivalence relation $(j,n)\sim (j+m,n+2)$ on $\Z_{>0}\times \Z_{>0}$.

\begin{theorem}
\label{thm: negative L space main}
Let $K$ be a plumbed L-space knot, $\overline{K}$ its mirror and $m>\max\{0,2g(K)-1\}$.  Then the Floer lasagna module   $\FL(W_{-m}(\overline{K}),\emptyset,\alpha)$ has the  generators indexed by the representative class $(j,n)$, with  $n=\alpha\mod 2, n\geq |\alpha|$, and $1\leq j\leq 2g_n(K)-1$.  These have Alexander grading $-\alpha(g_n(K)-j)$ and the following normalized Maslov gradings:
\begin{enumerate}
\item[(a)] For $\alpha$ even, we have normalized gradings
$$
-M(Y'_{j,n})+1+i,\ i\in \{0,1,3,\ldots\} 
$$
\item[(b)] For $\alpha$ odd, we have normalized gradings
$$
-M(Y'_{j,n})+1+i,\ i\in \{0,2,4,\ldots\},
$$
with $n\geq |\alpha|$ and $1\leq j\leq 2g_n(K)-1$. 
\end{enumerate}
\end{theorem}

\begin{proof}
By Lemma \ref{lem: Fp from off-diagonal L space} and duality the map $F_{P}(-\otimes B)$ vanishes on the off-diagonal blocks for the mirror $\HFLhat(\overline{K}_{-m}(k_{+},k_{-}))$. Therefore, the off-diagonal blocks do not contribute to the cabled Floer homology.
Similarly to the proof of Theorem \ref{thm: unknot negative intro}, one can prove that the dual top diagonal blocks do not contribute as well.
Indeed, we change $n$ to $n+2$, $k_{-}$ to $k_{-}+1$, $j$ to $j+m$. 

Assume first $1<j<2g_n(K)-1$. By Lemma \ref{lem: h and beta change} $\beta_n(j)$ increases by $1$, so 
$
M_{j+m,n+2}=M_{j,n}+1
$
and the levels in the dual top block are bounded by
$$
(n+2r)-M_{j+mr,n+2r}=n+r-M_{j,n}.
$$
On the other hand, $F_{P}(-\otimes B)$ still increases the level by $2$, and its $r$-th iterate increases the level by $2r$.

Therefore one only gets the contributions from the dual bottom diagonal blocks. 
Similarly to the proof of Lemma \ref{lem: Fp for dual bottom} the maps $F_{P}(-\otimes T)$ and $F_{P}(-\otimes B)$ on the dual bottom blocks are given by the same formulas as for the $0$-framed unknot, truncated at an appropriate level.
By Lemma \ref{lem: dual diagonal L space} in the dual bottom blocks  we get a representation $V_{n,n-i-1}$ which has one-dimensional coinvariants if $n-i-1$ is even or $i=0$. The rest of the proof follows the proof of Theorem \ref{thm: unknot negative intro}. 

Finally, if $j=1$ or $j=2g_n(K)-1$ then we change $(j,n)$ to $(n+2,j+m)$, observe that
$1<j+m<2g_{n+2}(K)-1$, and apply the previous argument.
\end{proof}

\subsection{Floer lasagna module: large positive case}

In this subsection we give a precise conjectural description for the Floer lasagna module   for $m>2g(K)-1$.  We propose the analogues of Corollary \ref{cor: Phi1 appendix} and Theorem \ref{thm: Fp top diagonal}.

\begin{conjecture}
\label{conj: Phi1 top diagonal L space}
Assume $0\le i<n-M_{j,n}$, then there exists a spanning set $D^j(J)$ in the top diagonal block at level $n-i$,
here $J\subset \{1,\ldots,n\}$ and $|J|=n-i$. These elements satisfy linear relations
\begin{equation}
\label{eq: D relations L space}
\sum_{x\in J'}D^j(J'\setminus x)=0,\ |J'|=n-i+1,
\end{equation}
compare with 
\eqref{eq: D relations}. The map $\Phi_1$ is given by
$$
\Phi_1(D^j_n(J))=U_{n+1}D^{j+m}_{n+2}(J).
$$
modulo off-diagonal and bottom diagonal terms.
\end{conjecture}

\begin{conjecture}
\label{conj: Fp top diagonal L space}
In the notations of Conjecture \ref{conj: Phi1 top diagonal L space} the map $F_P$ is given by 
\begin{equation}
\label{eq: Fp top block alg model 1 L space}
F_P\left(D^j_n(J)\otimes T\right)=D_{n+2}^{j+m}(J\cup n+1)+D_{n+2}^{j+m}(J\cup n+2),
\end{equation}
\begin{equation}
\label{eq: Fp top block alg model 2 L space}
F_P\left(D^j_n(J)\otimes B\right)=D_{n+2}^{j+m}(J\cup n+1\cup n+2) \mod \Imm F_P(-\otimes T).
\end{equation}
modulo off-diagonal and bottom diagonal terms.
\end{conjecture}

\begin{remark}
One can check that the candidate map \eqref{eq: Fp top block alg model 1 L space} respects the relations \eqref{eq: D relations L space}, while \eqref{eq: Fp top block alg model 2 L space} respects these relations up to the image of $F_P(-\otimes T)$, compare with Remark \ref{rem: FP well defined}.
\end{remark}

In the following theorem we use the equivalence relation on $\Z_{>0}\times \Z_{>0}$ given by $(j,n)\sim (j+m,n+2)$.

 \begin{theorem}
\label{thm: L space positive}
Let $K$ be a plumbed L-space knot and $m> \max \{0, 2g(K)-1\}$. Assuming Conjecture \ref{conj: Fp top diagonal L space}, we have the following description of the cabled Floer homology:
\begin{enumerate}
\item[(a)]  For $\alpha$ even and an equivalence class $(j,n)$ the space $\FL(W_m(K),\emptyset;\alpha)$ has generators in Alexander grading $\alpha(g_n(K)-j)$ and normalized Maslov gradings 
$$
M(Y'_{j,n})+2+\{0,1,3,\ldots\}
$$
\item[(b)] For $\alpha$ odd and an equivalence class $(j,n)$ the space $\FL(W_m(K),\emptyset;\alpha)$ has generators in Alexander grading $\alpha(g_n(K)-j)$ and normalized Maslov gradings  
$$
M(Y'_{j,n})+2+\{0,2,4,\ldots\}
$$
\end{enumerate}
Here $\alpha=k_{+}-k_{-}$, $n$ has the same parity as $\alpha$, $1\leq j\leq 2g_n(K)-1$,  and $M(Y'_{j,n})$ is defined by \eqref{eq: Maslov Y' L space}.
\end{theorem}

The value $g_n(K)-j$ only depends on the equivalence class of $(j,n)$. By Corollary \ref{cor: M stable L space} the value $M(Y'_{j,n})$  only depends on the equivalence class of $(j,n)$ as well.

\begin{proof}
The proof is similar to the proof of Theorem \ref{thm: 1 framed unknot}. By Lemma \ref{lem: Fp from off-diagonal L space} the map $F_P(-\otimes B)$ vanishes on off-diagonal blocks. 

For the bottom diagonal block, the limit vanishes similarly to the proof of Theorem \ref{thm: 1 framed unknot}. By Corollary \ref{cor: size blocks grows L space} the height of the bottom diagonal block increases by one after changing $(j,n)$ to $(j+m,n+2)$. On the other hand, $F_P(-\otimes B)$ still increases the level grading by 2.

For the top diagonal block, we use Conjecture \ref{conj: Fp top diagonal L space} to compare it with the truncated cube for the $0$-framed unknot, as in  Theorem \ref{thm: 1 framed unknot}. In particular, at level $n-i$ we get a representation $V_{n,n-i-1}$ which has one-dimensional coinvariants whenever $n-i-1$ is even or $i=0$.
\end{proof}

\begin{appendix}

\section{An algebraic model for $T(n,n)$}
\label{sec: appendix A}

In this section we define a complex $\CFL^{\alg}(T(n,n))$ which we conjecture to be an algebraic model for the full Heegaard Floer complex $\CFL(T(n,n))$. As before, we work over $\F$ and ignore signs.

We start with the following observation.

\begin{lemma}
\label{lem: subspace}
Consider the $\F$-vector space $U_{n,k}$ spanned by vectors $v_{M}$ where $M\subset \{1,\ldots,n\}$ and $|M|=k$, modulo relations
$$
\sum_{j\in M'}v_{M'\setminus j}=0,\ |M'|=k+1.
$$
Then $U_{n,k}$ is isomorphic to the $S_n$ representation $V_{n,k-1}$ from Section \ref{sec: rep theory}, in particular $\dim U_{n,k}=\binom{n-1}{k-1}$. The isomorphism sends the equivalence class of $v_M$ to 
$$
[v_M]\mapsto \begin{cases}
\xi_{M\setminus 1} & \mathrm{if}\ 1\in M\\
\sum_{m\in M}\xi_{M\setminus m} & \mathrm{if}\ 1\notin M.
\end{cases}
$$
\end{lemma}

\begin{proof}
Recall the cube complex $C_n$ from Section \ref{sec: rep theory}. We can identify 
$$
U_{n,k}=C_{n,k}/\Imm \partial_{k+1}=C_{n,k}/\Ker \partial_{k}\simeq \Imm \partial_k= \Ker \partial_{k-1}=V_{n,k-1}.
$$
The isomorphism sends the equivalence class of $v_M$ to $\partial(e_M)$. To express the result as a linear combination of $\xi_A$, one can use the proof of Proposition \ref{prop: Sn action on xi} and pick up all the terms in $\partial(e_M)$ not containing $1$.
\end{proof}

\subsection{Generators}

The complex $\CFL^{\alg}(T(n,n))$ has the following generators. We follow the nomenclature of Section \ref{sec: Tnn}.

\noindent {\bf 1.} The elements $X_1,\ldots,X_{n}$ which generate the homology $\HFL(T(n,n))$.

\noindent {\bf 2.} The bottom diagonal block at $j$ and level $i$ has generators $Z^{j}(J)$ where  $|J|=i+1$ and $0\le i\le \min(j-1,n-j)$. The elements $Z^{j}(J)$ are subject to relations
$$
Z^j(J)=X_j,\ |J|=1
$$
and
\begin{equation}
\label{eq: Z relations}
\sum_{x\in J'}Z^j(J'\setminus x)=0,\ |J'|=i+2.
\end{equation}
By Lemma \ref{lem: subspace} $Z^j(J)$ span a subspace $\cZ_{i}^j$ of dimension $\binom{n-1}{i}$.

\noindent {\bf 3.} The off-diagonal block with parameters $(j,s,c)$ has  generators $B^j(I,J,M)$ where $I,J,M$ are nonempty subsets such that
$$
|I|=s,\ |J|=c,\ |M|=j+c-s,\ J\subset I, M\subset\overline{I}. 
$$
Note that the inequalities
$$
1\le c\le s,\ c\le n-j,\ c+j>s
$$
ensure that such $J,M$ exist for any $I$.
These are subject to relations 
\begin{equation}
\label{eq: B relations}
\sum_{x\in J'}B^j(I,J'\setminus x,M)=0=\sum_{y\in   M'}B^j(I,J,M'\setminus y),\ J'\subset I,M'\subset \overline{I}.
\end{equation}
for $|J'|=c+1$, $|M'|=j+c-s+1$ and they span a subspace 
$\cB^j_{s,c}=\bigoplus_{I}\cB^j_{s,c}(I)$.  
By Lemma \ref{lem: subspace} we have 
$$
\dim \cB^j_{s,c}(I)=\binom{s-1}{c-1}\binom{n-s-1}{j+c-s-1}=\binom{s-1}{c-1}\binom{n-s-1}{n-j-c}.
$$
in agreement with Lemma \ref{lem: levels off diagonal}.
If $c=1$ then the choice of $J$ does not matter (modulo relations) and we simply write $B^j(I,-,M)$. If $c=s-j+1$ then   the choice of $M$ does not matter  and we write $B^j(I,J,-)$.

\noindent {\bf 4.} The top diagonal block at $j$ and level $n-i$ has generators $D^j(J)$, where $|J|=n-i$ and $i\le \min(j-2,n-j-1)$, subject to relations
\begin{equation}
\label{eq: D relations}
\sum_{x\in J'}D^j(J'\setminus x)=0,\ |J'|=n-i+1.
\end{equation}
These span a subspace $\cD^j_{i}$ of dimension
$$
\dim \cD^j_{i}=\binom{n-1}{n-i-1}=\binom{n-1}{i},
$$
in agreement with Theorem \ref{thm: GH}.

The Alexander, Maslov and level gradings of all these generators are given by Theorem \ref{thm: GH}, Lemma \ref{lem: levels} and Lemma \ref{lem: levels off diagonal}. Note that $B^j(I,J,M)$ has Alexander grading
$A(X_j)-e_I$. 

\begin{remark}
In Lemma \ref{lem: unlink action bottom diagonal} we  used slightly different notations $Z^j_A$ for the bottom diagonal generators. They are related to the generators $Z^j(B)$ here as in Lemma \ref{lem: subspace}, in particular $|A|=|B|-1$.
\end{remark}
 
\subsection{Differential}
The differential on $\CFL^{\alg}$ is given as follows:

\noindent {\bf 1.} The differential of $Z^j(J)$ is given by
\begin{equation}
\label{eq: diff Z}
\partial Z^{j}(J)=\sum_{x\in J}U_{x}V_{x}Z^j(J\setminus x),\ j\ge 2, \quad |J|\geq 2.
\end{equation}
In particular,
$$
\partial Z^j(\{x,y\})=(U_xV_x+U_yV_y)X_j.
$$

\noindent {\bf 2.} The differentials of $B^j(I,J,M)$ break into several groups. We focus on the case $1\le j\le \frac{n+1}{2}$, the differentials for $ \frac{n+1}{2}\le j\le n$ are symmetric.

\noindent {\bf 2a. (Generic case):}  For $c>1, |M|>1$ we can write 
\begin{equation}
\label{eq: diff B generic}
\partial B^j(I,J,M)=\sum_{x\in J}U_xB^j(I\setminus x,J\setminus x,M)+\sum_{y\in  M}V_yB^j(I\cup y,J,M\setminus y).
\end{equation}

\noindent {\bf 2b. (Bottom vertex):} For $s=j$ and $c=1$ we have 
\begin{equation}
\label{eq: diff B bottom}
\partial B^j(I,-,-)=U_IX_{j}+V_{\overline{I}}X_{j+1}. 
\end{equation}
\noindent {\bf 2c. (Left edge):} %Assume $1\le j\le \frac{n+1}{2}$. 
For $1\le s<j$ and $c=1$ we have
\begin{equation}
\label{eq: diff B left}
\partial B^j(I,-,M)=\sum_{y\in  M}V_yB^j(I\cup y,-,M\setminus y)+U_IZ^j(M)
\end{equation}

\noindent {\bf 2d. (Right edge):} %Assume $2\le j\le\frac{n-1}{2}$. 
For $2\le c\le j+1$ and $s=c+j-1$ we get 
\begin{equation}
\label{eq: diff B right case 1}
\partial B^j(I,J,-)=\sum_{x\in J} U_x B^j(I\setminus x,J\setminus x,-)+V_{\overline{I}}Z^{j+1}(J),
\end{equation}
while for $j+1< c\le n-j$ and $s=c+j-1$ we get
\begin{equation}
\label{eq: diff B right case 2}
\partial B^j(I,J,-)=\sum_{x\in J} U_x B^j(I\setminus x,J\setminus x,-)+\sum_{\stackrel{I_1\subset J}{|I_1|=s-2j}}V_{\overline{I}}V_{I_1}B^{j+1}(I_1,I_1,J\setminus I_1),
\end{equation}
In the boundary case $c=j+2,s=2j+1$, we have $|I_1|=1$ and we get 
$$
\partial B^j(I,J,-)=\sum_{x\in J} U_x B^j(I\setminus x,J\setminus x,-)+\sum_{x\in J}V_{\overline{I}}V_{x}B^{j+1}(x,x,J\setminus x).
$$

\noindent {\bf 3.} The differential of $D^j(J)$ is given by 
$$
\partial D^j(J)=
\sum_{\stackrel{I_1\subset J}{|I_1|=n-j+1}}U_{J\setminus I_1}B^{j-1}(I_1\cup \overline{J},I_1,J\setminus I_1)+\sum_{x\in J}U_xV_xD^j(J\setminus x)+
\sum_{\stackrel{I_2\subset J}{|I_2|=n-j-i}}V_{I_2}B^j(I_2,I_2,J\setminus I_2).
$$
%compare with Lemmas \ref{lem: diff top to top} and \ref{lem: diff to off diagonal} below.
We write the three components of the differential as $\partial_1,\partial_2$ and $\partial_3$. If $|J|=\max\{n-j+2, j+1\}$ then $D^j(J)$ is at the lowest level of the diagonal block, and the term $\partial_2$ vanishes.
This completes the definition of the differential on $\CFL^{\alg}(T(n,n))$.

\begin{remark}
The reader is invited to verify that the differential decreases the level and Maslov gradings by 1 and preserves the Alexander grading. %See also Lemmas \ref{lem: diff top to top},\ref{lem: diff to off diagonal} and \ref{lem:diff-off-off},  constraining the differential on $\CFL(T(n,n))$ in a similar way.
\end{remark}

\begin{example}
For $n=3$ one can compare the complex with Figure \ref{fig: T33}, see also Example \ref{ex: T33}.
\end{example}

\begin{example}
For $n=4$ we can compare the generators with the ones from \cite[Section 7.4]{BLZ}. The generators $B_{ij}$ in \cite{BLZ} correspond to our  $B^2(\{i,j\},-,-)$, the generators $c_k^{ij}$ in \cite{BLZ} correspond to $B^2(\{k\},-,\{i,j\})$  and the generators $C^k_{ij}$ correspond to $B^2(\{i,j,k\},\{i,j\},-)$. Finally, the generators $B'_{ij}$ correspond to $B^2(\{i,j\},\{i,j\},\{k,l\})$ where $i,j,k,l$ are all distinct.
Note that \cite{BLZ} has relations
$$
c_{k}^{ij}+c_k^{il}+c_k^{jl}=0,\ 
C^k_{ij}+C^{i}_{jk}+C^{j}_{ik}=0
$$
which correspond to our relations \eqref{eq: B relations}.
Here is a visualization of generators for $n=4$:

\adjustbox{scale=0.8,center}{
\begin{tikzcd}
 & & & & \cD^2_0 \arrow{dl} \arrow{drr} & & & & \cD^3_0 \arrow{dll}\arrow{dr} & & & & \\
 & & &  \cB^1_{3,3} \arrow{drr} \arrow{dl}& & & \cB^2_{2,2} \arrow{dr}\arrow{dl} & & & \cB^3_{1,1} \arrow{dll} \arrow{dr} & & & \\
 & & \cB^1_{2,2}  \arrow{drr} \arrow{dl} & & & \cB^2_{1,1} \arrow{dr}\arrow{dl} & & \cB^2_{3,2} \arrow{dr}\arrow{dl} & & & \cB^3_{2,1} \arrow{dll} \arrow{dr} & & \\
 & \cB^1_{1,1}  \arrow{drrr} \arrow{dl} & & & \cZ^2_1 \arrow{d} & & \cB^2_{2,1} \arrow{dll} \arrow{drr} & &  \cZ^3_1  \arrow{d} & & & \cB^3_{3,1} \arrow{dlll} \arrow{dr} & \\
X_1 & & & & X_2 & & & & X_3 & & & & X_4 \\
\end{tikzcd}
}
\end{example}

\begin{example}
Let us zoom into a bigger example for $n=6$ and $j=2,3$. The off-diagonal blocks  $\cB^{j}_{s,c}$  are shaded, and the $U,V$-coefficients of the differentials are shown schematically in the picture:

\medskip

\adjustbox{scale=0.8,center}{
\begin{tikzcd}[execute at end picture={
\foreach \Nombre in  {A,B,C,D,E,F,G,H}
  {\coordinate (\Nombre) at (\Nombre.center);}
\fill[lightgray,opacity=0.3] 
  (A) -- (B) -- (C) -- (D) -- cycle;
\fill[lightgray,opacity=0.3] 
  (E) -- (F) -- (G) -- (H) -- cycle;
}] 
 \cD^2_0 \arrow{drrrr}{V^4} & & & & & & \cD^3_0 \arrow{dll}{U^2} \arrow{d}{UV} \arrow{drrr}{V^3} & & & & & &\\
  & & & & |[alias=B]|\cB^2_{4,4} \arrow{dl}{U} \arrow{dr}{V} & & \cD^3_1 \arrow{dl}{U} \arrow{drr}{V^2} & & & |[alias=F]| \cB^3_{3,3} \arrow{dl}{U} \arrow{dr}{V} & & &\\
  & & & \cB^2_{3,3} \arrow{dl}{U} \arrow{dr}{V} & & |[alias=C]| \cB^2_{5,4}  \arrow{dl}{U} \arrow{drr}{V^2} &  & & \cB^3_{2,2} \arrow{dl}{U} \arrow{dr}{V} & & \cB^3_{4,3} \arrow{dl}{U} \arrow{dr}{V} & & \\
& & \cB^2_{2,2}\arrow{dl}{U} \arrow{dr}{V} & & \cB^2_{4,3} \arrow{dl}{U} \arrow{drr}{V^2} & &  & |[alias=E]| \cB^3_{1,1} \arrow{dl}{U} \arrow{dr}{V} & & \cB^3_{3,2} \arrow{dl}{U} \arrow{dr}{V} & &  |[alias=G]| \cB^3_{5,3} \arrow{dl}{U}& \\
& |[alias=A]| \cB^2_{1,1} \arrow{dl}{U} \arrow{dr}{V} & & \cB^2_{3,2} \arrow{dl}{U} \arrow{drrr}{V^3} & & & \cZ^3_2 \arrow{d}{UV} & & \cB^3_{2,1} \arrow{dll}{U^2} \arrow{dr}{V} & & \cB^3_{4,2} \arrow{dl}{U} & &\\
\cZ^2_{1} \arrow{d}{UV}   & & 
|[alias=D]| \cB^2_{2,1} \arrow{dll}{U^2} \arrow{drrrr}{V^4} & & & & \cZ^3_1 \arrow{d}{UV}& & & |[alias=H]| \cB^3_{3,1} \arrow{dlll}{U^3} & & &\\
 X_2 & & & & & & X_3 & & & & & &
\end{tikzcd}
}
\end{example}

Using the definitions, one can readily check that $\partial$ is well defined. More precisely we have the following.
%\AA{Z labels in the figures don't fully match item 2} \EG{Changed the indices for $Z$ in both figures, please check}
\begin{proposition}
The differential $\partial$ preserves the relations \eqref{eq: Z relations},\eqref{eq: B relations}, \eqref{eq: D relations} between the generators.
\end{proposition}

\begin{lemma}
We have $\partial^2=0$, so that $(\CFL^{\alg}(T(n,n)),\partial)$ is a chain complex.
\end{lemma}

\begin{proof}
The proof consists of a sequence of commuting squares and the relation \eqref{eq: B relations}. For the bottom diagonal block (1) we have a commuting square
$$
\begin{tikzcd}
Z^j(J\setminus x) \arrow{d}{U_yV_y}& 
Z^j(J) \arrow{l}{U_xV_x} \arrow{d}{U_yV_y} \\
Z^j((J\setminus x)\setminus y) & Z^j(J\setminus y)\arrow{l}{U_xV_x} . 
\end{tikzcd}
$$
For the off-diagonal generators $B^j(I,J,M)$ we have several cases as above. In the generic case (2a) we get
$$
\begin{tikzcd}
B^j(I\setminus x,J\setminus x,M) \arrow[swap]{d}{V_y}& B^j(I,J,M) \arrow{l}{U_x} \arrow{d}{V_y}\\
B^j((I\setminus x)\cup y,J\setminus x,M\setminus y) & B^j(I\cup y,J,M\setminus y) \arrow{l}{U_x}.
\end{tikzcd}
$$
%\textcolor{red}{on the top left, should it be $M$ and on the bottom left, should it be $M\setminus y$?}
For the left edge (2c) we get the square
$$
\begin{tikzcd}
Z^j(M) \arrow[swap]{d}{U_yV_y}& B^j(I,-,M) \arrow{l}{U_I} \arrow{d}{V_y}\\
Z^j(M\setminus y) & B^j(I\cup y,-,M\setminus y) \arrow{l}{U_IU_y}.
\end{tikzcd}
$$
For the right edge (2d) for $c\le j+1$ we get the square 
$$
\begin{tikzcd}
B^j(I,J,-) \arrow{r}{V_{\overline{I}}} \arrow{d}{U_x}& Z^{j+1}(J) \arrow{d}{U_xV_x}\\
B^j(I\setminus x,J\setminus x,-) \arrow{r}{V_{\overline{I}}V_x} & Z^{j+1}(J\setminus x)
\end{tikzcd}
$$
while for $c>j+1$ we get the squares
$$
\begin{tikzcd}
B^j(I,J,-) \arrow{rr}{V_{\overline{I}}V_{I_1}} \arrow{d}{U_x}& &B^{j+1}(I_1,I_1,J\setminus I_1) \arrow{d}{U_x}\\
B^j(I\setminus x,J\setminus x,-) \arrow{rr}{V_{\overline{I}}V_xV_{I_1\setminus x}} & & B^{j+1}(I_1\setminus x,I_1\setminus x,J\setminus I_1)
\end{tikzcd}
$$
Note that $V_{\overline{I}}V_xV_{I_1\setminus x}=V_{\overline{I}}V_{I_1}$. For $c=j+2$ we get a special square
$$
\begin{tikzcd}
B^j(I,J,-) \arrow{rr}{V_{\overline{I}}V_{x}} \arrow{d}{U_x}& &B^{j+1}(x,x,J\setminus x) \arrow{d}{U_x}\\
B^j(I\setminus x,J\setminus x,-) \arrow{rr}{V_{\overline{I}}V_x} & & Z^{j+1}(J\setminus x)
\end{tikzcd}
$$
For the top diagonal block (3) we have two types of squares: for $x\notin I_1$ we have
$$
\begin{tikzcd}
B^{j-1}(I_1\cup \overline{J},I_1,J\setminus I_1) \arrow{d}{V_x} & & D^j(J) \arrow{d}{U_xV_x} \arrow[swap]{ll}{U_{J\setminus I_1}}\\
B^{j-1}(I_1\cup \overline{J}\cup x,I_1,(J\setminus I_1)\setminus x)  & & D^j(J\setminus x) \arrow[swap]{ll}{U_{(J\setminus I_1)\setminus x}}
\end{tikzcd}
$$
while for $x\in I_2$ we have
$$
\begin{tikzcd}
 D^j(J) \arrow{d}{U_xV_x} \arrow{rr}{V_{I_2}}  & & B^j(I_2,I_2,J\setminus I_2) \arrow{d}{U_x}\\
  D^j(J\setminus x)  \arrow{rr}{V_{I_2\setminus x}} & &  B^j(I_2\setminus x,I_2\setminus x,J\setminus I_2)
\end{tikzcd}
$$
In a special case $i=j-2$ we have $|J|=n-j+2$ and $|I_1|=n-j+1$, so $I_1=J\setminus x$, and we get a square
$$
\begin{tikzcd}
B^{j-1}(\overline{\{x\}},J\setminus x,\{x\}) \arrow{d}{V_xV_{I_2\setminus x}} & & D^j(J)  \arrow[swap]{ll}{U_{x}} \arrow{d}{V_{I_2}}\\
B^{j}(I_2\setminus x,I_2\setminus x,J\setminus I_2)  & & B^j(I_2,I_2,J\setminus I_2) \arrow{ll}{U_x}
\end{tikzcd}
$$
Note that when $n=2j-1$ and $i=j-2$, then $|I_2|=1, I_2=\{x\}$ and we get %then $B^j(\emptyset, \emptyset, J\setminus x)=Z^j(J\setminus x)$. 

$$
\begin{tikzcd}
B^{j-1}(\overline{\{x\}},J\setminus x,\{x\}) \arrow{d}{V_x} & & D^j(J)  \arrow[swap]{ll}{U_{x}} \arrow{d}{V_{x}}\\
Z^j(J\setminus x)  & & B^j(I_2,I_2,J\setminus I_2) \arrow{ll}{U_x}
\end{tikzcd}
$$
\end{proof}

\subsection{The algebraic conjecture}

\begin{lemma}
We have 
$$
H^0\left(\CFL^{\alg}\left(T(n,n)\right),\partial\right)=\HFL(T(n,n)).
$$
Here the index $0$ refers to the level grading.
\end{lemma}

\begin{proof}
Let us describe all generators of the complex $\CFL^{\alg}(T(n,n))$ with levels $0$ and $1$. 

The elements $X_j$ have level $0$. The elements $Z^j(J)$ have level 1 if $|J|=2$, and all other diagonal generators $Z^j(J),D^j(J)$ are positioned at higher levels. 

For the off-diagonal generators, recall that
$$
\lvl(B^j(I,J,M))=j-s+2c-1.
$$
Since $c+j>s$, we have $j-s+2c-1>c-1\ge 0$. Therefore there are no off-diagonal elements at level $0$, and the only element at level 1 corresponds to $c=1$ and $s=j+c-1=j$, this is the bottom vertex case (2b) above.

Therefore $H^0\left(\CFL^{\alg}(T(n,n)),\partial\right)$ is generated by $X_j$ modulo relations given by the images of level 1 generators. By \eqref{eq: diff B bottom} we have $U_IX_j=V_{\overline{I}}X_{j+1}$ for $|I|=j$.

If $2\le j\le n-1$ and $1\le a,b\le n$ are distinct,  the generator $Z^{j}(\{a,b\})$ is well defined and
$$
\partial Z^{j}(\{a,b\})=(U_aV_a+U_bV_b)X_j,
$$
so $(U_aV_a+U_bV_b)X_j=0$ in the homology. 

We claim that the latter relation also holds for $j=1$ and $j=n$. Indeed,
$$
V_a(U_aX_1)=V_a(V_{\overline{a}}X_2)=V_1\cdots V_nX_2
$$
does not depend on $a$, the argument for $j=n$ is similar. 

Therefore all of the relations in $H^0\left(\CFL^{\alg}(T(n,n)),\partial\right)$ agree with those in Equation \eqref{eq: HFL full relations}, and 
$$
H^0\left(\CFL^{\alg}(T(n,n)),\partial\right)\simeq \HFL(T(n,n)).
$$
\end{proof}

\begin{conjecture}
\label{conj: algebraic complex}
The complex $\left(\CFL^{\alg}(T(n,n)),\partial\right)$ is acyclic in non-zero levels. 
\end{conjecture}

Assuming Conjecture \ref{conj: algebraic complex}, we conclude that $\CFL^{\alg}(T(n,n))$ is a minimal resolution of $\HFL(T(n,n))$, and hence it is isomorphic to the minimal complex for $\CFL(T(n,n))$.

It is plausible that one can use the complex $\CFL^{\alg}(T(n,n))$ to write an algebraic model for the homotopies $h_{a,b}$.
We will only need the following fact about them.

\begin{lemma}
\label{lem: h from off diagonal to top diagonal}
Assume Conjecture \ref{conj: algebraic complex}. Suppose that $x$ is a generator in an off diagonal block at $(j',s,c)$ and $\xi_{a,b}(x)$ has a component $y$ in the top diagonal block at $A(X_j)$. Then $\lvl(y)=\max(n-j+2,j+1)$.
\end{lemma}

\begin{proof}
Recall that $\partial(y)=\partial_1(y)+\partial_2(y)+\partial_3(y)$ with $\partial_2(y)$ in the top diagonal block and $\partial_1(y),\partial_3(y)\neq 0$. The equation  $[\partial,\xi_{a,b}]=U_aV_a-U_bV_b$ can only be satisfied if $x$ is the term in either $\partial_1(y)$ or $\partial_3(y)$ with the coefficient linear in $U_i$ (resp. linear in $V_i$). By %Lemma \ref{lem: diff to off diagonal} 
the formula for $\partial D^j(J)$
this is only possible in the following cases:
\begin{enumerate}
\item[(1)] $n-(i+j)=1 \Leftrightarrow \lvl(y)=n-i=j+1.$  
\item[(2)] $j-i-1=1\Leftrightarrow \lvl(y)=n-i=n-j+2$.
\end{enumerate}
Since $\lvl(y)\ge \max(n-j+2,j+1)$, we have $\lvl(y)=\max(n-j+2,j+1)$ in both cases.
\end{proof}

\section{An algebraic model for $\varphi_1$}
\label{sec: appendix B}

Next, we define the chain  map 
$$
\varphi_1^{\alg}:\CFL^{\alg}(O_1(n,0))\to \CFL^{\alg}(O_1(n+1,1))
$$
lifting the map on homology from Proposition \ref{prop: Phi1 full homology}. We assume that the strands $1,\ldots,n,n+1$ are oriented positively, and the $(n+2)$-nd strand is oriented negatively. By Proposition \ref{prop: cob commute with orientation reversal} this determines the map $\varphi_1^{\alg}$ for any other orientation of the first $n$ strands. 

We add the subscripts $n,n+2$ to distinguish the generators. The variables $U_{n+2}$ and $U_{n+1}$ are identified on the right. Throughout we assume $1\le j\le \frac{n+1}{2}$, the case $j\ge  \frac{n+1}{2}$ is similar.   

\noindent {\bf 0.} For the zero level we have 
\begin{equation}
\label{eq: def phi1 X}
\varphi_1^{\alg}(X_{j,n})=U_{n+1}X_{j+1,n+2}.
\end{equation}

\noindent {\bf 1.} For the bottom diagonal block we have 
\begin{equation}
\label{eq: def phi1 Z}
\varphi_1^{\alg}(Z^j_{n}(J))=U_{n+1}Z^{j+1}_{n+2}(J).
\end{equation}

\noindent {\bf 2.} For most of the off-diagonal block we have 
\begin{equation}
\label{eq: def phi1 B generic}
\varphi_1^{\alg}(B^j_n(I,J,M))=B^{j+1}_{n+2}(I\cup n+1,J,M)
\end{equation}
For the right edge and $c\ge j+1$ we get % need to be more careful. Specifically, for $c=j+2$ we get 
%$$
%\varphi_1^{\alg}(B^j(I,J,-))=B^{j+1}(I\cup n+1,J,-)+ V_{\overline{I}}B^{j+2}(n+1,n+1,J).
%$$
%while for $c>j+2$ we get
\begin{equation}
\label{eq: def phi1 B right edge}
\varphi_1^{\alg}(B^j_n(I,J,-))=B^{j+1}_{n+2}(I\cup n+1,J,-)+ \sum_{\stackrel{I_1\subset J}{|I_1|=s-2j-1}}V_{\overline{I}}V_{I_1}B^{j+2}_{n+2}(I_1\cup n+1,I_1\cup n+1,J\setminus I_1).
\end{equation}

%\EG{Except for right edge where we have some extra terms} 

\noindent {\bf 3.} For the top diagonal block and $|J|>n-j+2$ we  have 
\begin{equation}
\label{eq: def phi1 D}
\varphi_1^{\alg}(D^j_n(J))=\varphi_1^{\diag}(D^j_n(J))+\varphi_1^L(D^j_n(J))+\varphi_1^R(D^j_n(J)),
\end{equation}
where 
$$
\varphi_1^{\diag}(D^j_n(J))=U_{n+1}D^{j+1}_{n+2}(J),
$$
$$
\varphi_1^L(D^j_n(J))=\sum_{\stackrel{I_1\subset J}{|I_1|=n-j+2}}U_{J\setminus I_1}B^{j}_{n+2}(I_1\cup \overline{J}\cup n+1,I_1,(J\setminus I_1)\cup n+2),
$$
$$
\varphi_1^R(D^j_{n}(J))=\sum_{\stackrel{I_2\subset J}{|I_2|=n-i-j-1}}V_{I_2}B^{j+1}_{n+2}(I_2\cup n+1,I_2\cup n+1,J\setminus I_2).
$$
%\EG{check last term, and different formula for the lowest level}
%\EG{specify the size of $I_1$ and $I_2$}
In the boundary case $|J|=n-j+2$ we have instead
\begin{equation}
\label{eq: def phi1 D exception}
\varphi_1^{\alg}(D^j_n(J))=B^j_{n+2}([n+1],J,-)+\sum_{\stackrel{I_2\subset J}{|I_2|=n-2j+1}}V_{I_2}B^{j+1}_{n+2}(I_2\cup n+1,I_2\cup n+1,J\setminus I_2).
\end{equation}
Here $[n+1]=\{1,\ldots,n+1\}$.
Again, one can check that the map $\varphi_1^{\alg}$ is compatible with the relations \eqref{eq: Z relations},\eqref{eq: B relations}, \eqref{eq: D relations} and therefore well-defined.

\begin{lemma}
The map $\varphi_1^{\alg}$ defined by \eqref{eq: def phi1 X}, \eqref{eq: def phi1 Z}, \eqref{eq: def phi1 B generic}, \eqref{eq: def phi1 B right edge}, \eqref{eq: def phi1 D} is a chain map assuming $U_{n+1}=U_{n+2}$ in $\CFL^{\alg}(O_1(n+1,1))$.
\end{lemma}

\begin{proof}
The equation $\partial\varphi_1^{\alg}(Z^j_n(J))=\varphi_1^{\alg}(\partial(Z^j_n(J)))$  is clear.
For the off-diagonal block, the {\bf generic case (2a)} this is clear. For the {\bf bottom vertex (2b)} we get
$$
\partial B^{j+1}_{n+2}(I\cup n+1,-,-)=U_{I}U_{n+1}X_{j+1,n+2}+V_{\overline{I}}U_{n+2}X_{j+2,n+2}=\varphi_1^{\alg}(\partial B^j_n(I,-,-)).
$$
For the {\bf left edge (2c)} we get
$$
\partial B^{j+1}_{n+2}(I\cup n+1,-,M)=\sum_{y\in M}V_yB^{j+1}_{n+2}(I\cup n+1\cup y,-,M\setminus y)+U_IU_{n+1}Z^{j+1}_{n+2}(M)=\varphi_{1}^{\alg}(\partial B^j_n(I,-,M)).
$$
The case of the {\bf right edge (2d)} and $c<j+2$ is similar. For the {\bf right edge (2d)} and $c=j+2$ we get
$$
\varphi_1^{\alg}(\partial B^j_n(I,J,-))=\sum_{x\in J}U_x\varphi_1^{\alg}(B^j_n(I\setminus x,J\setminus x,-))+\sum_{x\in J}V_{\overline{I}}V_x\varphi_1^{\alg}(B^{j+1}_n(x,x,J\setminus x))=
$$
$$
\sum_{x\in J}U_xB^{j+1}_{n+2}((I\setminus x)\cup n+1,J\setminus x,-)+\sum_{x\in J}V_{\overline{I}}V_xB^{j+2}_{n+2}(x\cup n+1,x,J\setminus x)
$$
while 
$$\partial B^{j+1}_{n+2}(I\cup n+1,J,-)=\sum_{x\in J}U_xB^{j+1}_{n+2}((I\setminus x)\cup n+1,J\setminus x,-)+V_{\overline{I}}U_{n+2}Z^{j+2}_{n+2}(J)
$$
and
$$
\partial B^{j+2}_{n+2}(n+1,-,J)=U_{n+1}Z^{j+2}_{n+2}(J)+\sum_{x\in J}V_xB^{j+2}_{n+2}(x\cup n+1,-,J\setminus x),
$$
so 
$$
\varphi_1^{\alg}(\partial B^j_n(I,J,-))=\partial B^{j+1}_{n+2}(I\cup n+1,J,-)+V_{\overline{I}}\partial B^{j+2}_{n+2}(n+1,-,J)=\partial \varphi_1^{\alg}(B^j_n(I,J,-)).
$$
For the {\bf right edge (2d)} and $c>j+2$ we similarly get
$$
\varphi_1^{\alg}(\partial B^j_n(I,J,-))=\sum_{x\in J}U_x\varphi_1^{\alg}(B^j_n(I\setminus x,J\setminus x,-))+\sum_{\stackrel{I_1\subset J}{|I_1|=s-2j}}V_{\overline{I}}V_{I_1}\varphi_1^{\alg}(B^{j+1}_n(I_1,I_1,J\setminus I_1))=
$$
\begin{multline}
\label{eq: phi1 right edge 1}
\sum_{x\in J}U_xB^{j+1}_{n+2}((I\setminus x)\cup n+1,J\setminus x,-))+\sum_{\stackrel{x\in J,I_2\subset J\setminus x}{|I_2|=s-2j-2}}U_xV_{\overline{I}}V_xV_{I_2}B^{j+2}_{n+2}(I_2\cup n+1,I_2\cup n+1,(J\setminus x)\setminus I_2)+\\
\sum_{\stackrel{I_1\subset J}{|I_1|=s-2j}}V_{\overline{I}}V_{I_1}B^{j+2}_{n+2}(I_1\cup n+1,I_1,J\setminus I_1)
\end{multline}
while 
\begin{equation}
\label{eq: phi1 right edge 2}
\partial B^{j+1}_{n+2}(I\cup n+1,J,-)=\sum_{x\in J}U_xB^{j+1}_{n+2}((I\setminus x)\cup n+1,J\setminus x,-)+\sum_{\stackrel{I_3\subset J}{|I_3|=s-2j-1}}V_{\overline{I}}U_{n+2}V_{I_3}B^{j+2}_{n+2}(I_3,I_3,J\setminus I_3),
\end{equation}
\begin{multline}
\label{eq: phi1 right edge 3}
\partial B^{j+2}_{n+2}(I_3\cup n+1,I_3\cup n+1,J\setminus I_3)=U_{n+1}B^{j+2}_{n+2}(I_3,I_3,J\setminus I_3)+\sum_{x\in I_3}U_xB^{j+2}_{n+2}((I_3\setminus x)\cup n+1,(I_3\setminus x)\cup n+1,J\setminus I_3)+\\
\sum_{y\in J\setminus I_3}V_yB^{j+2}_{n+2}(I_3\cup y\cup  n+1,I_3\cup n+1,(J\setminus I_3)\setminus y).
\end{multline}
We can identify the terms in \eqref{eq: phi1 right edge 1},\eqref{eq: phi1 right edge 2},\eqref{eq: phi1 right edge 3} by setting $I_2=I_3\setminus x$ and $I_1=I_3\cup y$ respectively, and using \eqref{eq: B relations} which implies
$$
B^{j+2}_{n+2}(I_1\cup n+1,I_1,J\setminus I_1)+\sum_{y\in I_1}B^{j+2}_{n+2}(I_1\cup n+1,(I_1\setminus y)\cup n+1,J\setminus I_1)=0.
$$
Therefore
$$
\varphi_1^{\alg}(\partial B^j_n(I,J,-))=\partial\left( B^{j+1}_{n+2}(I\cup n+1,J,-)+\sum_{\stackrel{I_3\subset J}{|I_3|=s-2j-1}}V_{\overline{I}}V_{I_3}B^{j+2}_{n+2}(I_3\cup n+1,I_3\cup n+1,J\setminus I_3)\right)=
$$
$$
\partial\varphi_1^{\alg}(B^j_n(I,J,-)).
$$

Finally, we consider the {\bf top diagonal block}.
We claim that the following equations hold:
\begin{equation}
\label{eq: phi1 chain 1}
\varphi_1^{\diag}(\partial_2(D^j_n(J)))=\partial_2(\varphi_1^{\diag}(D^j_n(J))),
\end{equation}
\begin{equation}
\label{eq: phi1 chain 2}
\varphi_1^{\alg}(\partial_1(D^j_n(J))+\varphi_1^L(\partial_2(D^j_n(J))=\partial_1\varphi_1^{\diag}(D^{j}_n(J))+\partial \varphi_1^L(D^j_n(J)),
\end{equation}
\begin{equation}
\label{eq: phi1 chain 3}
\varphi_1^{\alg}(\partial_3(D^j_n(J))+\varphi_1^R(\partial_2(D^j_n(J))=\partial_3(\varphi_1^{\diag}(D^j_n(J)))+\partial \varphi_1^R(D^j_n(J)).
\end{equation}

The equation \eqref{eq: phi1 chain 1} is clear. Let us check \eqref{eq: phi1 chain 2}. We have 
$$
\partial \varphi_1^L(D^j_n(J))=\sum_{I_1\subset J}U_{J\setminus I_1}\partial B^{j}_{n+2}(I_1\cup \overline{J}\cup n+1,I_1,(J\setminus I_1)\cup n+2)=
$$
\begin{align}
\label{eq: phi_1 d top block 1}
\sum_{I_1\subset J}U_{J\setminus I_1}\sum_{x\in I_1}U_xB^{j}_{n+2}((I_1\setminus x)\cup \overline{J}\cup n+1,I_1\setminus x,(J\setminus I_1)\cup n+2)+\\
\label{eq: phi_1 d top block 2}
\sum_{I_1\subset J}U_{J\setminus I_1}\sum_{y\in J\setminus I_1}V_yB^{j}_{n+2}(I_1\cup y\cup\overline{J}\cup n+1,I_1,((J\setminus I_1)\setminus y)\cup n+2)+\\
\label{eq: phi_1 d top block 3}
\sum_{I_1\subset J}U_{J\setminus I_1}U_{n+2}B^{j}_{n+2}(I_1\cup \overline{J}\cup n+1\cup n+2,I_1,J\setminus I_1)
\end{align}
For \eqref{eq: phi_1 d top block 1}, denote $I'_1=I_1\setminus x$, then we get %\EG{Stopped here}
$$
\sum_{I'_1,x\in J\setminus I'_1} U_{J\setminus I'_1}B^j_{n+2}(I'_1\cup \overline{J}\cup n+1,I'_1,((J\setminus I'_1)\setminus x)\cup n+2)=
$$
$$
\sum_{I'_1} U_{J\setminus I'_1}B^j_{n+2}(I'_1\cup \overline{J}\cup n+1,I'_1,J\setminus I'_1)=\varphi_1^{\alg}(\partial_1(D^j_n(J))).
$$
Here we used the relation \eqref{eq: B relations}.
For \eqref{eq: phi_1 d top block 2}, denote $J'=J\setminus y$, then we get
$$
\sum_y\sum_{I_1\subset J'}U_{J'\setminus I_1}U_yV_yB^{j}_{n+2}(I_1\cup\overline{J'}\cup n+1,I_1,(J'\setminus I_1)\cup n+2)=\varphi_1^L(\partial_2(D^j_n(J))).
$$
For \eqref{eq: phi_1 d top block 3}, we get exactly $$U_{n+1}\partial_1D^{j+1}_{n+2}(J)=\partial_1(\varphi_1^{\diag}(D^j_n(J))).$$
The proof of \eqref{eq: phi1 chain 3} is similar.

The special case $|J|=n-j+2$ is similar to the case of the right edge and we leave it to the reader.
\end{proof}

\begin{remark}
We can visualize the cancellations in \eqref{eq: phi_1 d top block 1}-\eqref{eq: phi_1 d top block 3} in the following diagram:
$$
\begin{tikzcd} 
 & D^j_n(J) \arrow[swap]{dl}{\partial_1} \arrow{d}{\partial_2} \arrow{dr}{\partial_3} 
 \arrow{rrrrr}{\varphi_1^L} \arrow[bend left,swap]{rrrrrrr}{\varphi_1^{\diag}} \arrow[bend left]{rrrrrrrrr}{\varphi_1^{R}} 
 & & & & & \bullet \arrow{dl} \arrow{d} \arrow{dr}{U_{n+1}}& & U_{n+1}D^{j+1}_{n+2}(J) \arrow[swap]{dl}{\partial_1} \arrow{d}{\partial_2} \arrow{dr}{\partial_3} & & \bullet  \arrow{dl} \arrow{d} \arrow{dr} &\\
\bullet \arrow[bend right,swap,pos=0.2]{rrrrr}{\varphi_1^{\alg}} & \bullet \arrow[bend right,swap,pos=0.8]{rrrrr}{\varphi_1^L} 
\arrow[bend right,swap,pos=0.8]{rrrrrrr}{\varphi_1^{\diag}}
\arrow[bend right,pos=0.8]{rrrrrrrrr}{\varphi_1^{R}}
& \bullet \arrow[bend right,swap,pos=0.8]{rrrrrrrrr}{\varphi_1^{\alg}} & & & \bullet & \bullet & \bullet  & \bullet & \bullet & \bullet & \bullet \\
\end{tikzcd}
$$
\end{remark}

\begin{corollary}
\label{cor: Phi1 appendix}
Assuming Conjecture \ref{conj: algebraic complex}, the map $\varphi_1^{\alg}$ defined by \eqref{eq: def phi1 X}-\eqref{eq: def phi1 D exception} is homotopic to the level preserving component of the  chain map $\varphi_1$ induced by the cobordism map $f_{P_1}$.  In particular, for $1\le j\le \frac{n+1}{2}$ and the top diagonal block at level $n-i,i<j-2$ we get
$$
\varphi_1(D^j_n(J))=U_{n+1}D^{j+1}_{n+2}(J).
$$
modulo off-diagonal and bottom diagonal blocks.
\end{corollary}

\begin{proof}
By \eqref{eq: def phi1 X} and Proposition \ref{prop: Phi1 full homology}, the map $\varphi_1^{\alg}$ yields a correct map in full homology. One can check that it preserves the level grading.
By Lemma \ref{lem: lifting to resolutions} and Corollary \ref{cor: homotopic to level filtered} the map $\varphi_1$ is level-filtered and its level-preserving component is homotopic to $\varphi_1^{\alg}$. 

Similarly to Corollary \ref{cor: Maslov separation}, one can check using Maslov grading that the level-lowering components of $\varphi_1(D^j_n(J))$  cannot belong to the top diagonal block, and hence are contained in off-diagonal and bottom diagonal blocks. 
\end{proof}

\begin{theorem}
\label{thm: Fp top diagonal}
Assuming Conjecture \ref{conj: algebraic complex}, for $1\le j\le \frac{n+1}{2}$ and the top diagonal block at level $n-i,i<j-2$ the map $F_P$ is given by 
\begin{equation}
\label{eq: Fp top block alg model 1}
F_P(D^j_n(J)\otimes T)=D_{n+2}^{j+1}(J\cup n+1)+D_{n+2}^{j+1}(J\cup n+2),
\end{equation}
\begin{equation}
\label{eq: Fp top block alg model 2}
F_P(D^j_n(J)\otimes B)=D_{n+2}^{j+1}(J\cup n+1\cup n+2) \mod \Imm F_P(-\otimes T).
\end{equation}
modulo off-diagonal and bottom diagonal terms.
\end{theorem}

\begin{proof}
Recall that by Lemma \ref{lem:h-PP-1} we have 
$$
U_{n+1}f_P(-\otimes T)\simeq h\circ \varphi_1
$$
where $[\partial, h]=U_{n+1}V_{n+1}-U_{n+2}V_{n+2}$.
By Lemma \ref{lem: h from off diagonal to top diagonal}    the homotopy $h$ cannot jump from off-diagonal to diagonal terms, and on the top diagonal block $h$ is given by 
$$
h(D^j_{n+2}(J))=D^j_{n+2}(J\cup n+1)+D^j_{n+2}(J\cup n+2),
$$
assuming that $J\subset \{1,\ldots,n\}.$
Therefore
$$
U_{n+2}F_P(D^j_n(J)\otimes T)=U_{n+1}D_{n+2}^{j+1}(J\cup n+1)+U_{n+1}D_{n+2}^{j+1}(J\cup n+2)+\mathrm{off-diagonal\ terms},
$$
and \eqref{eq: Fp top block alg model 1} follows. The proof of the second equation \eqref{eq: Fp top block alg model 2} is similar.
\end{proof}

\begin{remark}
\label{rem: FP well defined}
Since $D_n^j(J)$ satisfy relations \eqref{eq: D relations}, one might question whether \eqref{eq: Fp top block alg model 1} and \eqref{eq: Fp top block alg model 2} are well defined.

If $J'$ is the set with $|J'|=|J|+1$ then 
$$
\sum_{x\in J'}D^{j+1}_{n+2}((J'\setminus x)\cup n+1)+D^{j+1}_{n+2}(J')=0,\quad \sum_{x\in J'}D^{j+1}_{n+2}((J'\setminus x)\cup n+2)+D^{j+1}_{n+2}(J')=0.
$$
We can add these equations and get
$$
\sum_{x\in J'}D^{j+1}_{n+2}((J'\setminus x)\cup n+1)+\sum_{x\in J'}D^{j+1}_{n+2}((J'\setminus x)\cup n+2)=0.
$$
Therefore \eqref{eq: Fp top block alg model 1} respects the equations \eqref{eq: D relations}.
Similarly, 
$$
\sum_{x\in J'}D^{j+1}_{n+2}((J'\setminus x)\cup n+1\cup n+2)+D^{j+1}_{n+2}(J'\cup n+1)+D^{j+1}_{n+2}(J'\cup n+2)=0,
$$
so 
$$
\sum_{x\in J'}D^{j+1}_{n+2}((J'\setminus x)\cup n+1\cup n+2)=F_P(D^j_n(J')\otimes T).
$$
Therefore \eqref{eq: Fp top block alg model 2} only respects the equations \eqref{eq: D relations} modulo the image of $F_P(-\otimes T)$. 
\end{remark}

\end{appendix}

\bibliographystyle{amsplain}
\bibliography{references}
 
\begin{comment}

\end{comment}
\end{document}